\documentclass[a4paper]{amsart}
\usepackage{ifthen}
\usepackage{amsmath}
\usepackage{amssymb}
\usepackage{amsfonts}
\usepackage{amsthm}
\usepackage{amscd}
\usepackage{enumerate}
\usepackage[shortlabels]{enumitem} 
\usepackage{hyperref}
\usepackage{tikz-cd}

\newtheorem{bigtheorem}{Theorem}

\newtheorem{theorem}{Theorem}[section]
\newtheorem{lemma}[theorem]{Lemma}
\newtheorem{proposition}[theorem]{Proposition}
\newtheorem{corollary}[theorem]{Corollary}
\theoremstyle{remark}
\newtheorem{remark}[theorem]{Remark}
\newtheorem{question}[theorem]{Question}
\theoremstyle{definition}
\newtheorem{definition}[theorem]{Definition}
\newtheorem*{ack}{Acknowledgements}
\numberwithin{equation}{section}

\renewcommand{\O}{\mathcal{O}}
\newcommand{\X}{\mathcal{X}}
\newcommand{\A}{\mathcal{A}}
\newcommand{\B}{\mathcal{B}}

\newcommand{\F}{\mathbb{F}}
\newcommand{\Z}{\mathbb{Z}}
\newcommand{\Q}{\mathbb{Q}}

\renewcommand{\H}{\mathrm{H}}
\newcommand{\R}{\mathrm{R}}
\newcommand{\Kh}{K^h}
\newcommand{\mmu}{\boldsymbol{ \mu}}
\newcommand{\Gm}{\mathbb{G}_{\textrm{m}}}

\newcommand{\Aff}{\mathbb{A}}

\newcommand{\m}{\mathfrak{m}}

\newcommand{\et}{{\textnormal{\'et}}}

\newcommand{\p}{\mathfrak{p}}
\newcommand{\q}{\mathfrak{q}}
\newcommand{\V}{\mathcal{V}}
\renewcommand{\P}{\mathbb{P}}

\DeclareMathOperator{\gr}{gr}
\DeclareMathOperator{\K}{K}
\DeclareMathOperator{\inv}{inv}
\DeclareMathOperator{\Br}{Br}
\DeclareMathOperator{\fil}{fil}
\DeclareMathOperator{\cores}{cores}
\DeclareMathOperator{\res}{res}
\DeclareMathOperator{\Tr}{Tr}
\DeclareMathOperator{\Spec}{Spec}
\DeclareMathOperator{\rsw}{rsw}
\DeclareMathOperator{\sw}{sw}

\DeclareMathOperator{\Pic}{Pic}
\DeclareMathOperator{\Gal}{Gal}
\DeclareMathOperator{\cd}{cd}
\DeclareMathOperator{\Hom}{Hom}
\DeclareMathOperator{\Proj}{Proj}
\DeclareMathOperator{\Sym}{Sym}
\DeclareMathOperator{\Frac}{Frac}
\DeclareMathOperator{\Char}{char}
\DeclareMathOperator{\dlog}{dlog}
\DeclareMathOperator{\Frob}{Frob}
\DeclareMathOperator{\Supp}{Supp}
\DeclareMathOperator{\card}{card}

\DeclareMathOperator{\Fil}{Fil}
\DeclareMathOperator{\file}{Ev} 
\newcommand{\film}{\widetilde{\fil}} 
\newcommand{\tv}[3]{{[\overrightarrow{#2#3}]_{#1}}} 
\newcommand{\evmap}[1]{|#1|} 

\usepackage{color}

\title{Evaluating the wild Brauer group}
\author{Martin Bright}
\address{Mathematisch Instituut \\ Niels Bohrweg 1 \\ 2333 CA Leiden \\ Netherlands}
\email{m.j.bright@math.leidenuniv.nl}
\author{Rachel Newton}
\address{Department of Mathematics\\
King's College London\\
Strand\\ London WC2R 2LS\\
UK}
\email{rachel.newton@kcl.ac.uk}

\dedicatory{In memory of Sir Peter Swinnerton-Dyer}

\subjclass[2020]{Primary 14F22; Secondary 14G12, 14G20, 14F30}

\begin{document}

\begin{abstract}Classifying elements of the Brauer group of a variety $X$ over a $p$-adic field by the $p$-adic accuracy needed to evaluate them gives a filtration on $\Br X$.
We relate this filtration to that defined by Kato's Swan conductor. The refined Swan conductor controls how the evaluation maps vary on $p$-adic discs:
this provides a geometric characterisation of the refined Swan conductor. We give applications to rational points on varieties over number fields, including failure of weak approximation for varieties admitting a non-zero global $2$-form.
\end{abstract}

\maketitle

\section{Introduction}
Let $k$ be a $p$-adic field with ring of integers $\O_k$, uniformiser $\pi$ and residue field $\F$, and let $X/k$ be a smooth geometrically irreducible variety. 
The most na\"ive filtration on the Brauer group $\Br X$, and the one we aim to understand, is that given by evaluation of elements of $\Br X$ at the $k$-points of $X$.
If $\A\in \Br X$ has order coprime to $p$, then~\cite[\S5]{bad} shows that the evaluation map for $\A$ factors through reduction to the special fibre. 
In this article, we describe the variation of the evaluation map in the considerably more complicated case of elements of order a power of $p$ in $\Br X$.

To define the evaluation filtration, fix a smooth model $\X/\O_k$ having geometrically integral special fibre $Y/\F$ with function field $F$.
Given $\A \in \Br X$, one can ask whether the evaluation map $\evmap{\A} :\X(\O_k) \to \Br k$ factors through the reduction map $\X(\O_k) \to \X(\O_k / \pi^i)$ for any $i\geq 1$.
We first define some notation. 

Let $k'$ be a finite extension of $k$ of ramification index $e(k'/k)$, with ring of integers $\O_{k'}$ and uniformiser $\pi'$. For $r\geq 1$ and $P \in \X(\O_{k'})$, let $B(P,r)$ be the set of points $Q\in \X(\O_{k'})$ such that $Q$ has the same image as $P$ in $\X(\O_{k'}/(\pi')^r)$. Define 
\begin{align*}
\file_n \Br X &= \{\A\in \Br X\mid \forall\; k'/k \textrm{ finite, } \forall\; P\in \X(\O_{k'}), & \\ 
& \hspace{5em} \evmap{\A} \textrm{ is constant on } B(P,e(k'/k)(n+1))\} & (n\geq 0), \\
\file_{-1} \Br X &= \{\A\in \Br X\mid \forall\; k'/k \textrm{ finite, } \evmap{\A} \textrm{ is constant on }  \X(\O_{k'})\}, & \\
\file_{-2} \Br X &= \{\A\in \Br X\mid  \forall\; k'/k \textrm{ finite, }  \evmap{\A} \textrm{ is zero on }  \X(\O_{k'})\}.& 
\end{align*}
Let $\{\fil_n\Br X\}_{n\geq 0}$ denote the filtration given by Kato's Swan conductor (see Definition~\ref{def:Swan}) and, for $n \ge 1$, write $\rsw_{n}(\A)$ for the refined Swan conductor of $\A\in \fil_n\Br X$ (see Definition~\ref{def:rsw}). We have $\rsw_{n}(\A)=[\alpha, \beta]_{\pi,n}$ for some $(\alpha,\beta)\in\Omega^2_F\oplus\Omega^1_F$.
In Section~\ref{sec:tame} we will define a residue map $\partial \colon \fil_0 \Br X \to \H^1(Y,\Q/\Z)$.

\begin{bigtheorem}\label{thm:main}
Let $k$ be a finite extension of $\Q_p$.
Let $X$ be a smooth, geometrically irreducible variety over $k$, and let $\X \to\Spec \O_k$ be a smooth model of $X$. 
Suppose that the special fibre $Y$ of $\X$ is geometrically irreducible. Then
\begin{enumerate}
\item\label{it:tame-2}  $\file_{-2} \Br X=\{\A\in\fil_0\Br X\mid \partial\A=0 \}$;
\item\label{it:tame-1}$\file_{-1} \Br X=\{\A\in\fil_0\Br X\mid \partial\A\in \H^1(\F,\Q/\Z) \}$;
\item\label{it:tame0} $\file_{0} \Br X=\fil_0\Br X$;
\item\label{it:wild} for $n \ge 1$, $\file_n \Br X = \{\A\in \fil_{n+1}\Br X\mid \rsw_{n+1}(\A)\in[\Omega^2_F, 0]_{\pi,n+1}\}.$
\end{enumerate}
\end{bigtheorem}

\begin{remark}
\begin{enumerate}
\item By definition of the refined Swan conductor, \[\fil_{n}\Br X\subset  \{\A\in \fil_{n+1}\Br X\mid \rsw_{n+1}(\A)\in[\Omega^2_F, 0]_{\pi,n+1}\},\]
with equality if $p\nmid n+1$. See Lemma~\ref{lem:dbna} for more details.
\item In the case of $\H^1(K) = \H^1(K,\Q/\Z)$, where $K$ is the function field of $X$, Kato's filtration and the refined Swan conductor have been extensively studied in the literature, and are closely related to ramification theory (see Section~\ref{sec:comparisons}).
We believe that Theorem~\ref{thm:main} is the first geometric characterisation of Kato's filtration on $\H^2(K) = \Br K$, with Theorem~\ref{thm:onelemma_intro} below giving a geometric description of the refined Swan conductor.
The modified version of Kato's filtration featuring in Theorem~\ref{thm:main} does not seem to have appeared elsewhere, though it is analogous to the ``non-logarithmic'' version of Kato's filtration on $\H^1(K$) defined by Matsuda~\cite{Matsuda} in equal characteristic.
\item The reason for considering points over finite extensions of $k$, instead of just over $k$ itself, is that the filtration obtained is better behaved.
(A function that is non-constant on points over some field extension can be constant on the rational points, simply because there are ``too few'' points of $Y(\F)$: see~\cite[Remark~5.20] {bad} for an example.)
\item A consequence of Theorem~\ref{thm:main} is that the evaluation filtration does not change if $Y$ is replaced by a non-empty open subset.
\item In fact, our proof of Theorem~\ref{thm:main} shows that it remains true if we modify the definition of $\file_n \Br X$ by restricting to unramified finite extensions $k'/k$ instead of all finite extensions, see Corollaries~\ref{cor:Evnr} and~\ref{cor:Evnr2}.
\item The model $\X$ is not assumed to be proper.
If $X$ has bad reduction, for example admitting a proper model whose special fibre has several components of multiplicity $1$, then the filtrations $\fil_n$ and $\file_n$ are defined separately for each component, by looking at the smooth model obtained by restricting to the smooth locus of that component.
The relationship between the filtrations corresponding to different components is in general complicated.  However, in Section~\ref{sec:blowups} we shall study the specific case in which one component arises from blowing up the model in a smooth point of another component, which will be a central ingredient of our proofs.
\item The inclusion $\fil_n \Br X[p] \subset \file_n \Br X[p]$ is implicit in work of Uematsu~\cite{U}, at least in the case when $k$ contains a primitive $p$th root of unity.
\item Yamazaki~\cite{Y} has proved a result very similar to Theorem~\ref{thm:main} in the case that $X$ is a smooth proper curve.
In that case, one can extend the Brauer--Manin pairing to the Picard group $\Pic X$.
Yamazaki defines a filtration on $\Pic X$ by considering the kernels of reduction modulo powers of $\pi$, and shows that the induced filtration on $\Br X$ coincides with Kato's filtration.
(When $X$ is a curve, the group $\Omega^2_F$ is trivial, so our filtration in Theorem~\ref{thm:main} also coincides with Kato's, by definition of the refined Swan conductor.)
\item Sato and Saito~\cite{SS} have shown that $\file_{-2} \Br X$ coincides with the image of $\Br\X$ in $\Br X$ when $\X$ is regular and proper over $\O_k$, but without the assumption of smoothness.
(They also assume that $\X$ satisfies purity for the Brauer group, but this is now known to hold for all regular schemes~\cite{Kestutis}.)
In Corollary~\ref{cor:brix}, we will show how our results give a new proof of this when $\X$ is smooth over $\O_k$.
\end{enumerate}
\end{remark}

In order to prove Theorem~\ref{thm:main}, we examine the behaviour of the evaluation maps on the graded pieces of Kato's filtration on the Brauer group. The results of this study for $n \ge 1$ are summarised in Theorem~\ref{thm:onelemma_intro} below; for a stronger and more detailed statement, see Theorem~\ref{thm:onelemma}. In order to state Theorem~\ref{thm:onelemma_intro}, we need to introduce some more notation. Let $P\in \X(\O_k)$ and let $P_0$ denote the image of $P$ in $Y(\F)$. Elements in the image of the reduction map $B(P,n)\to \X(\O_k/\pi^{n+1})$ can be identified with tangent vectors in $T_{P_0}Y$ (see Lemma~\ref{lem:tangent}). Write $\tv{n}{P}{Q}$ for the tangent vector corresponding to the image in $\X(\O_k/\pi^{n+1})$ of a point $Q\in B(P,n)$.
In the statement of the following theorem, we denote by $x/p$ the image of $x\in\F_p$ under the identification of $\Z/p$ with the $p$-torsion in $\Q/\Z$. The integer $e$ is the absolute ramification index of $k$ and we set $e'=ep/(p-1)$.

\begin{bigtheorem}\label{thm:onelemma_intro}Let $k$ be a finite extension of $\Q_p$.
Let $X$ be a smooth, geometrically irreducible variety over $k$, and let $\X \to\Spec \O_k$ be a smooth model of $X$. 
Suppose that the special fibre $Y$ of $\X$ is geometrically irreducible.
Let $n>0$, 
let $\A\in\fil_n \Br X$, and let $\rsw_{n}(\A) = [\alpha, \beta]_{\pi,n}$ for some $(\alpha,\beta)\in \Omega^2_F \oplus \Omega^1_F$.
Let $P\in \X(\O_k)$, and let $P_0 \in Y(\F)$ be the reduction of $P$.
Then $\alpha$ and $\beta$ are regular at $P_0$ and we have the following description of the evaluation map $\evmap{\A} \colon \X(\O_k) \to \Br k$.  
\begin{enumerate}
\item \label{onelemma_intro_1} For $Q \in B(P,n)$,
\[
\inv \A(Q) = \inv \A(P) + \frac{1}{p}\Tr_{\F/\F_p} \beta_{P_0}(\tv{n}{P}{Q}).
\]
In particular, if $\beta_{P_0}\neq 0$ then $\evmap{\A}$ takes $p$ distinct values on $B(P,n)$.

\item \label{surj2-intro} If $\beta=0$ and $\alpha_{P_0}\neq 0$ then there exists $Q \in B(P,1)$ such that $\evmap{\A}$ takes $p$ distinct values on $B(Q,n-1)$.

\end{enumerate}
\end{bigtheorem}

\begin{remark} 
\begin{enumerate}
\item Case~\eqref{surj2-intro} is only possible if $p\mid n$: see Lemma~\ref{lem:dbna}.
\item Complementing the formula in~\eqref{onelemma_intro_1}, Theorem~\ref{thm:onelemma}\eqref{rla} gives a description of the evaluation map in the case where $\beta=0$ and $n>2$. See also Lemma~\ref{lem:p=n=2} for the case where $\beta=0$ and $p=n=2$.
\item If $\A$ has order $p^t$ in the Brauer group then evaluating $\A$ at points in $\X(\O_k)$ gives values in $\Br k[p^t]\cong p^{-t}\Z/\Z$. Theorem~\ref{thm:onelemma}\eqref{surj3} gives sufficient conditions for all $p^t$ possible values to be obtained.
\end{enumerate}
\end{remark}

Elements in $\Br X$ of order coprime to $p$ have been thoroughly treated in the literature, in particular by Colliot-Th\'el\`ene--Saito~\cite{CTS}, Colliot-Th\'el\`ene--Skorobogatov~\cite{CTSk} and Bright~\cite{bad}.  The computation of the evaluation map in the coprime to $p$ case is greatly aided by the fact that the map factors through reduction to the special fibre. In a similar way, Theorem~\ref{thm:onelemma_intro} enables the computation of the evaluation map for Brauer group elements of order a power of $p$.   Thus it facilitates a systematic treatment of Brauer--Manin obstructions, which will have both theoretical and computational implications for the study of rational points on varieties. Some first consequences of Theorem~\ref{thm:onelemma_intro} are outlined below (see Theorems~\ref{thm:potential} and~\ref{thm:app}).

\subsection*{Applications to the Brauer--Manin obstruction}\label{sec:BM}
Manin~\cite{Manin} introduced the use of the Brauer group to study rational points on varieties over number fields.
Let $V$ be a smooth, proper, geometrically irreducible variety over a number field $L$.  The evaluation maps $\evmap{\B} \colon V(L_v) \to \Br L_v$ for each $\B \in \Br V$ and place $v$ of $L$ combine to give a pairing
\[
\Br V \times V(\mathbb{A}_L) \to \Q/\Z,
\]
where $\mathbb{A}_L$ denotes the ring of ad\`eles of $L$.  Manin observed that the diagonal image of $V(L)$ is contained in the right kernel of this pairing, denoted $V(\mathbb{A}_L)^{\Br}$.  If $V(\mathbb{A}_L)$ is non-empty but $V(\mathbb{A}_L)^{\Br}$ is empty, then there is a Brauer--Manin obstruction to the Hasse principle; if $V(\mathbb{A}_L)^{\Br}$ is not equal to the whole of $V(\mathbb{A}_L)$, then there is a Brauer--Manin obstruction to weak approximation.

The following question posed by Swinnerton-Dyer 
~\cite[Question~1]{CTSk} is of great relevance to the computation of $V(\mathbb{A}_L)^{\Br}$. 
\begin{question}[Swinnerton-Dyer]\label{q:S-D}
Let $L$ be a number field and let $S$ be a finite set of places of $L$ containing the Archimedean places. Let $\V$ be a smooth projective $\O_S$-scheme with geometrically integral fibres, and let $V/L$ be the generic fibre. Assume that $\Pic\bar{V}$ is finitely generated and torsion-free. 
Is there an open and closed set $Z \subset \prod_{v \in S} V(L_v)$ such that
\begin{equation}\label{eq:BMset}
V(\mathbb{A}_L)^{\Br} = Z \times \prod_{v \notin S} V(L_v)?
\end{equation}
\end{question}
Informally: does the Brauer--Manin obstruction involve only the places of bad reduction and the Archimedean places? 

In~\cite[Theorem~3.1]{CTSk} Colliot-Th\'{e}l\`{e}ne and Skorobogatov prove that if the transcendental Brauer group of $V$ (meaning the image of the natural map $\Br V\to \Br\bar{V}$) is finite and $S$ contains all the primes dividing its order then the answer to Question~\ref{q:S-D} is yes. 
Skorobogatov and Zarhin show in \cite{SZ} that the transcendental Brauer groups of Abelian varieties and K3 surfaces over number fields are finite; their question as to whether this is true for smooth projective varieties more generally remains open (see \cite[Question 1]{SZ}).

Before we address Question~\ref{q:S-D}, let us introduce one further question. 

\begin{question}[Wittenberg]\label{q:W}
 If a smooth projective variety $V$ over a number field $L$ satisfies the
   Hasse principle and weak approximation over all finite extensions, does it follow that $\H^2(V,\O_V)=0$?
\end{question}

Question~\ref{q:W} is one of a family of questions posed by Olivier Wittenberg in a private discussion with Jean-Louis Colliot-Th\'{e}l\`{e}ne in 2010, as follows: if a smooth projective variety $V$ over a number field $L$ satisfies the
   Hasse principle and weak approximation over all finite extensions, does
   it follow that $V$ is rationally connected? Does it at least follow that
   $\H^i(V,\O_V)=0$ for all $i>0$? Does it at least follow that $\H^2(V,\O_V)=0$? Note that the analogue of Question~\ref{q:W} with $\H^2(V,\O_V)$ replaced by $\H^1(V,\O_V)$ has a positive answer -- this follows from \cite[Corollary~2.4]{Harari}, for example. 
   
   Our next result owes its existence to Wittenberg's suggestion that our methods could be used to address Question~\ref{q:W}. It gives a positive answer to that question, and also shows that the answer to Question~\ref{q:S-D} is no, in general. In fact, it shows that if $\H^0(V,\Omega^2_V) \neq 0$ then every prime of good ordinary reduction is involved in a Brauer--Manin obstruction over an extension of the base field.

\begin{bigtheorem}\label{thm:potential}
Let $V$ be a smooth, proper variety over a number field $L$ with $\H^0(V,\Omega^2_V) \neq 0$.
Let $\p$ be a prime of $L$ at which $V$ has good ordinary reduction, with residue characteristic $p$.
Then there exist a finite extension $L'/L$, a prime $\p'$ of $L'$ lying over $\p$, and an element $\A \in \Br V_{L'}\{p\}$ such that the evaluation map $\evmap{\A} \colon V(L'_{\p'}) \to \Br L'_{\p'}$ is non-constant.
In particular, if $V(\mathbb{A}_{L'})\neq \emptyset$ then $\A$ obstructs weak approximation on $V_{L'}$.
\end{bigtheorem}

It has been conjectured that a smooth, projective variety over a number field $L$ has good ordinary reduction at a positive density set of primes of $L$: Joshi~\cite[Conjecture~3.1.1]{joshi} attributes this conjecture to Serre.
In the cases of Abelian surfaces and K3 surfaces, it is known to be true~\cite{BZ,joshi}.

The assumption that $\H^0(V,\Omega^2_V) \neq 0$ implies, via Hodge theory, that the second Betti number $b_2(\bar{V})$ and geometric Picard number $\rho(\bar{V})$ are not equal. Since $\Br\bar{V}$ contains a copy of $(\Q/\Z)^{b_2-\rho}$ (see~\cite[Proposition~5.2.9]{CTSbook}), this implies that there exists a finite extension $L'/L$ such that the transcendental Brauer group of $V_{L'}$ is non-trivial. On the other hand, if the transcendental Brauer group is trivial then~\cite[Theorem~3.1]{CTS} shows that the answer to Question~\ref{q:S-D} is yes.

Having seen in Theorem~\ref{thm:potential} that the places involved in the Brauer--Manin obstruction need not be confined to the places of bad reduction and the Archimedean places, one may be interested in the following question:
\begin{question}\label{q:S}
Let $V$ be a smooth, proper, geometrically irreducible variety over a number field $L$ such that $\Pic\bar{V}$ is finitely generated and torsion-free. Does there exist a finite set $S$ of places of $L$ and a closed set $Z\subset \prod_{v \in S} V(L_v)$ such that 
\begin{equation}\label{eq:BMset2}
V(\mathbb{A}_L)^{\Br} = Z \times \prod_{v \notin S} V(L_v)?
\end{equation}
If so, can one give an explicit description of such a set $S$?
\end{question}

The assumption in Questions~\ref{q:S-D} and~\ref{q:S} that $\Pic\bar{V}$ be finitely generated and torsion-free is necessary: for example, in the case of an elliptic curve $E$ over $\mathbb{Q}$ with finite Tate--Shafarevich group and trivial Mordell--Weil group, \cite[Proposition~6.2.4]{SBook} shows that 
\[E(\mathbb{A}_\mathbb{Q})^{\Br}=E(\mathbb{R})^0\times \prod_{p\textrm{ prime}}\{\mathcal{O}_E\},\]
 where $E(\mathbb{R})^0$ denotes the connected component of the identity in $E(\mathbb{R})$. This contradicts the description of \eqref{eq:BMset2}. More generally, note that non-trivial torsion in $\Pic\bar{V}$ implies that the abelianisation of $\pi_1^{\et}(\bar{V})$ is non-trivial. For a smooth, proper, geometrically integral variety $V$ over a number field $L$ satisfying $V(L)\neq\emptyset$ and $\pi_1^{\et}(\bar{V})\neq 0$, Harari \cite[\S2]{Harari} has shown that for any finite set $\Sigma$ of places of $L$, the variety $V$ does not satisfy weak approximation outside $\Sigma$. The proof uses a descent obstruction which, in the case of an abelian covering, is coarser than the Brauer--Manin obstruction. It follows that in this setting the Brauer--Manin set is not of the form described in \eqref{eq:BMset2}.

If one assumes that 
the transcendental Brauer group of $V$ is finite in Question~\ref{q:S} then it follows from the Hochschild--Serre spectral sequence that the quotient of $\Br V$ by the image of $\Br L$ is finite; therefore the existence of the finite set $S$ 
is a consequence of the Albert--Brauer--Hasse--Noether Theorem and the continuity of the Brauer--Manin pairing. The finiteness of the quotient of  $\Br V$ by the image of $\Br L$ also implies that the Brauer--Manin set is open as well as closed. On the other hand, without the finiteness assumption on the transcendental Brauer group, the existence of the finite set $S$ in Question~\ref{q:S} is not \emph{a priori} obvious.

Theorem~\ref{thm:app} below gives a positive answer to Question~\ref{q:S}, without any finiteness assumption on the transcendental Brauer group of $V$. One consequence of this theorem is that for a K3 surface over $\Q$ the only places that play a r\^{o}le in the Brauer--Manin obstruction are the Archimedean places, the primes of bad reduction, and the prime $2$ (see Remark~\ref{rmk:fewerplaces}).

\begin{bigtheorem}\label{thm:app} 
Let $L$ be a number field.  Let $V$ be a smooth, proper, geometrically irreducible variety over $L$.
Assume $\Pic\bar{V}$ is finitely generated and torsion-free. 
Then there is a finite set of places $S$ of $L$ such that, for all $\A \in \Br V$ and all places $\p \notin S$, the evaluation map $\evmap{\A} \colon V(L_\p) \to \Br L_\p$ is constant. Furthermore, the set $S$ can be taken to consist of the following places of $L$:
\begin{enumerate}
\item Archimedean places;
\item places of bad reduction for $V$;
\item\label{ram2} places $\p$ satisfying $e_\p \ge p-1$, where $e_{\p}$ is the absolute ramification index of $\p$ and $p$ is the residue characteristic of $\p$;  
\item\label{omega-1} places $\p$ for which, for any smooth proper model $\V \to\Spec \O_\p$ of $V$, the group $\H^0(\V(\p),\Omega^1_{\V(\p)})$ is non-zero, where $\V(\p)$ is the special fibre of $\V$.
\end{enumerate}
\end{bigtheorem}

\begin{remark}
\begin{enumerate}

\item James Newton observed that the set in~\eqref{omega-1} is contained in the set consisting of places that are ramified in $L/\Q$ and places lying above $2$. Indeed, let $\p$ be a place of $L$ with $e_{\p/p}=1$ and residue field $\F_\p$ of characteristic $p>2$. Since $\Pic\bar{V}$ is torsion-free, we have $\H^1_{\et}(\bar{V}, \Z/p\Z) = 0$. Hence, the integral comparison theorem of Fontaine--Messing (\cite{FM}, see also~\cite[Theorems~3.1.3.1 and 3.2.3]{BM}) shows that the de Rham cohomology group $\H^1_{\mathrm{dR}}(\V(\p)/\F_\p)$ is zero. Now by the degeneration of the Hodge--de Rham spectral sequence for $\V(\p)$ (Deligne--Illusie~\cite[Corollary~2.5]{DI}), this implies $\H^0(\V(\p),\Omega^1_{\V(\p)})=0$.
\item In some cases of interest, such as when $V$ is a K3 surface, one can show that the set in~\eqref{omega-1} is always empty -- see Remark~\ref{rmk:fewerplaces}.
\end{enumerate}
\end{remark}

\subsection*{Outline of the paper}
Section~\ref{sec:Kato} contains some technical results and background relating to Kato's refined Swan conductor.
In Section~\ref{sec:tame} we define a residue map $\partial \colon \fil_0 \Br X \to \H^1(Y,\Q/\Z)$ and use it to describe the evaluation maps for elements of $\fil_0\Br X$. The main body of the paper, Sections~\ref{sec:compare-rsw}--\ref{sec:endgames}, is concerned with the proof of Theorem~\ref{thm:onelemma_intro}. Its proof will involve a chain of blowups with an associated decreasing sequence of Swan conductors at the exceptional divisors. Eventually, we will obtain Swan conductor zero, whereupon evaluations are controlled by the residue map, as in Section~\ref{sec:tame}.
Section~\ref{sec:compare-rsw} contains some technical results that will be used in Section~\ref{sec:endgames} to relate the refined Swan conductor of a Brauer group element to that of its residue along the exceptional divisor. In Section~\ref{sec:blowups}, we show how information about the refined Swan conductor is retained under blowups. Section~\ref{sec:Pn} gathers some calculations pertaining to the exceptional divisors of these blowups. Section~\ref{sec:tangent} relates lifts of points to tangent vectors and shows how to keep track of them when blowing up. In Section~\ref{sec:endgames} we bring everything together to complete the proof of Theorem~\ref{thm:onelemma_intro}. 
In Section~\ref{sec:main} we prove Theorem~\ref{thm:main}.
In Section~\ref{sec:comparisons} we compare various other filtrations in the literature with our modified version of Kato's filtration which gives rise to the evaluation filtration on the Brauer group. Section~\ref{sec:BMO} is concerned with applications to the Brauer--Manin obstruction and contains the proofs of Theorems~\ref{thm:potential} and \ref{thm:app}.

\subsection*{Notation}
If $A$ is an Abelian group and $n$ a positive integer, then $A[n]$ and $A/n$ denote the kernel and cokernel, respectively, of multiplication by $n$ on $A$.
If $\ell$ is prime, then $A\{\ell\}$ denotes the $\ell$-power torsion subgroup of $A$.

We use extensively the notation introduced in~\cite[\S 1]{K}.  In particular, the notation $(\Z/n)(r)$ has a particular meaning in characteristic $p$.
Write $n = p^s m$ with $p \nmid m$.  For any scheme $S$ smooth over a field of characteristic $p$, the object $(\Z/n)(r)$ of the bounded derived category $D^b(S_\et)$ is defined by
\[
(\Z/n)(r) := \mmu_m^{\otimes r} \oplus W_s \Omega^r_{S,\log}[-r].
\]
Here $W_s \Omega^\bullet_{S,\log}$ is the logarithmic de Rham--Witt complex of Illusie: see~\cite[I.5.7]{I}.
We further use Kato's notation
\[
\H^q_n(R) := \H^q(R_\et, (\Z/n)(q-1)), \quad \H^q(R) := \varinjlim_{n} \H^q_n(R)
\]
whenever either $n$ is invertible in $R$, or $R$ is smooth over a field of characteristic $p$.

The definition of $(\Z/n)(r)$ ensures that, in any characteristic, we have an exact triangle
\begin{equation}\label{eq:kummer}
(\Z/n)(1) \to \Gm \xrightarrow{n} \Gm \to (\Z/n)(1)[1]
\end{equation}
where the part prime to the characteristic is the Kummer sequence, and the $p$-part in characteristic $p$ is~\cite[Proposition~I.3.23.2]{I}.
Given $a \in R^\times = \H^0(R,\Gm)$, we denote the image of $a$ in $\H^1(R,(\Z/n)(1))$ by $\{ a \}$.  The exact triangle~\eqref{eq:kummer} further shows that, for a field $K$ of any characteristic, we have $\H^2(K) = \Br K$.

\begin{ack}
We thank Evis Ieronymou for many valuable comments and suggestions on previous drafts of this article. We thank Daniel Loughran for several helpful comments and references that have improved the introduction. We thank Bhargav Bhatt, Bas Edixhoven, Christopher Lazda, Yu Min, James Newton, Alexei Skorobogatov and Olivier Wittenberg for useful conversations. We are very grateful to the anonymous referee(s) who suggested the geometric approach that we have now adopted, following Kato in~\cite{K}, and made numerous other suggestions that have improved the paper and its exposition.
We made substantial progress during the workshops \emph{Rational Points 2017} and \emph{Rational Points 2019} at Franken-Akademie Schloss Schney, and during the 2019 trimester programme \emph{\`{A} la red\'{e}couverte des points rationnels/Reinventing rational points} at the Institut Henri Poincar\'{e}, and we are very grateful to the organisers of these events.
Rachel Newton was supported by EPSRC grant EP/S004696/1 and EP/S004696/2, and UKRI Future Leaders Fellowship MR/T041609/1 and MR/T041609/2.
\end{ack}

\section{Kato's refined Swan conductor}\label{sec:Kato}

In this section, we gather some technical results from~\cite{K} relating to Kato's refined Swan conductor, extending them as necessary. For this section only, $K$ denotes a Hensel\-ian discrete valuation field with ring of integers $\O_K$ and residue field $F$ of characteristic $p>0$. 
Let $\pi$ be a uniformiser in $\O_K$ and denote by $\m_K$ the maximal ideal of $\O_K$.
Our goal is to describe, for each $q \ge 1$, the following constructions. 
\begin{itemize}
\item An increasing filtration $\{\fil_n \H^q(K)\}_{n \ge 0}$ on $\H^q(K)$ (Definition~\ref{def:Swan}). The \emph{Swan conductor} $\sw(\chi)$ of $\chi \in \H^q(K)$ is then defined to be the smallest $n \ge 0$ such that $\chi \in \fil_n \H^q(K)$.
The prime-to-$p$ part of $\H^q(K)$ is entirely contained in $\fil_0 \H^q(K)$, as are all elements split by an unramified extension of $K$ (Proposition~\ref{prop:kato61}(1)).
\item An injection $\lambda_\pi \colon \H^q(F) \oplus \H^{q-1}(F) \to \H^q(K)$ whose image coincides with $\fil_0 \H^q(K)$ (see \eqref{eq:lambda-def} and Proposition~\ref{prop:kato61}(1)). We also define $\iota^q_{n} \colon \H^q_n(F) \to \H^q_n(K)$ to be the first component of $\lambda_\pi$ (restricted to $\H^q_n(F)$), see Section~\ref{sec:iota-local}. The ``residue map'' $\partial \colon \fil_0 \H^q(K) \to \H^{q-1}(F)$ is defined to be the inverse of $\lambda_\pi$ followed by projection to the second component (Definition~\ref{def:residue}).
\item For each $r \ge 1$, a surjection $\delta_r \colon W_r \Omega^{q-1}_F \to \H^q_{p^r}(F)$ (Definition~\ref{def:delta}). Following Kato, we sometimes also denote the composition $\lambda_\pi \circ \delta_r$ simply by $\lambda_\pi$.
\item For each $n \ge 1$, an injection (the \emph{refined Swan conductor}), denoted $\rsw_n$, from $\fil_{n} \H^q(K) / \fil_{n-1} \H^q(K)$ to an object that is essentially $\Omega^q_F \oplus \Omega^{q-1}_F$ (Definition~\ref{def:rsw}).
\end{itemize}
The principal case of interest will be when $K$ is the Henselisation of the function field of a variety $X$, and we will be interested in the case $q=2$, for which $\H^q(K) = \Br(K)$.
In the details of our proofs in future sections, we will also need the equal-characteristic case (when $K$ is the function field of a variety in characteristic $p$, Henselised at a prime divisor).  Moreover, we will need the maps $\lambda_\pi$ and $\delta_r$ in a more general context, as defined in Section~\ref{sec:lambda-smooth}.

\subsection{Vanishing cycles and the Swan conductor}\label{sec:V}
In order to define the Swan conductor, we need to deal not only with fields but also with more general rings.
The following definition serves this purpose.

Let $A$ be a ring over $\O_K$, and let $i,j$ be the inclusions of the special and generic fibres, respectively, into $\Spec A$.  Denote $R = A / \m_K A$.
If $\Char K>0$, assume that $A \otimes_{\O_K} K$ is smooth over some field, so that $(\Z/n)(q-1)$ is defined.
Define
\[
V^q_n(A) := \H^q(R_\et, i^* \R j_* (\Z/n)(q-1))
\]
and $V^q(A) := \varinjlim_n V^q_n(A)$. Note that, in the case when $K$ has characteristic zero, $(\Z/n)(q-1)$
is simply a sheaf for the \'{e}tale topology. 

As an example, in the case $A=\O_K$ with $\Char K=0$ we can identify $j_*(-)$ with $\H^0(K_{nr},-)$, where $K_{nr}$ is the maximal unramified extension of $K$.  The spectral sequence implicit in the definition of $V^q_n(\O_K)$ is the inflation-restriction spectral sequence for the extension $K_{nr}/K$, and we find $V^q_n(\O_K) = \H^q_n(K)$.

The construction is functorial in the following sense.  
Let $K'/K$ be an extension of Henselian discrete valuation fields;
let $\O_{K'}$ be the ring of integers of $K'$ and let $F'$ be the residue field.
Suppose that we have a commutative diagram
\[
\begin{CD}
A @>>> A' \\
@AAA @AAA \\
\O_K @>>> \O_{K'}
\end{CD}
\]
where again $A' \otimes_{\O_{K'}} K' (= A' \otimes_{\O_K} K)$ is smooth over some field.
Let $R' = A'/\m_{K'}A'$ and let $i',j'$ be the inclusions of the special and generic fibres, respectively, of $\Spec A' \to \Spec \O_{K'}$.
We have a commutative diagram
\[
\begin{CD}
\Spec R' @>{i'}>> \Spec A' @<{j'}<< \Spec(A' \otimes_{\O_K} K) \\
@V{g}VV @V{f}VV @V{h}VV \\
\Spec R @>{i}>> \Spec A @<{j}<< \Spec(A \otimes_{\O_K} K) \\
\end{CD}
\]
of schemes.
Define
\[
V^q_n(A') := \H^q(R'_\et, (i')^* \R (j')_* (\Z/n)(q-1)).
\]
Applying $(i')^*$ to the natural base-change map
\[
f^* \R j_* (\Z/n)(q-1) \to \R j'_* (\Z/n)(q-1)
\]
and using $(i')^* f^* = g^* i^*$ gives a map
\[
g^* i^* \R j_* (\Z/n)(q-1) \to (i')^* \R j'_* (\Z/n)(q-1)
\]
and so, by adjunction, a natural map $V^q_n(A) \to V^q_n(A')$ for all $q,n$.

In the case of the field $K$, we have product maps
\[
\H^q_n(K) \times (K^\times)^{\oplus r} \to \H^{q+r}_n(K)
\]
defined by $(\chi,a_1,\dotsc,a_r) \mapsto \chi \cup \{ a_1 \} \cup \dotsb \cup \{ a_r \}$, where $\{-\}$ denotes the Kummer map $K^\times/(K^\times)^n \to \H^1(K,\mmu_n)$ as in~\eqref{eq:kummer}.
This construction can be generalised to $V_n^q(A)$ as follows.
The natural map of sheaves
\[
\R j_* (\Z/n)(q-1) \to i_* i^* \R j_* (\Z/n)(q-1).
\]
gives a natural map
\begin{equation}\label{eq:HtoV}
\H^q_n(A \otimes_{\O_K} K) = \H^q(A, \R j_* (\Z/n)(q-1)) \to V^q_n(A)
\end{equation}
for all $q,n$, which Gabber~\cite{Gabber} (see also~\cite[Theorem~09ZI]{stacks-project}) has proved to be an isomorphism if $(A, \m_K A)$ is Henselian.
(Note that this generalises the observation $V^q_n(\O_K) = \H^q_n(K)$ above.)
In that case, we mimic the construction just described and define a product
\begin{equation}\label{eq:Vcup}
\begin{gathered}
V^q_n(A) \times ((A \otimes_{\O_K} K)^\times)^{\oplus r} \to V^{q+r}_n(A) \\
(\chi, a_1, \dotsc, a_r) \mapsto \{ \chi, a_1, \dotsc, a_r \}
\end{gathered}
\end{equation}
using the Kummer map $(A \otimes_{\O_K} K)^\times \to \H^1(A \otimes_{\O_K} K, (\Z/n)(1))$ and the cup product
\[
\H^q_n(A \otimes_{\O_K} K) \times \H^1(A \otimes_{\O_K} K, (\Z/n)(1))^{\oplus r} \to \H^{q+r}_n(A \otimes_{\O_K} K).
\]
For general $A$, let $A^{(h)}$ denote the Henselisation at the ideal $\m_K A$; then the natural map $V^q_n(A) \to V^q_n(A^{(h)})$ is an isomorphism, because the stalks of $i^* \R j_* (\Z/n)(q-1)$ do not change when $A$ is replaced by $A^{(h)}$. 
We deduce $V^q_n(A) \cong \H^q(A^{(h)} \otimes_{\O_K} K)$,
which allows us to define the product~\eqref{eq:Vcup} for $A$ as well.
The products for different $n$ are compatible and so give rise to a product
\[
V^q(A) \times ((A \otimes_{\O_K} K)^\times)^{\oplus r} \to V^{q+r}(A).
\]

Taking $A=\O_K[T]$, we can now define Kato's Swan conductor.

\begin{definition}\label{def:Swan}[Kato {\cite[\S 2]{K}}]
The increasing filtration $\{ \fil_n \H^q(K) \}_{n \ge 0}$ is defined by
\[
\chi \in \fil_n \H^q(K) \Longleftrightarrow \{ \chi, 1 + \pi^{n+1} T \} = 0 \text{ in } V^{q+1}(\O_K[T]).
\]
For $\chi \in \H^q(K)$, define the \emph{Swan conductor} $\sw(\chi)$ to be the smallest $n \ge 0$ satisfying $\chi \in \fil_n \H^q(K)$.
\end{definition}

\begin{remark}
For $r \ge 1$, the map $\H^q_{p^r}(K) \to \H^q(K)$ allows us to pull the filtration back to $\H^q_{p^r}$.
By~\cite[Proposition~1.8]{K}, the map $V^{q+1}_{p^r}(\O_K[T]) \to V^{q+1}(\O_K[T])$ is injective, showing that $\chi \in \H^q_{p^r}(K)$ lies in $\fil_n \H^q_{p^r}(K)$ if and only if $\{\chi, 1 + \pi^{n+1} T\}=0$ in $V^{q+1}_{p^r}(\O_K[T])$.
\end{remark}

\begin{remark}
There is an equivalent definition using only Galois cohomology.  Namely, Kato~\cite[Proposition~6.3]{K} states that $\chi \in \H^q(K)$ lies in $\fil_n \H^q(K)$ if and only if, for any Henselian discrete valuation field $L/K$ such that $\O_K \subset \O_L$ and $\m_L = \O_L \m_K$, we have $\{ \chi_L, 1 + \pi^{n+1} \O_L \} = 0$ in $\H^{q+1}(L)$.
\end{remark}

\begin{remark}
Suppose that $K$ is a finite extension of $\Q_p$ and take $q=1$.  The filtration on $\H^1(K) = \Hom(\Gal(\bar{K}/K),\Q/\Z)$ coincides with that induced by the upper ramification filtration on $\Gal(\bar{K}/K)$.  This is an exercise in local class field theory, using the fact that the local reciprocity map identifies the upper ramification filtration on the Galois group with the unit filtration on $K^\times$.
\end{remark}

\begin{remark}
If $K$ is a finite extension of $\Q_p$ then the filtration on $\H^2(K) = \Br(K)$ is uninteresting: we have $\fil_0 \H^2(K) = \H^2(K)$.  This follows, for example, from Proposition~\ref{prop:kato61} and is related to the fact that every element of the Brauer group of $K$ is split by an unramified extension.
The filtration on $\Br(K)$ is only interesting when the residue field $F$ is not perfect.
\end{remark}

\subsection{The maps $\iota^q_n$ and $\lambda_\pi$}\label{sec:lambda}

\subsubsection{$\iota^q_n$ in the Henselian local case}\label{sec:iota-local}

Let $A$ be a Henselian local ring with fraction field $L$ and residue field $\ell$ of characteristic $p>0$.
In~\cite[\S 1.4]{K}, Kato defines a homomorphism (the ``canonical lifting'')
\[
\iota^q_n \colon \H^q_n(\ell) \to \H^q_n(L)
\]
for all positive integers $n$, which we now recall.
If $n$ is coprime to $p$, define $\iota^q_n$ simply as the composite
\[
\H^q(\ell, (\Z/n)(q-1)) \cong \H^q(A_\et, (\Z/n)(q-1)) \to \H^q(L, (\Z/n)(q-1)).
\]
If $n=p^s$ is a power of $p$, then this works only for $q=1$; in the other cases, we have not defined a sheaf $(\Z/n)(q-1)$ on $A_\et$.
In those cases we define $\iota^q_n$ by the formula
\[
\iota^q_n(\{ \chi, \bar{a}_1, \dotsc, \bar{a}_{q-1} \}) =
\{ \iota^1_n(\chi), a_1, \dotsc, a_{q-1} \}
\]
for $\chi \in \H^1_n(\ell)$ and $a_1, \dotsc, a_{q-1} \in A^\times$;
Kato has proved that this characterises a well-defined homomorphism.

\subsubsection{$\lambda_\pi$ in the Henselian DVR case}\label{sec:lambda-dvr}

We now return to our Henselian discrete valuation field $K$ with ring of integers $\O_K$ and residue field $F$.
For a fixed uniformiser $\pi \in \O_K$, we define 
\begin{equation}\label{eq:lambda-def}
\lambda_\pi \colon \H^q_n(F) \oplus \H^{q-1}_n(F) \to \H^q_n(K)
\end{equation}
as
\[
\lambda_\pi(\chi,\chi') = \iota^q_n(\chi) + \{ \iota^{q-1}_n(\chi'), \pi \}.
\]
Kato has proved that this is injective.  An important relationship between $\lambda_\pi$ and the filtration $\fil_n$ is given by the first part of the following proposition.
\begin{proposition}[Kato, {\cite[Proposition~6.1]{K}}]\label{prop:kato61}
Let $p=\Char F>0$. 
\begin{enumerate}
\item $\fil_0 \H^q(K)$ coincides with the image of $\lambda_\pi$, and furthermore
\[
\fil_0\H^q(K) = \H^q(K)(\textrm{non-p}) \oplus \ker(\H^q(K)\{p\} \to \H^q(K_{nr})\{p\}),
\]
where $(\textrm{non-p})$ denotes the prime-to-$p$ part and $K_{nr}$ the maximal unramified extension of $K$.
\item We have $\H^q(K) = \fil_0 \H^q(K)$ if $[F:F^p]=p^c < \infty$ and $q>c+1$.
\end{enumerate}
\end{proposition}

\subsubsection{The smooth-over-DVR case}\label{sec:lambda-smooth}

We will also need a version of $\lambda_\pi$ for more general rings.  Specifically, we need a definition that works for the ring $\O_K[T]$ in order to define the refined Swan conductor; and in order to prove the main result of Section~\ref{sec:tame} we need a definition that works for the coordinate ring of an affine piece of our smooth model $\X$.  Both of these are rings smooth over a Henselian discrete valuation ring.
So let $A$ be a ring smooth over $\O_K$, with $R = A/\m_K A$;
we seek an injective homomorphism
\[
\H^q_{p^r}(R) \oplus \H^{q-1}_{p^r}(R) \to V^q_{p^r}(A)
\]
coinciding with the definition of $\lambda_\pi$ in Section~\ref{sec:lambda-dvr} in the case $A=\O_K$.

In~\cite[\S 1.9]{K}, Kato treats the case $r=1$, which is all that is needed in order to define the refined Swan conductor.
The extra case we will need in Section~\ref{sec:tame} is for $r > 1$ in the case of mixed characteristic.
Assume therefore $\Char K=0$.
We will now explain how to define a homomorphism
\[
\lambda_\pi \colon \H^q_{p^r}(R) \oplus \H^{q-1}_{p^r}(R) \to V^q_{p^r}(A)
\]
for all $q \ge 2$ and $r \ge 1$ (Definition~\ref{def:lambda}). 
In Lemma~\ref{lem:lambda-comparison} we prove that our definition of $\lambda_\pi$ coincides with that of~\eqref{eq:lambda-def} in the case when $A=\O_K$ and $n=p^r$. 
We closely follow~\cite[\S1.9]{K} throughout, though we believe this contains a sign error which we will correct below.

Since $R$ has $p$-cohomological dimension $\le 1$ (see~\cite[X, Th\'eor\`eme~5.1]{SGA43}), the spectral sequence calculating $V^q_{p^r}(A)$ reduces to a short exact sequence
\begin{equation}\label{eq:infres1}
0 \to \H^1(R_\et, i^* \R^{q-1} j_* (\Z/p^r)(q-1)) \xrightarrow{\alpha} V^q_{p^r}(A) \to \H^0(R_\et, i^* \R^q j_* (\Z/p^r)(q-1)) \to 0.
\end{equation}
Following Bloch and Kato in~\cite{BK}, write $M_r^{q-1} = i^* \R^{q-1} j_* (\Z/p^r)(q-1)$.
By~\cite[Theorem~1.4]{BK}, there is a finite decreasing filtration $\{ U^m M_r^{q-1} \}_{m \ge 0}$ on $M_r^{q-1}$ with $U^0 M_r^{q-1} = M_r^{q-1}$ and with graded pieces $\gr^m(M_r^{q-1}) = U^m M_r^{q-1} / U^{m+1} M_r^{q-1}$ that can be described as follows: there is an isomorphism
\begin{equation}\label{eq:gr0}
\gr^0(M_r^{q-1}) \cong W_r\Omega^{q-1}_{R,\log} \oplus W_r \Omega^{q-2}_{R,\log}
\end{equation}
and, for $m \ge 1$, a surjection
\begin{equation}\label{eq:rhom}
\rho_m \colon \Omega^{q-2}_R \oplus \Omega^{q-3}_R \to \gr^m(M^{q-1}_r).
\end{equation}
(For $i<0$ we set $\Omega^i_R=0$.)

\begin{lemma}\label{lem:lambda-def}
Suppose $q \ge 2$. The natural map 
\[
\H^1(R_\et,M^{q-1}_r) \to \H^1(R_\et,\gr^0(M^{q-1}_r)) \cong \H^1(R_\et, W_r\Omega^{q-1}_{R,\log} \oplus W_r \Omega^{q-2}_{R,\log})
\]
is an isomorphism.
\end{lemma}
\begin{proof}
We first show $\H^1(R_\et,\gr^m(M^{q-1}_r))=0$ for $m \ge 1$, using~\eqref{eq:rhom}.
On the one hand, we have $\H^1(R_\et,\Omega^i_R)=0$ for $i \ge 0$ because $\Omega^i_R$ is a coherent sheaf on the affine scheme $\Spec R$.
On the other hand, if we let $K_m = \ker(\rho_m)$, then $\H^2(R_\et,K_m)=0$ because $\cd_p(R) \le 1$.
The long exact sequence in cohomology coming from the short exact sequence
\[
0 \to K_m \to \Omega^{q-2}_R \oplus \Omega^{q-3}_R \to \gr^m(M^{q-1}_r) \to 0
\]
shows $\H^1(R_\et,\gr^m(M^{q-1}_r))=0$ for $m \ge 1$.

A simple induction now gives $\H^1(R_\et, U^m M^{q-1}_r)=0$ for $m \ge 1$, and the short exact sequence 
\[
0 \to U^1 M_r^{q-1} \to M_r^{q-1} \to \gr^0(M_r^{q-1}) \to 0
\]
completes the proof.
\end{proof}

The group $\H^1(R_\et, W_r\Omega^{q-1}_{R,\log} \oplus W_r \Omega^{q-2}_{R,\log})$ appearing in Lemma~\ref{lem:lambda-def} is, by definition of the sheaf $(\Z/n)(q-1)$, equal to $\H^q_{p^r}(R) \oplus \H^{q-1}_{p^r}(R)$.

\begin{definition}\label{def:lambda}
Suppose $\Char K=0$.
For $q \ge 2$, define
\[
\lambda_\pi \colon \H^q_{p^r}(R) \oplus \H^{q-1}_{p^r}(R) \to V^q_{p^r}(A)
\]
to be $(-1)^{q-1}$ times the inverse of the isomorphism of Lemma~\ref{lem:lambda-def}, composed with the map labelled $\alpha$ in~\eqref{eq:infres1}.  Moreover, define
\[
\iota^q_{p^r} \colon \H^q_{p^r}(R) \to V^q_{p^r}(A)
\]
to be the first component of $\lambda_\pi$.
\end{definition}

\begin{remark}
In the case $A=\O_K$, the \'etale cohomology groups become Galois cohomology: the sequence~\eqref{eq:infres1} becomes
\begin{multline*}
0 \to \H^1(F, \H^{q-1}(K_{nr}, (\Z/p^r)(q-1))) \to \H^q(K, (\Z/p^r)(q-1)) \\ \xrightarrow{\res} \H^q(K_{nr}, (\Z/p^r)(q-1))
\end{multline*}
where $K_{nr}$ is the maximal unramified extension of $K$.
The map $(-1)^{q-1}\lambda_\pi$ factors as
\begin{multline*}
\H^1(F, W_r \Omega^{q-1}_{\log}(F^s) \oplus W_r \Omega^{q-2}_{\log}(F^s)) \xleftarrow{\cong}
\H^1(F, \H^{q-1}(K_{nr}, (\Z/p^r)(q-1))) \\ \to \H^q(K, (\Z/p^r)(q-1)).
\end{multline*}
(Here $F^s$ is a separable closure of $F$, and we identify $\Gal(K_{nr}/K) \cong \Gal(F^s/F)$ without further comment.)
\end{remark}

\subsubsection{The map $\delta_r$}

We keep the notation of the previous subsection: $A$ is a ring smooth over $\O_K$, and $R=A/\m_K A$.
By~\cite[\S1.3]{K} (see also~\cite[Lemme 2]{CTSS} and~\cite[I.3.3]{I}), there is an exact sequence
\begin{equation}\label{eq:F-1}
0\to W_r\Omega^{q-1}_{\log}\to W_r\Omega^{q-1}\xrightarrow{C^{-1}-1} W_r\Omega^{q-1}/dV^{r-1}\Omega^{q-2}\to 0
\end{equation}
of sheaves on $R_\et$.
Here we have abused notation by using $C^{-1}$ to denote the map denoted $\mathrm{F}$ in~\cite[\S1.3]{K} and~\cite[\S I.2.17]{I}, which coincides with the inverse Cartier operator when $r=1$. The map $V$ is the Verschiebung defined in~\cite[\S I.1]{I}.
By definition, we have $\H^q_{p^r}(R) = \H^1(R_\et, W_r\Omega^{q-1}_{\log})$.

\begin{definition}\label{def:delta}
We define
\[
\delta_r \colon W_r\Omega_R^{q-1}/dV^{r-1}\Omega_R^{q-2} \to \H^q_{p^r}(R)
\]
to be the boundary map in the long exact sequence of cohomology corresponding to~\eqref{eq:F-1}.
We also use $\delta_r$ to denote the map $W_r\Omega_R^{q-1} \to \H^q_{p^r}(R)$ obtained by precomposing with the natural surjection $W_r\Omega_R^{q-1}\to W_r\Omega_R^{q-1}/dV^{r-1}\Omega_R^{q-2}$.
\end{definition}

Following Kato, we sometimes use $\lambda_\pi$ to denote the composition
\begin{equation}\label{eq:lambda-comp}
W_r \Omega^{q-1}_R \oplus W_r \Omega^{q-2}_R \xrightarrow{\delta_r}
\H^q_{p^r}(R) \oplus \H^{q-1}_{p^r}(R) \xrightarrow{\lambda_\pi} V^q_{p^r}(A).
\end{equation}

The following lemma will be used in the proofs of Lemmas~\ref{lem:lambda-comparison} and~\ref{lem:lambda-basechange}.
\begin{lemma}\label{lem:delta-wedge}
Let $w \in W_r \Omega^{q-1}_R$ and $a \in R^\times$.  Let $\dlog(a) \in W_r \Omega^1_{R,\log}$ be as in~\cite[\S I.3.23]{I} and let
\[
f_a \colon \H^q_{p^r}(R) \to \H^{q+1}_{p^r}(R)
\]
be the homomorphism $\H^1(R_\et, W_r \Omega^{q-1}_{R,\log}) \to \H^1(R_\et, W_r \Omega^q_{R,\log})$
induced by
\[
W_r \Omega^{q-1}_{R,\log} \to W_r \Omega^q_{R,\log}, \qquad x \mapsto x \wedge \dlog(a).
\]
Then we have
\[
\delta_r(w \wedge \dlog(a)) = f_a(\delta_r(w)) = \delta_r(w) \cup \{ a \}.
\]
\end{lemma}
In other words, we can take the wedge product ``inside the $\H^1$''.
\begin{proof}
The sequences~\eqref{eq:F-1} give the rows of a diagram of sheaves on $R_\et$ as follows:
\[
\begin{CD}
0 @>>> W_r \Omega^{q-1}_{\log} @>>> W_r \Omega^{q-1} @>{C^{-1}-1}>> W_r \Omega^{q-1}/dV^{r-1}\Omega^{q-2} @>>> 0 \\
@. @V{\wedge \dlog\bar{a}}VV @V{\wedge \dlog\bar{a}}VV @V{\wedge  \dlog\bar{a}}VV \\
0 @>>> W_r \Omega^{q}_{\log} @>>> W_r \Omega^{q} @>{C^{-1}-1}>> W_r \Omega^{q}/dV^{r-1}\Omega^{q-1} @>>> 0.
\end{CD}
\]
It is easy to check that this diagram commutes, using $C^{-1}(\dlog(a))=\dlog(a)$ (see \cite[\S I.3.23]{I}).
Taking cohomology now gives the first claimed equality.  The second is the final statement of~\cite[\S 1.3]{K}.
\end{proof}

\subsubsection{Compatibility of the two definitions of $\lambda_\pi$}
We can now prove that the definition of $\lambda_\pi$ in Definition~\ref{def:lambda} is compatible with that described in \S\ref{sec:lambda-dvr}.

\begin{lemma}\label{lem:lambda-comparison}
In the case $\Char K=0$, $A=\O_K$ and $q \ge 2$, the map $\lambda_\pi$ defined in Definition~\ref{def:lambda} agrees with that defined in~\S\ref{sec:lambda-dvr}.
\end{lemma}
\begin{proof}
We will prove this by induction on $q$, by showing that our map $\lambda_\pi$ of Definition~\ref{def:lambda} satisfies the characterisation given in~\S\ref{sec:lambda-dvr}.
Let $\iota^1_{p^r} \colon \H^1(F) \to \H^1(K)$ be the natural map defined in \S\ref{sec:iota-local}.
Recall that, for $q \ge 2$, the notation $\iota^q_{p^r}$ refers to the first component of $\lambda_\pi$.
Specifically, we will show:
\begin{enumerate}
\item\label{lambda-1} for $\chi \in \H^1_{p^r}(F)$ and $a \in \O_K^\times$, we have $\iota^2_{p^r}(\{ \chi, \bar{a} \}) = \{ \iota^1_{p^r}(\chi), a \}$;
\item\label{lambda-2} for $\chi \in \H^1_{p^r}(F)$, we have $\lambda_\pi(0, \chi) = \{\iota^1_{p^r}(\chi), \pi \}$;
\item\label{lambda-3} for $q \ge 2$, $\chi \in \H^q_{p^r}(F)$ and $a \in \O_K^\times$, we have $\iota^q_{p^r}(\{ \chi, \bar{a} \}) = \{ \iota^{q-1}_{p^r}(\chi), a \}$;
\item\label{lambda-4} for $q \ge 2$, $\chi \in \H^{q-1}_{p^r}(F)$ and $a \in \O_K^\times$, we have $\lambda_\pi(0, \{ \chi, \bar{a} \}) = - \{ \lambda_\pi(0,\chi), a \}$.
\end{enumerate}
Then~(\ref{lambda-1}) and~(\ref{lambda-2}) prove the claim for $q=2$; statement~(\ref{lambda-3}) shows that the first component of $\lambda_\pi$ agrees with the previous definition of $\iota^q_{p^r}$ in \S\ref{sec:iota-local} for all $q \ge 2$; and all four statements together imply the final part of the characterisation, namely $\lambda_\pi(0,\chi) = \{ \iota^{q-1}_{p^r}(\chi), \pi \}$ for $\chi \in \H^{q-1}_{p^r}(F)$.

We prove~(\ref{lambda-3}) and~(\ref{lambda-4}) first.
Let $q \ge 2$ and let $a$ be an element of $\O_K^\times$.  Let $\bar{a} \in F^\times$ be the reduction of $a$ and let $\{\bar{a}\}$ be its class in $\H^1(F,(\Z/p^r)(1))$.  We claim that the following diagram commutes.
\begin{equation*}\label{eq:lambdacd1}
\begin{CD}
\H^q(K, (\Z/p^r)(q-1))@>{\phantom{longlabel}\{a\} \cup \phantom{longlabel}}>>\H^{q+1}(K, (\Z/p^r)(q))\\
@AAA @AAA\\
\H^1(F, \H^{q-1}(K_{nr}, (\Z/p^r)(q-1)))@>{\phantom{label}\{a\} \cup\phantom{label}}>> \H^1(F, \H^{q}(K_{nr}, (\Z/p^r)(q)))\\
@V{\cong}VV @VV{\cong}V\\
\H^1(F, W_r \Omega^{q-1}_{\log}(F^s))\oplus @>{(\dlog\bar{a} \wedge, \dlog\bar{a} \wedge)}>> \H^1(F, W_r \Omega^{q}_{\log}(F^s))\oplus  \\
 \H^1(F, W_r \Omega^{q-2}_{\log}(F^s)) @.  \H^1(F, W_r \Omega^{q-1}_{\log}(F^s))\\
@| @|\\
\H^q_{p^r}(F) \oplus \H^{q-1}_{p^r}(F) @>{(\cup (-1)^{q-1}\{\bar{a}\},\cup (-1)^{q}\{\bar{a}\})}>> \H^{q+1}_{p^r}(F) \oplus \H^{q}_{p^r}(F)
\end{CD}
\end{equation*}
Here the horizontal maps are as follows. The first horizontal map is cup product with the class of $a$ in $\H^1(K, (\Z/p^r)(1))$.
The second horizontal map is induced by cup product with the class of $a$ in $\H^1(K_{nr}, (\Z/p^r)(1))$.
The third horizontal map is that induced on cohomology by the homomorphism $\omega \mapsto \dlog\bar{a} \wedge \omega$ on each factor.
The fourth horizontal map is given by cup products as written. 
That the top square commutes is~\cite[\S1.2, Lemma~2 (2)]{K3} (and it is in order to apply that lemma that we have put the cup products on the left).
That the middle square commutes is shown by the formula after~\cite[Corollary~1.4.1]{BK}.
The bottom square commutes by Lemma~\ref{lem:delta-wedge} and the anti-commutativity of the wedge product. 

In this diagram, the map going upwards from bottom left to top left is, by definition, $(-1)^{q-1} \lambda_\pi$, and that from bottom right to top right is $(-1)^q \lambda_\pi$.  Let $\chi \in \H^q_{p^r}(F)$.  The commutativity of the diagram gives
\begin{align*}
(-1)^q \lambda_\pi(\chi \cup (-1)^{q-1}\{\bar{a}\}, 0) &=
(-1)^{q-1} \{ a \} \cup \lambda_\pi(\chi,0) \\
- \lambda_\pi(\chi \cup \{\bar{a}\}, 0) &= (-1)^{q-1} (-1)^q \lambda_\pi(\chi,0) \cup \{ a \} \\
\lambda_\pi(\chi \cup \{\bar{a}\}, 0) &= \lambda_\pi(\chi,0) \cup \{ a \},
\end{align*}
proving (\ref{lambda-3}).  Similarly, for $\chi \in \H^{q-1}_{p^r}(F)$ we have
\begin{align*}
(-1)^q \lambda_\pi(0, \chi \cup (-1)^q\{\bar{a}\}) &=
(-1)^{q-1} \{ a \} \cup \lambda_\pi(0,\chi) \\
\lambda_\pi(0, \chi \cup \{\bar{a}\}) &= (-1)^{q-1} (-1)^q \lambda_\pi(0, \chi) \cup \{ a \} \\
\lambda_\pi(0, \chi \cup \{\bar{a}\}) &= -\lambda_\pi(0,\chi) \cup \{ a \},
\end{align*}
proving (\ref{lambda-4}).

The case $q=1$ is similar but easier.  The sequence~\eqref{eq:infres1} is simply the inflation-restriction sequence, and one checks that the following diagram commutes.
\begin{equation*}\label{eq:lambdacd2}
\begin{tikzcd}[column sep = huge]
 \H^1(K, \Z/p^r)\arrow[r,"\{a\} \cup"]& \H^2(K, (\Z/p^r)(1))\\
 \H^1(F,\Z/p^r)\arrow[d, equal]\arrow[u, "\inf"]\arrow[r,"a"] & \H^1(F, \H^1(K_{nr}, (\Z/p^r)(1)))\ar[u] \ar[d,"\cong"]\\
 \H^1(F,\Z/p^r)\arrow[r, "{(\dlog\bar{a},0)}"]\arrow[d, "\epsilon"]
 & \H^1(F, W_r \Omega^1_{\log}(F^s) \oplus W_r \Omega^0_{\log}(F^s))\arrow[d,equal]\\
 \H^1_{p^r}(F)\arrow[r,"{(\cup\{\bar{a}\},0)}"] &\H^2_{p^r}(F) \oplus\H^1_{p^r}(F)
 \end{tikzcd}
\end{equation*}
The bottom left vertical map is induced by the isomorphism $\epsilon: \Z/p^r \to W_r \Omega^0_{\log}(F^s)$ (see~\cite[Proposition~3.28]{I}). The first horizontal map is cup product with the class of $a$ in $\H^1(K,(\Z/p^r)(1))$, the second horizontal map is that induced on cohomology by sending $1$ to the class of $a$ in $\H^1(K_{nr}, (\Z/p^r)(1))$, the third horizontal map is that induced on cohomology by $1 \mapsto (\dlog\bar{a},0)$, and the fourth horizontal map is the cup product map as written.

The left-hand column, from bottom to top, is the map $\iota^1_{p^r}$, and the right-hand column from bottom to top is $-\lambda_\pi$. Let $\chi \in \H^1_{p^r}(F)$. Commutativity of this diagram gives
\begin{align*}
-\lambda_\pi(\chi \cup \{\bar{a}\},0) &= \{a\} \cup \iota^1_{p^r}(\chi) \\
\lambda_\pi(\chi \cup \{ \bar{a}\}, 0) &= \iota^1_{p^r}(\chi) \cup \{ a \}
\end{align*}
proving (\ref{lambda-1}).

Finally, we have a third diagram
\begin{equation*}\label{eq:lambdacd3}
\begin{tikzcd}
 \H^1(K, \Z/p^r)\arrow[r, "{\{\pi\} \cup}"]& \H^2(K, (\Z/p^r)(1))\\
 \H^1(F,\Z/p^r)\arrow[r, "\pi"]\arrow[u, "\inf"]\arrow[d, equal]& \H^1(F, \H^1(K_{nr}, (\Z/p^r)(1)))\arrow[u] \arrow[d,"\cong"]\\
 \H^1(F,\Z/p^r)\arrow[r, "{(0, \epsilon)}"] & \H^1(F, W_r \Omega^1_{\log}(F^s) \oplus W_r \Omega^0_{\log}(F^s)) 
 \end{tikzcd}
\end{equation*}
in which the bottom horizontal map is induced by the canonical isomorphism $\epsilon\colon \Z/p^r \to W_r \Omega^0_{\log}(F^s)$. 
Again, this commutes by~\cite[formula after Corollary~1.4.1]{BK} and~\cite[\S1.2, Lemma~2 (2)]{K3}.
For $\chi \in \H^1_{p^r}(F)$, commutativity of this diagram gives $-\lambda_\pi(0,\chi) = \{\pi\} \cup \iota^1_{p^r}(\chi)$, proving (\ref{lambda-2}) and completing the proof of the lemma.
\end{proof}

\subsubsection{Change of ring}
We finish this section on $\lambda_\pi$ by describing how it behaves with respect to base change.

Let $K'/K$ be an extension of Henselian discrete valuation fields of characteristic zero with finite ramification index.
Let $\O_{K'}$ be the ring of integers of $K'$ and let $F'$ be the residue field.
Suppose that we have a commutative diagram
\[
\begin{CD}
A @>{\phi}>> A' \\
@AAA @AAA \\
\O_K @>>> \O_{K'}
\end{CD}
\]
where $A$ is smooth over $\O_K$ and $A'$ is smooth over $\O_{K'}$.
Let $R' = A'/\m_{K'}A'$ and let $i',j'$ be the inclusions of the special and generic fibres, respectively, of $\Spec A' \to \Spec \O_{K'}$.
As described in Section~\ref{sec:V}, there are natural maps $\phi_* \colon V^q_n(A) \to V^q_n(A')$.
Let $\bar{\phi} \colon R \to R'$ be the map on residue rings induced by $\phi$, and
$\bar{\phi}_* \colon \H^q_n(R) \to \H^q_n(R')$ and $\bar{\phi}_* \colon W_r \Omega^q_R \to W_r \Omega^q_{R'}$ the induced maps.
Let $\pi'$ be a uniformiser in $\O_{K'}$.

\begin{lemma}\label{lem:lambda-basechange}
In the situation described above, let $e$ be the ramification index of $K'/K$ and define $\bar{a} \in F'$ to be the reduction of $\phi(\pi)(\pi')^{-e}$.
Then, for all $q \ge 2$ and $r \ge 1$, the following diagram commutes:
\begin{equation}\label{eq:2.4}
\begin{tikzcd}[column sep=huge]
W_r \Omega^{q-1}_{R'} \oplus W_r \Omega^{q-2}_{R'} \arrow[r,"\delta_r"] &
\H^q_{p^r}(R') \oplus \H^{q-1}_{p^r}(R') \arrow[r,"\lambda_{\pi'}"] & V^q_{p^r}(A') \\
W_r \Omega^{q-1}_{R} \oplus W_r \Omega^{q-2}_{R} \arrow[r,"\delta_r"]
\arrow{u}[']{(\alpha,\beta) \mapsto (\bar{\phi}_*\alpha + \bar{\phi}_*\beta \wedge \dlog(\bar{a}), e \bar{\phi}_*\beta)}
& \H^q_{p^r}(R) \oplus \H^{q-1}_{p^r}(R) \arrow[r,"\lambda_{\pi}"] 
\arrow{u}[']{(\alpha,\beta) \mapsto (\bar{\phi}_*\alpha + \{ \bar{\phi}_*\beta, \bar{a} \}, e \bar{\phi}_*\beta)}
& V^q_{p^r}(A) \arrow[u,"\phi_*"']
\end{tikzcd}
\end{equation}
\end{lemma}
\begin{proof}
We go through the steps of the construction of $\lambda_\pi$.
Let $g \colon \Spec R' \to \Spec R$ be the morphism corresponding to $\bar{\phi}$.
The natural map
\[
i^* \R j_* (\Z/n)(q-1) \to g_* (i')^* \R j'_* (\Z/n)(q-1)
\]
of sheaves on $R_\et$ induces a map between the sequences~\eqref{eq:infres1} for $A$ and $A'$.
The definition of the Bloch--Kato filtration on $M^{q-1}_r = i^* \R^{q-1} j_* (\Z/p^r)(q-1)$ shows that the map $\phi_*: M^{q-1}_{r,A} \to M^{q-1}_{r,A'}$ respects the filtration, so induces a map on $\gr^0$.

The sheaf $M^{q-1}_{r,A}$ is locally generated by symbols; we now explain what this means.  Let $t_1, \dotsc, t_{q-1}$ be local sections of $i^* j_* \Gm$.  The Kummer sequence allows us to push these forward into $M^1_{r,A}$, and the cup product of the resulting classes is a local section of $M^{q-1}_{r,A}$ which we denote by $\{ t_1, \dotsc, t_{q-1} \}$.
Bloch and Kato prove that the resulting ``symbols'' locally generate the sheaf $M^{q-1}_{r,A}$ in the \'etale topology.

The isomorphism $\gr^0(M_r^{q-1}) \to W_r\Omega^{q-1}_{R,\log} \oplus W_r \Omega^{q-2}_{R,\log}$
 of~\eqref{eq:gr0} is described in the formula after~\cite[Corollary~1.4.1]{BK}:
 if $x_1, \dotsc, x_{q-1}$ are local sections of $\Gm$ on $R_\et$ and $\tilde{x}_1, \dotsc, \tilde{x}_{q-1}$ are any lifts of the $x_i$ to $i^* (\Gm)_A$, then
we have
\begin{align*}
\{ \tilde{x}_1, \dotsc, \tilde{x}_{q-1} \} &\mapsto (\dlog x_1 \wedge \dotsb \wedge \dlog x_{q-1}, 0) \\
\{ \tilde{x}_1, \dotsc, \tilde{x}_{q-2}, \pi \} &\mapsto (0, \dlog x_1 \wedge \dotsb \wedge \dlog x_{q-2}).
\end{align*}
 
Working in $M^{q-1}_{r,A'}$ we have
\begin{align*}
\phi_* \{ \tilde{x}_1, \dotsc, \tilde{x}_{q-2}, \pi \}
&= \{ \phi(\tilde{x}_1), \dotsc, \phi(\tilde{x}_{q-2}), \phi(\pi) \} \\
&= \{ \phi(\tilde{x}_1), \dotsc, \phi(\tilde{x}_{q-2}), a (\pi')^e \} \\
&= \{ \phi(\tilde{x}_1), \dotsc, \phi(\tilde{x}_{q-2}), a \} + e \{ \phi(\tilde{x}_1), \dotsc, \phi(\tilde{x}_{q-2}), \pi' \}
\end{align*}
where $a=\phi(\pi)(\pi')^{-e}$.
Therefore the isomorphisms~\eqref{eq:gr0} for $A$ and $A'$ satisfy the following commutative diagram:
\[
\begin{CD}
\gr^0(M^{q-1}_{r,A'}) @>>> W_r \Omega^{q-1}_{R',\log} \oplus W_r \Omega^{q-2}_{R',\log} \\
@AAA @AA{(\alpha,\beta) \mapsto (\bar{\phi}_*\alpha + \bar{\phi}_*\beta \wedge \dlog\bar{a}, e \bar{\phi}_*\beta)}A \\
\gr^0(M^{q-1}_{r,A}) @>>> W_r \Omega^{q-1}_{R,\log} \oplus W_r \Omega^{q-2}_{R,\log} \\
\end{CD}
\]
The rest of the proof follows from Lemma~\ref{lem:delta-wedge}.
\end{proof}

\subsection{The refined Swan conductor}

Equipped with the map $\lambda_\pi$ for the ring $\O_K[T]$, we can now define Kato's refined Swan conductor. We follow the exposition in~\cite[4.5--4.6]{Borger}. 

Firstly, let us define the group in which the refined Swan conductor lives. 
To start with, consider the scheme $\Spec \O_K$ with the divisor $D$ given by the closed point.  The coherent sheaf $\Omega^1_{\O_K}(\log D)$ of differentials with log poles at $D$ is, concretely, the $\O_K$-submodule of $\Omega^1_K$ generated by $\Omega^1_{\O_K}$ together with $\dlog\pi = d\pi/\pi$ for any uniformiser $\pi$; this submodule is independent of the choice of $\pi$.  (See~\cite[\href{https://stacks.math.columbia.edu/tag/0FMU}{Section 0FMU}]{stacks-project} for more details on differentials with log poles.)
This sheaf fits into a short exact sequence
\[
0 \to \Omega^1_{\O_K} \to \Omega^1_{\O_K}(\log D) \xrightarrow{\res} \O_K \to 0,
\]
and a choice of uniformiser $\pi$ gives a splitting $a \mapsto a \dlog\pi$.
Now define $\omega^1_F$ to be the pullback of this coherent sheaf under the closed immersion $\Spec F \to \Spec \O_K$.  Concretely, we have $\omega^1_F = \Omega^1_{\O_K}(\log D) \otimes_{\O_K} F$.
(A more direct definition is that $\omega^1_F$ is the sheaf of absolute K\"ahler differentials on the log scheme obtained by equipping $\Spec F$ with the ``standard log point'' log structure coming from its embedding in $\Spec \O_K$, see \cite[1.1, 1.7]{KatoLog}.)
For $q\geq 1$, let $\omega^q_F$ denote the $q$th exterior power of $\omega^1_F$. The map $\Omega^{q-1}_F\to\omega^q_F$ given by $\eta\mapsto \eta\wedge \dlog \pi$ yields a splitting of the natural exact sequence
\[0\to\Omega^q_F\to \omega^q_F \xrightarrow{\res} \Omega^{q-1}_F\to 0.\]
The refined Swan conductor at level $n$ will be an element of $\m_K^{-n}\otimes_{\O_K} \omega^q_F$, which is also non-canonically isomorphic to $\Omega^q_F \oplus \Omega^{q-1}_F$.

From \eqref{eq:lambda-comp} applied to $A=\O_K[T]$, we have the map $\lambda_\pi \colon \Omega^q_{F[T]} \oplus \Omega^{q-1}_{F[T]} \to V^{q+1}(\O_K[T])$.
Kato~\cite[Theorem~5.1]{K} proves the following: if $\chi$ is an element of $\fil_n \H^q(K)$ for $n \ge 1$, then there exists a unique $(\alpha,\beta) \in \Omega^q_F \oplus \Omega^{q-1}_F$ such that
\begin{equation}\label{eq:rsw}
\{ \chi, 1 + \pi^n T \} = \lambda_\pi(T\alpha, T\beta) \text{ in } V^{q+1}(\O_K[T]).
\end{equation}
Note that here $\lambda_\pi$ really means $\lambda_\pi \circ \delta_1$, as in \eqref{eq:lambda-comp}.

\begin{definition}\label{def:rsw} 
Let $n \ge 1$.
Given $\chi \in \fil_n \H^q(K)$, the \emph{refined Swan conductor} of $\chi$ is
\[
\rsw_{n}(\chi) = \pi^{-n}(\alpha+\beta\wedge \dlog \pi ) \in \m_K^{-n}\otimes_{\O_K} \omega^q_F 
\]
where $\alpha,\beta$ are as in~\eqref{eq:rsw}. We will often write $[\alpha, \beta]_{\pi, n}$ as shorthand for $\pi^{-n}(\alpha+\beta\wedge \dlog \pi )$.
\end{definition}
For $n\geq 1$, the refined Swan conductor defines an injective homomorphism
\[\rsw_n: \fil_n \H^q(K)/\fil_{n-1} \H^q(K)\hookrightarrow \m_K^{-n}\otimes_{\O_K} \omega^q_F\]
as shown in \cite[Corollary~5.2]{K}.

\begin{remark} 
The pair $(\alpha, \beta)$ in~\eqref{eq:rsw} depends on the choice of uniformiser $\pi$; however,
$\rsw_n(\chi)$ is independent of the choice of $\pi$, as stated by Kato~\cite[Definition~5.3]{K}.
This motivates the choice of $\m_K^{-n}\otimes_{\O_K} \omega^q_F$ as the target group and the definition of $\rsw_n(\chi)$.
The precise dependence of $(\alpha,\beta)$ can also be seen as a consequence of Lemma~\ref{lem:rsw-basechange}.
\end{remark}

We now prove several auxiliary results about the refined Swan conductor.

\begin{lemma}\label{lem:rsw-basechange} 
Suppose $\Char K=0$.
Let $K'/K$ be a finite extension of Henselian discrete valuation fields of ramification index $e$. Let $\pi'$ be a uniformiser in $K'$, let $F'$ be the residue field of $K'$ and define $\bar{a} \in F'$ to be the reduction of $\pi(\pi')^{-e}$.  
Let $\chi$ be an element of $\fil_n \H^q(K)$, and let
\[
\res \colon \H^q(K) \to \H^q(K')
\]
be the restriction map.
Then $\res(\chi)$ lies in $\fil_{en} \H^q(K')$. Furthermore, for $n\geq 1$, setting $\rsw_{n}(\chi) = [\alpha, \beta]_{\pi, n}$ we have
\[
\rsw_{en}(\res(\chi)) = [\bar{a}^{-n}(\alpha + \beta \wedge \dlog\bar{a}), \bar{a}^{-n} e \beta]_{\pi', en}.
\]
\end{lemma}
\begin{proof}
That $\res(\chi)$ lies in $\fil_{en} \H^q(K')$ follows from the characterisation of $\fil_n$ given in~\cite[Proposition~6.3]{K}.
Lemma~\ref{lem:lambda-basechange} gives
\[
\lambda_{\pi'} ( T(\alpha + \beta \wedge \dlog\bar{a}), e T \beta )
= \{ \res{\chi}, 1 + \pi^n T \} = \{ \res{\chi}, 1 + (\pi')^{en} a^n T \},
\]
where $a = \pi(\pi')^{-e}$.
Applying Lemma~\ref{lem:lambda-basechange} a second time to the automorphism of $\O_K[T]$ defined by $T \mapsto a^n T$ proves the claimed formula.
\end{proof}

The following lemma is implicit in~\cite[Proposition~5.4]{K}, which is stated without proof.  For completeness, we provide a proof here.

\begin{lemma}\label{lem:dbna} 
Let $\chi$ be an element of $\fil_n \H^q(K)$ with $\rsw_{n}(\chi) = [\alpha,\beta]_{\pi, n}$.
Then $d\alpha=0$ and $d\beta = (-1)^q n\alpha$.
\end{lemma}

We first prove a lemma whose first part will be used in the proof of Lemma~\ref{lem:dbna} and whose second part will be used in the proof of Lemma~\ref{lem:rsw-same} below. We use $C$ to denote the Cartier operator.  Recall the definition of $\delta_1 \colon \Omega^q_R \to \H^{q+1}_p(R)$ from Definition~\ref{def:delta}.

\begin{lemma}\label{lem:cupT}
Let $R = F(T)$ and let $\alpha$ be an element of $\Omega^{q-1}_F$ for some $q \ge 2$. 
\begin{enumerate}
\item \label{Ta} We have $\delta_1(T d\alpha) =  (-1)^q \{\delta_1(T \alpha) , T \} \in \H^{q+1}_p(R).$ 
\item \label{Ca} Suppose $d\alpha=0$, so that $C(\alpha)$ is defined. Then $\delta_1(T C(\alpha)) = \delta_1(T^p \alpha)$.
\end{enumerate}
\end{lemma}
\begin{proof}
By~\cite[1.3.2]{K}, the subgroup $d \Omega^{q-1}_R$ is in the kernel of $\delta_1 \colon \Omega^q_R \to \H^{q+1}_p(R)$, so
\[
0 = \delta_1(d(T\alpha)) = \delta_1(T d\alpha) + \delta_1(dT \wedge \alpha)
\]
and therefore
\[
\delta_1(T d\alpha) = - \delta_1(dT \wedge \alpha) = (-1)^q \delta_1(T \alpha \wedge dT/T) = (-1)^q \{\delta_1(T \alpha) , T \}
\]
by the last formula of~\cite[\S1.3]{K}, proving~(\ref{Ta}). 

To prove~(\ref{Ca}), note firstly that $d(T^p \alpha) = T^p d\alpha = 0$, so that $C(T^p \alpha)$ is defined.
The image of $C^{-1}-1$ is in the kernel of $\delta_1$, so we have
\[
0 = \delta_1((C^{-1}-1)(C(T^p \alpha))) = \delta_1(T^p \alpha - C(T^p \alpha))
\]
and therefore
\[
\delta_1(T^p \alpha) = \delta_1(C(T^p \alpha)) = \delta_1(T C(\alpha)). \qedhere
\]
\end{proof}

\begin{proof}[Proof of Lemma~\ref{lem:dbna}] 
By definition of $\rsw_{n}$, we have $\{ \chi, 1+\pi^n T \} = \lambda_\pi(T\alpha, T\beta)$ in $V^{q+1}(\O_K[T])$.
We would like to take the cup product with $-\pi^n T$, but as this is not a unit in $K[T]$ we first have to pass to a larger ring.
Let $A$ be the Henselisation of the localisation of $\O_K[T]$ at the ideal $\m_K 
\O_K[T]$. 
By~\cite[1.8.1]{K}, the natural map $V^{q+2}(\O_K[T]) \to V^{q+2}(A) = \H^{q+2}(A \otimes_{\O_K} K)$ is injective. Working in $V^{q+2}(A)$, we compute
\begin{align}
0 &= \{ \chi, 1+\pi^n T, - \pi^n T \} \notag \\
&= \{ \lambda_\pi(T\alpha, T\beta), - \pi^n T \} \notag \\
&= \{ \lambda_\pi(T\alpha, T\beta), T \} + n \{ \lambda_\pi(T\alpha, T\beta), \pi \} + \{ \lambda_\pi(T\alpha, T\beta), -1 \}. \label{threebits}
\end{align}
The last term in~\eqref{threebits} is zero; this follows from Lemma~\ref{lem:delta-wedge} and $\dlog(-1)=0$.
Let $\iota_p^q \colon \H^q_p(F(T)) \to V^q(A)$ be the canonical lifting map of Section~\ref{sec:iota-local}, which is the first component of $\lambda_\pi$.
By Lemma~\ref{lem:cupT}\eqref{Ta} and the definition~\eqref{eq:lambda-def}, the first term of~\eqref{threebits} is 
\begin{align*}
\{ \lambda_\pi(T\alpha,T\beta), T \}
&= \{ \iota^{q+1}_p(\delta_1(T\alpha)),T \} + \{ \iota^q_p(\delta_1(T\beta)),\pi,T \} \\
&= \{ \iota^{q+1}_p(\delta_1(T\alpha)),T \} - \{ \iota^q_p(\delta_1(T\beta)),T,\pi \} \\
&= \iota^{q+2}_p \{ \delta_1(T\alpha),T \} - \{ \iota^{q+1}_p \{ \delta_1(T\beta),T\} ,\pi \} \\
&= (-1)^{q+1} \big( \iota^{q+2}_p(\delta_1(Td\alpha)) + \{ \iota^{q+1}_p(\delta_1(Td\beta)), \pi \} \big) \\
&= (-1)^{q+1} \lambda_\pi(Td\alpha, T d\beta).
\end{align*}

For the middle term of~\eqref{threebits} we have
\[
\{ \lambda_\pi(T\alpha, T\beta), \pi \} = \{ \iota_p^{q+1}(\delta_1(T\alpha)), \pi \} + \{ \iota_p^q(\delta_1(T\beta)), \pi, \pi \} \\
= \lambda_\pi(0, T\alpha),
\]
because again $\{ \iota_p^q(\delta_1(T\beta)),-1\}=0$ and $\{ \pi, -\pi \} = 0$.
This produces
\[
(-1)^{q+1} \lambda_\pi(Td\alpha, Td\beta) + n \lambda_\pi(0, T\alpha) = 0,
\]
in $V^{q+2}(A)$ and therefore also in $V^{q+2}(\O_K[T])$.
The result now follows from the injectivity of $\lambda_\pi$ and~\cite[Lemma~3.8]{K}.
\end{proof}

We conclude this subsection with a description of the refined Swan conductor of $p\chi$ when $\sw(\chi)$ is sufficiently large, in the $\Char(K)=0$ setting.

\begin{lemma}\label{lem:rsw-same} 
Suppose $\Char(K)=0$, let $e$ be the absolute ramification index of $K$, and set $e'=ep/(p-1)$. Let $u=p/\pi^e$ and let $\bar{u}\in F^\times$ be its reduction.
Let $\chi\in \H^q(K)$ have $\sw(\chi)=n>0$ and write $\rsw_{n}(\chi)=[\alpha,\beta]_{\pi, n}$.
\begin{enumerate}
\item \label{swdown} Suppose $n\geq e'-1$. Then $p\chi\in\fil_{n-e}\H^q(K)$. 
\item \label{rswsame} Suppose $n> e'$. Then $\rsw_{n-e}(p\chi)=[\bar{u}\alpha,\bar{u}\beta]_{\pi,n-e}$ and consequently, $\sw(p\chi)=n-e$.
\item\label{rswe'} Suppose $n=e'$. Then $d\alpha=d\beta=0$ and \[\rsw_{e'-e}(p\chi)=\Bigl[\bar{u}\alpha + C(\alpha),\bar{u}\beta+C(\beta)\Bigr]_{\pi,e'-e}.\] 
\end{enumerate}
\end{lemma}

\begin{proof}
To prove~\eqref{swdown}, let $m=n-e$.
Then
\begin{align*}
\{p\chi, 1+\pi^{m+1}T\} &= \{\chi, (1+\pi^{m+1}T)^{p}\}\\
&= \{\chi, 1+\pi^{n+1}T'\}
\end{align*}
where $T'=\frac{(1+\pi^{m+1}T)^{p}-1}{\pi^{n+1}}=a_1T+a_2T^2+\dots + a_p T^p$ with $a_1,\dots,a_p\in \O_K$. Therefore, $T'\in\O_K[T]$. Since $\sw(\chi)=n$, we have $\{\chi, 1+\pi^{n+1}T'\}=0$ in $V^{q+1}(\O_K[T])$, by the functoriality of $V^{q+1}(\cdot )$. Therefore, $\sw(p\chi)\leq m=n-e$, which completes the proof of~\eqref{swdown}.

Now we move on to prove~\eqref{rswsame} and~\eqref{rswe'}. We are assuming that $n\in\Z$ satisfies $n\geq e'>e$, whereby $n-e\geq 1$. Similarly to the calculation above, we have
\begin{equation}\label{eq:rsw2}
 \{p\chi, 1+\pi^{n-e}T\}
= \{\chi, 1+\pi^{n}T''\}
 \end{equation}
 where $T''=\frac{(1+\pi^{n-e}T)^{p}-1}{\pi^{n}}=b_1T+b_2T^2+\dots + b_p T^p$ with $b_1=p/\pi^e\in \O_K^\times$ and $b_2,\dots,b_{p-1}\in\pi\O_K$. If $n>e'$ then $b_p\in\pi\O_K$; if $n=e'$ then $b_p=1$. 
Now~\cite[6.3.1]{K} gives
\begin{equation}\label{eq:lambda}
\{ \chi, 1+\pi^n T'' \} =
\lambda_\pi(\alpha \bar{b}_1 T, \beta \bar{b}_1 T ) +
\lambda_\pi(\alpha \bar{b}_p T^p, \beta \bar{b}_p T^p).
\end{equation}

If $n>e'$ then $\bar{b}_p=0$. Thus,~\eqref{rswsame} follows from~\eqref{eq:rsw2} and~\eqref{eq:lambda}.
Finally, suppose $n=e'$ and therefore $b_p=1$. Since $n=ep/(p-1)\in\Z$, we have $p\mid n$, whereby Lemma~\ref{lem:dbna} yields $d\alpha=d\beta=0$. Now Lemma~\ref{lem:cupT}\eqref{Ca} shows that $\lambda_\pi(\alpha T^p,\beta T^p)=\lambda_\pi(C(\alpha)T, C(\beta)T)$,
giving \eqref{rswe'}.
\end{proof}

\begin{remark}\label{remark:ea0} In the context of Lemma~\ref{lem:rsw-same}, suppose $\beta\wedge d\bar{u}=0$ as happens, for example, if $d\bar{u}=0$.
One perhaps surprising consequence of Lemma~\ref{lem:rsw-same} in this setting is that for $\chi\in \H^q(K)$ with $\sw(\chi)=n>e'$ and $\rsw_{n}(\chi)=[\alpha,\beta]_{\pi, n}$, we have $e\alpha=0$, where $e$ is the absolute ramification index of $K$. This can be seen by applying Lemma~\ref{lem:dbna} twice: applying it to $\chi$ yields $d\beta=(-1)^q n\alpha$, and applying it to $p\chi$ yields $\bar{u}d\beta=(-1)^q(n-e)\bar{u}\alpha$.
For $\chi\in \H^q(K)$ with $\sw(\chi)=e'$, we obtain the more complicated formula $d(C(\beta))=(-1)^{q+1}e(\bar{u}\alpha+C(\alpha))$, which reduces to $e\alpha=0$ if $\alpha$ and $\beta$ are exact.   
\end{remark}

\subsection{The residue map $\partial$}\label{sec:residue}

Let $n \ge 1$.
By Proposition~\ref{prop:kato61}, the image of
\[
\lambda_\pi \colon \H^q_{n}(F) \oplus \H^{q-1}_{n}(F) \to \H^q_{n}(K) = \Br K[n]
\]
coincides with $\fil_0 \H^q_{n}(K)$.

\begin{definition}\label{def:residue}
Define a homomorphism 
\[
\partial \colon \fil_0 \H^q_{n}(K) \to \H^{q-1}_{n}(F)
\]
to be the inverse of $\lambda_\pi$ followed by projection onto the second factor, c.f.~\cite[\S 7.5]{K}. Its inductive limit is a homomorphism $\fil_0 \H^q(K) \to \H^{q-1}(F)$, which we also denote by $\partial$. We will refer to $\partial$ as the \emph{residue map}.
\end{definition} 

\subsection{Comparison with the classical residue map on the Brauer group}

For a Henselian discrete valuation field $K$ of characteristic zero with perfect residue field, there is a standard definition of a residue map on $\Br K$, as in for example~\cite[\S XII.3]{Serre},~\cite[\S 1.1]{CTSD}, or~\cite[\S1.4.3]{CTSbook}, where it is called the Witt residue. In our setting, this definition carries over unchanged to define a residue map
\[
\partial' \colon \Br(K_{nr}/K) \to \H^1(F ,\Q/\Z).
\]
We will now recall this definition and verify that it is compatible with ours. For the rest of this section, assume $\Char K=0$.  First note that $\Br(K_{nr}/K)=\fil_0 \H^2(K)$, as follows from Proposition~\ref{prop:kato61}(1) and~\cite[p.~35]{CTSbook}. 

Let $\delta \colon \H^1(F,\Q/\Z) \to \H^2(F,\Z)$ be the connecting map coming from the short exact sequence $0 \to \Z \to \Q \to \Q/\Z \to 0$ of Galois modules.
It is an isomorphism.
Let $\partial'$ be the composite map
\[
\H^2(K_{nr}/K, K_{nr}^\times) \xrightarrow{v} \H^2(K_{nr}/K, \Z) \cong \H^2(F,\Z) \xleftarrow{\delta} \H^1(F,\Q/\Z),
\]
where $v \colon K_{nr}^\times \to \Z$ is the valuation.
Let $A$ be the ring of integers in $K$, and let $\iota'$ be the composite of the natural maps
\[
\Br F \xleftarrow{\sim} \Br A \to \Br(K_{nr}/K).
\]
By the same argument as~\cite[\S XII.3, Theorem~2]{Serre} and the remark following it, the sequence
\begin{equation}\label{eq:residue}
0 \to \Br F \xrightarrow{\iota'} \Br(K_{nr}/K) \xrightarrow{\partial'} \H^1(F,\Q/\Z) \to 0
\end{equation}
is exact.

To state the following proposition, we make use of the exact triangle~\eqref{eq:kummer}
of complexes of sheaves on the \'etale site of any field, for any $n \ge 1$.
Recall also the ``canonical lifting'' map $\iota_n^2 \colon \H^2_n(F) \to \fil_0 \H^2_n(K)$, which is the first component of $\lambda_\pi$ (see Section~\ref{sec:lambda-dvr}).

\begin{proposition}\label{prop:residues}
For any integer $n \ge 1$, the following diagram commutes:
\begin{equation}\label{eq:residues}
\begin{CD}
\H^2_n(F) @>{\iota^2_n}>> \fil_0 \H^2_n(K) @>\partial>> \H^1_n(F) \\
@VVV @VVV @VVV \\
\Br F @>{\iota'}>> \Br(K_{nr}/K) @>{\partial'}>> \H^1(F,\Q/\Z)
\end{CD}.
\end{equation}
Here the two left-hand vertical maps come from the triangle~\eqref{eq:kummer}, and the right-hand one from the natural inclusion $\Z/n \to \Q/\Z$.
\end{proposition}

We first prove a lemma on cup products.

\begin{lemma}\label{lem:cup-boundary}
Let $L$ be a field, and let $n$ be a positive integer.  Let $u \colon L^\times \to \H^1(L,(\Z/n)(1))$ and $t \colon \H^2(L,(\Z/n)(1)) \to \Br L$ be the maps coming from the triangle~\eqref{eq:kummer}.
Let $\delta \colon \H^1(L,\Z/n) \to \H^2(L,\Z)$ be the connecting map coming from the short exact sequence $0 \to \Z \to \Z \to \Z/n \to 0$ of Galois modules.
For $\chi \in \H^1(L,\Z/n)$ and $a \in L^\times$, we have $\delta\chi \cup a = t(\chi \cup u(a))$.
\end{lemma}
Note that this definition of $\delta$ agrees with the previous one when $\H^1(L,\Z/n)$ is considered as a subgroup of $\H^1(L,\Q/\Z)$.
\begin{proof}
It suffices to prove the lemma separately for $n$ invertible in $L$, and for $n=p^r$ where $p>0$ is the characteristic of $L$ and $r \ge 1$.  For $n$ invertible in $L$, we have $(\Z/n)(1) = \mmu_n$, the triangle~\eqref{eq:kummer} is the Kummer sequence, and the lemma is proved in~\cite[proof of Proposition~4.7.1]{GS} .

For $n = p^r$, let $L^s$ be a separable closure of $L$.  The triangle~\eqref{eq:kummer} is the short exact sequence
\[
0 \to (L^s)^\times \xrightarrow{p^r} (L^s)^\times \xrightarrow{\dlog} W_r \Omega^1_{L^s,\log} \to 0
\]
of Galois modules, $u$ is the map $\dlog \colon L^\times \to W_r \Omega^1_{L,\log}$, and $t$ is the boundary map $\H^1(L, W_r \Omega^1_{L^s,\log}) \to \H^2(L,(L^s)^\times)$.
Note that the above sequence is isomorphic to that obtained by taking the short exact sequence $0 \to \Z \to \Z \to \Z/p^r \to 0$ and forming the tensor product with $(L^s)^\times$.  The result then follows from~\cite[Proposition~3.4.8]{GS}.
\end{proof}

\begin{proof}[Proof of Proposition~\ref{prop:residues}]
We first express $\iota'$ in terms of Galois cohomology.
The strict Henselisation $A^\textrm{sh}$ is the ring of integers in $K_{nr}$ and has residue field $F^s$, a separable closure of $F$.
The Hochschild--Serre spectral sequence, together with $\Pic(A^\textrm{sh}) = \Br(A^\textrm{sh}) = 0$ \cite[Corollary~IV.1.7]{M}, gives an isomorphism 
\[
\H^2(K_{nr}/K, (A^\textrm{sh})^\times) \cong \Br A,
\]
compatible with the usual isomorphisms
\[
\H^2(F,(F^s)^\times) \cong \Br F \quad \text{and} \quad \H^2(K_{nr}/K, K_{nr}^\times) \cong \Br(K_{nr}/K).
\]
So $\iota'$ is identified with the composite
\[
\H^2(F,(F^s)^\times) \xleftarrow{\sim} \H^2(K_{nr}/K, (A^\textrm{sh})^\times) \to \H^2(K_{nr}/K, K_{nr}^\times).
\]

Both rows of the diagram~\eqref{eq:residues} are split exact sequences:
the map $\chi \mapsto \{ \iota^1_n(\chi), \pi \}$ is (by definition) a section of $\partial$; and
the map $\chi \mapsto \delta\chi \cup \pi$ is a section of $\partial'$.
(Here we identify the absolute Galois group of $F$ with $\Gal(K_{nr}/K)$.)
It is therefore enough to show that the following diagram commutes:
\[
\begin{CD}
\H^2_n(F) @>{\iota^2_n}>> \fil_0 \H^2_n(K) @<{\chi \mapsto \{ \iota^1_n(\chi), \pi \}}<< \H^1_n(F) \\
@VVV @VVV @VVV \\
\H^2(F,(F^s)^\times) @>{\iota'}>> \H^2(K_{nr}/K, K_{nr}^\times) @<{\chi \mapsto \delta\chi \cup \pi}<< \H^1(F,\Q/\Z)
\end{CD}.
\]
That the right-hand square commutes follows from Lemma~\ref{lem:cup-boundary} applied to $K$.  (Note that $\iota^1_n$ is simply the identification of Galois groups just mentioned)

For the left-hand square, $\iota^2_n$ is defined separately in Section~\ref{sec:lambda-dvr} for $n$ invertible in $F$, and for $n=p^r$.
If $n$ is invertible in $F$, then the commutativity follows immediately from the definition and the Kummer sequence on $A$.
For $n=p^r$,
it suffices to prove it for elements $\{ \chi, \bar{a} \}$ where $\chi \in \H^1(F,\Z/n)$ and $\bar{a} \in F^\times$.  By definition, we have $\{ \chi, \bar{a} \} = \chi \cup u(\bar{a})$, so Lemma~\ref{lem:cup-boundary} shows that the image of this element in $\H^2(F,(F^s)^\times)$ is equal to $\delta\chi \cup \bar{a}$; applying $\iota'$ gives $\delta\chi \cup a$, where $a \in A^\times$ is a lift of $\bar{a}$ and we have as before identified $\Gal(F^s/F)$ with $\Gal(K_{nr}/K)$.  On the other hand, first applying $\iota^2_n$ gives $\{ \iota^1_n(\chi), a \} = \chi \cup u(a)$ and Lemma~\ref{lem:cup-boundary} again shows that the image in $\H^2(K_{nr}/K, K_{nr}^\times)$ is $\delta\chi \cup a$, as desired.
\end{proof}

\section{The tame part}\label{sec:tame}

We return to the situation of the introduction.  Let $k$ be a finite extension of $\Q_p$ with ring of integers $\O_k$, uniformiser $\pi$ and residue field $\F$.  Let $X/k$ be a smooth, geometrically irreducible variety over $k$, and let $\X$ be a smooth $\O_k$-model of $X$ having geometrically irreducible special fibre $Y$.
Denote by $K$ the function field of $X$ and by $F$ the function field of $Y$.
Let $\Kh$ be the field of fractions of a Henselisation of the discrete valuation ring $\O_{\X,Y}$.

The natural map $\Br X \to \Br \Kh$ allows us to pull back Kato's definition of the Swan conductor, and the associated filtration, to $\Br X$.
In this section we look at the smallest piece $\fil_0 \Br X$ of Kato's filtration on $\Br X$.
By Proposition~\ref{prop:kato61} and~\cite[Corollaire~1.3]{GIII}, this is the same as the subgroup of $\Br X$ consisting of those elements whose image in $\Br \Kh$ is split by an unramified extension of $\Kh$.
Equivalently, such an element is split by a finite extension $L/K$, where $L$ is the field of fractions of a discrete valuation ring \'etale over $\O_K = \O_{\X,Y}$.
To see this equivalence, note that the maximal unramified extension $\Kh_{nr}$ of $\Kh$ is the field of fractions of a strict Henselisation of $\O_K$, and therefore is the colimit of all such extensions $L/K$. Recall the residue map $\partial$ defined in Definition~\ref{def:residue}. We will denote the composition $\fil_0 \Br X \to \fil_0 \H^2(\Kh) \xrightarrow{\partial} \H^1(F)$ also by $\partial$.
Recall that $\partial\colon \Br k \to \H^1(\F,\Q/\Z)$ is an isomorphism, by a standard calculation of local class field theory.
The main result of this section is the following.

\begin{proposition}\label{prop:tame}
\begin{enumerate}
\item\label{residue} If $\A\in\fil_0 \Br X$, then $\partial(\A)\in\H^1(Y,\Q/\Z) \subset \H^1(F,\Q/\Z)$.
\item\label{tame} Let $P \in \X(\O_k)$ reduce to a point $P_0 \in Y(\F)$.  Then the following diagram commutes:
\[
\begin{CD}
\fil_0 \Br X @>{\partial}>> \H^1(Y,\Q/\Z) \\
@V{P^*}VV @VV{P_0^*}V \\
\Br k @>{\partial}>{\cong}> \H^1(\F,\Q/\Z)
\end{CD} .
\]
\end{enumerate}
\end{proposition}

The following corollary is immediate.

\begin{corollary}\label{cor:tame}
\begin{enumerate}[(i)]
\item For $\A \in \fil_0 \Br X$, the evaluation map $\evmap{\A} \colon \X(\O_k) \to \Br k$ depends only on $\partial(\A)$.
\item \label{constantmodp} For $\A \in \fil_0 \Br X$ and $P \in X(k)$ reducing to a smooth point $P_0 \in Y(\F)$, the evaluation $\A(P)$ depends only on $P_0$.
\end{enumerate}
\end{corollary}

Proposition~\ref{prop:tame} will be used in the proof of Theorem~\ref{thm:onelemma_intro}, in combination with the following lemma. We use $\frac{1}{p}$ to denote the map $ \H^1(\F,\Z/p)\to\H^1(\F,\Q/\Z) $ induced by identifying $\Z/p$ with the $p$-torsion in $\Q/\Z$.

\begin{lemma}\label{lem:inv_trace}
Let $\delta_1 \colon \F \to \H^1(\F,\Z/p)$ be the Artin--Schreier map. Suppose that $A \in \Br k[p]$ satisfies $\partial(A) =\frac{1}{p} \delta_1(x)$, with $x \in \F$.
Then $\inv(A) = \frac{1}{p}\Tr_{\F/\F_p}(x)$.
\end{lemma}
\begin{proof}
Let $q=p^r$ be the cardinality of $\F$.
By definition, $\inv(A)$ is obtained by evaluating $\partial(A) \in \H^1(\F,\Q/\Z) = \Hom(\Gal(\bar{\F}/\F), \Q/\Z)$ at the $q$-power Frobenius element $\Frob \in \Gal(\bar{\F}/\F)$.
On the other hand, the $1$-cocycle $\delta_1(x)$ is defined as follows: let $y \in \bar{\F}$ be such that $y^p-y=x$; then, for $\sigma \in \Gal(\bar{\F}/\F)$, we define $\delta_1(x)(\sigma) = \sigma(y)-y \in \Z/p$.
Combining these definitions gives
\[
\inv(A) = \frac{1}{p}( \Frob(y)-y)
=\frac{1}{p}\Bigl(  (y^{p^r} - y^{p^{r-1}}) + \dotsb + (y^p-y)\Bigr)
= \frac{1}{p}\Tr_{\F/\F_p}(x).
\qedhere
\]
\end{proof}

In Section~\ref{sec:tame_p} we prove Proposition~\ref{prop:tame} for $\A\in \Br X[p^r]$. The result for Brauer group elements of order prime to $p$ follows from comparison with the classical residue map (see Proposition~\ref{prop:residues}) and well-known properties of the latter: see for example~\cite[Proposition~5.1]{bad}, together with~\cite[Theorem~1.4.14 and Theorem~2.3.5]{CTSbook} for comparing the various different residue maps.

\subsection{Evaluation of tame elements of $p$-power order}\label{sec:tame_p}

We first prove a lemma.

\begin{lemma}\label{lem:tauinj} Let $i\colon Y\to \X$ be the inclusion of the special fibre and let $j\colon X\to\X$ be the inclusion of the generic fibre. Let $\tau \colon \Spec F \to Y$ be the inclusion of the generic point.  Let $q,r \ge 1$.  Then the map
\[
i^* \R^q j_* (\Z/p^r)(q-1) \to \tau_* \tau^* i^* \R^q j_* (\Z/p^r)(q-1)
\]
of sheaves on $Y$ is injective.
\end{lemma}
\begin{proof}
We use induction on $r$.  For the case $r=1$, it suffices to prove the statement after adjoining a $p$th root of unity to the base field $k$, and then this is~\cite[Proposition~6.1(i)]{BK}.

For any $q,m$, the sheaf $\tau^* i^* \R^q j_* (\Z/p^r)(m)$ on $F_\et$ is the sheaf corresponding to the $\Gal(F^s/F)$-module $\H^q(\Kh_{nr}, (\Z/p^r)(m))$.
Consider the long exact sequence in cohomology on $\Kh_{nr}$ coming from the short exact sequence
\begin{equation}\label{eq:pr}
0 \to \Z/p^{r-1} \to \Z/p^r \to \Z/p \to 0
\end{equation}
of Galois modules.
We have a commutative diagram
\[
\begin{CD}
\K_{q-1}(\Kh_{nr}) @= \K_{q-1}(\Kh_{nr}) \\
@VVV @VVV \\
\H^{q-1}(\Kh_{nr}, (\Z/p^r)(q-1)) @>>> \H^{q-1}(\Kh_{nr}, (\Z/p)(q-1))
\end{CD}
\]
in which the vertical maps are the Galois symbols, which are surjective by~\cite[\S 5]{BK}; this shows that the bottom map is surjective.
It follows that the long exact sequence of cohomology of~\eqref{eq:pr} gives
\[
0 \to \H^q(\Kh_{nr}, (\Z/p^{r-1})(q-1)) \to \H^q(\Kh_{nr}, (\Z/p^r)(q-1)) \to \H^q(\Kh_{nr}, (\Z/p)(q-1)).
\]
Consider this as a sequence of sheaves on $F_\et$.
Applying $\tau_*$ gives the bottom row of the following commutative diagram of sheaves on $Y$.
\[
{\small
\begin{tikzcd}[column sep = tiny]
& i^* \R^q j_* (\Z/p^{r-1})(q-1) \arrow[r] \arrow[d]& i^* \R^q j_* (\Z/p^r)(q-1) \arrow[r]\arrow[d] & i^* \R^q j_* (\Z/p)(q-1)\arrow[d] \\
0 \arrow[r]&  \tau_* \tau^* i^* \R^q j_* (\Z/p^{r-1})(q-1) \arrow[r] & \tau_* \tau^* i^* \R^q j_* (\Z/p^r)(q-1) \arrow[r] & \tau_* \tau^* i^* \R^q j_* (\Z/p)(q-1)
\end{tikzcd}
}
\]
By induction, the two outer vertical maps are injective, and therefore the middle one is as well.
\end{proof}

To prove Proposition~\ref{prop:tame}(\ref{residue}),
we will prove a result for general $q$, in the case that $\X$ is affine.

\begin{lemma}\label{lem:integral}
Suppose that $\X = \Spec A$ is affine, and define $R = A / \m_k A$.
Let $r \ge 1$ and $q \ge 2$.
Let $\chi$ be an element of $\fil_0 \H^q_{p^r}(\Kh)$, whereby $\chi = \lambda_\pi(\alpha,\beta)$ for a unique $(\alpha,\beta) \in \H^q_{p^r}(F) \oplus \H^{q-1}_{p^r}(F)$.
If $\chi$ lies in the image of $V^q_{p^r}(A)$, then $(\alpha,\beta)$ lies in the image of $\H^q_{p^r}(R) \oplus \H^{q-1}_{p^r}(R)$.
\end{lemma}
\begin{proof}
Let $\lambda_\pi \colon \H^q_{p^r}(R) \oplus \H^{q-1}_{p^r}(R) \to V^q_{p^r}(A)$ be the map defined in Definition~\ref{def:lambda}.
The sequences~\eqref{eq:infres1} give a commutative diagram as follows.
\begin{equation}\label{eq:Psi}
\begin{tikzcd}
0 \arrow[r]& \H^q_{p^r}(R) \oplus \H^{q-1}_{p^r}(R)\ar[d, "a"] \arrow[r, "{\lambda_\pi}"]& V^q_{p^r}(A)\ar[d, "b"] \arrow[r]& \H^0(R, i^* \R^q j_* (\Z/{p^r})(q-1))\ar[d, "c"] \\
0 \arrow[r]& \H^q_{p^r}(F) \oplus \H^{q-1}_{p^r}(F) \arrow[r, "{\lambda_\pi}"]&  \H^q_{p^r}(\Kh) \arrow[r, "{\res}"]& \H^0(F, \H^q_{p^r}(\Kh_{nr})) 
\end{tikzcd}
\end{equation}
By assumption $\chi = \lambda_\pi(\alpha,\beta)$ lies in the image of $b$.
To show that $(\alpha,\beta)$ lies in the image of $a$, it is enough to prove that $c$ is injective; but this follows from Lemma~\ref{lem:tauinj}.
\end{proof}

\begin{proof}[Proof of Proposition~\ref{prop:tame}]
We first prove part (\ref{residue}). 
Let $\A$ lie in $\fil_0 \Br X[p^r]$; then $\partial(\A)$ lies in $\H^1(F, \Z/p^r)$ and we must show that it actually lies in $\H^1(Y, \Z/p^r)$.
By~\cite[Corollaire~I.10.3]{SGA1}, this subgroup consists of all classes in $\H^1(F, \Z/p^r)$ such that the corresponding torsor is unramified on $Y$;
this condition may be checked on an affine cover of $Y$.
Let $\Spec A$ be any affine open subset of $\X$ that meets $Y$.
Lift $\A$ (using the Kummer sequence) to $\H^2(X,(\Z/p^r)(1))$.
Looking at~\eqref{eq:HtoV} for the morphisms $\Spec\O_{\X,Y} \to \Spec A \to \X$ shows that
the map $\H^2(X,(\Z/p^r)(1)) \to \H^2(\Kh, (\Z/p^r)(1)) = V^2_{p^r}(\O_{\X,Y})$  factors through $V^2_{p^r}(A)$,
and so Lemma~\ref{lem:integral} shows that $\partial(\A)$ lies in $\H^1(\Spec(A/\m_k A), \Z/p^r)$.
The affine schemes $\Spec(A/\m_k A)$ arising in this way cover $Y$, proving the statement.

Part (\ref{tame}) now follows easily from Lemma~\ref{lem:lambda-basechange}.
\end{proof}

We conclude this section with an alternative description of the kernel of $\partial$. 

The natural map $\Br K \to \Br \Kh$ allows us to extend the definition of the classical residue map $\partial'$ to $\Br(\Kh_{nr}/K) = \fil_0 \Br K$. The following lemma is a generalisation of a result of~\cite[\S 1.1]{CTSD} to the case of imperfect residue field.
\begin{lemma}\label{lem:unramd}
The kernel of $\partial' \colon \fil_0 \Br K \to \H^1(F,\Q/\Z)$ coincides with the image of $\Br \O_{\X,Y} \to \Br K$.
\end{lemma}
\begin{proof}
Let $i \colon \Spec F \to \Spec \O_{\X,Y}$ and $j \colon \Spec K \to \Spec \O_{\X,Y}$ be the inclusions of the special and generic points, respectively.
As in~\cite[\S 2]{GIII}, where the case of perfect residue field is treated, the short exact sequence
\[
0 \to \Gm \to j_* \Gm \to i_* \Z \to 0
\]
of sheaves on $\Spec \O_{\X,Y}$ gives rise to an exact sequence
\[
0 \to \Br \O_{\X,Y} \to \H^2(\O_{\X,Y}, j_* \Gm) \to \H^2(F,\Z).
\]
The Leray spectral sequence shows that $\H^2(\O_{\X,Y}, j_* \Gm)$ is the kernel of the natural map $\Br K \to \Br \Kh_{nr}$.
Applying the same construction to the Henselisation $A = \O_{\X,Y}^h$ gives a commutative diagram with exact rows
\[
\begin{CD}
0 @>>> \Br \O_{\X,Y} @>>> \Br(\Kh_{nr}/K) @>>> \H^2(F,\Z) \\
@. @VVV @VVV @| \\
0 @>>> \Br A @>>> \Br(\Kh_{nr}/\Kh) @>>> \H^2(F,\Z)
\end{CD}.
\]
If $\alpha \in \fil_0 \Br K$ satisfies $\partial'(\alpha)=0$, then the exact sequence~\eqref{eq:residue} shows that the image of $\alpha$ in $\Br\Kh$ lies in the image of $\iota'$, which is the image of $\Br A$.  From the above diagram it then follows that $\alpha$ lies in the image of $\Br \O_{\X,Y}$.
\end{proof}

\begin{corollary}\label{cor:brix}
The kernel of $\partial \colon \fil_0 \Br X \to \H^1(Y,\Q/\Z)$ coincides with the image of the natural map $\Br\X \to \Br X$.
\end{corollary}
\begin{proof}
The purity theorem~\cite[Theorem~1.2]{Kestutis} shows $\Br \X = \Br X \cap \Br \O_{\X,Y}$, with the intersection taking place inside $\Br K$.  (This particular case of the purity theorem was proved by Gabber~\cite[Theorem~2.2]{Gabber-inj}.)
By Lemma~\ref{lem:unramd}, this consists of those elements of $\fil_0 \Br X$ lying in the kernel of $\partial'$, and by Proposition~\ref{prop:residues} this coincides with the kernel of $\partial$. 
\end{proof}

\section{Comparisons of some (refined) Swan conductors}\label{sec:compare-rsw}


In this section we begin working towards the proof of Theorem~\ref{thm:onelemma_intro}.
Our main tool will be to blow up the model in a smooth point; this results in a new component of the special fibre, corresponding to a new discrete valuation and so to a new Swan conductor and refined Swan conductor.
This and the following sections are devoted to studying the effect of blowing up on the Swan conductor and refined Swan conductor; this study will lead in Section~\ref{sec:endgames} to an inductive proof of Theorem~\ref{thm:onelemma_intro}.
We begin by recalling some more notions from Kato's paper~\cite{K}.

\subsection{Unramified elements}\label{sec:unram}
Let $X$ be a normal irreducible scheme with function field $K$. For $x \in X$, let $K_x$ be the field of fractions of the Henselisation $\O_{X,x}^h$ of $\O_{X,x}$ and let $\kappa(x)$ denote the residue field at $x$. Following \cite[\S 1.5]{K}, we say that an element $\chi$ of $\H^q(K)$ is unramified on $X$ if for any $x \in X$, the image of $\chi$ in $\H^q(K_x)$ belongs to the image of the canonical lifting map $\iota^q(\O_{X,x}^h): \H^q(\kappa(x)) \to \H^q(K_x)$, which is the first component of $\lambda_\pi$ (see Section~\ref{sec:iota-local}). In the case $q=1$, $\chi$ is unramified on $X$ in this sense if and only if $\chi$ belongs to $\H^1(X,\Q/\Z) \subset \H^1 (K)$. 

\subsection{(Refined) Swan conductors in a geometric setting}\label{sec:geomswan}
As in Subsection~\ref{sec:unram}, let $X$ be a normal irreducible scheme with function field $K$ and let $\chi\in \H^q(K)$. For $\p \in X^1=\{x\in X \mid \dim \O_{X,x}=1\}$, the field $K_\p=\Frac(\O_{X,\p}^h)$ is a Henselian discrete valuation field. Let $\chi_\p$ denote the image of $\chi$ in $\H^q(K_\p)$. Following~\cite[Part~II]{K}, we denote by $\sw_\p(\chi)$ the Swan conductor of $\chi_\p$. Now let $n\geq 1$ and suppose that $\chi_\p$ lies in $\fil_n\H^q(K_\p)$ (i.e.\ $\sw_{\p}(\chi)\leq n$). We denote by $\rsw_{\p,n}(\chi)$ the refined Swan conductor $\rsw_n(\chi_\p)$. For an irreducible subset $Z\subset X$ of codimension one with generic point $z$, we let $\sw_Z (\chi)=\sw_z (\chi)$ and $\rsw_{Z,n}(\chi)=\rsw_{z,n}(\chi)$ for $n\geq 1$.

\subsection{Comparisons}
Blowing up a smooth model leads to a new model having two components in its special fibre.
Along each of these components, a given element of the Brauer group of the generic fibre has a Swan conductor and a refined Swan conductor.
In order to understand how these relate to each other, we put ourselves in the following more general setting and follow~\cite[Section~7]{K}.

For the rest of this section, let $A$ be an excellent regular local ring with field of fractions $K$ and residue field $\ell$ of characteristic $p>0$ such that $[\ell:\ell^p]=p^c$. Let $(\pi_i)_{1\leq i\leq r}$ be part of a regular system of parameters of $A$, let $\p_i=\pi_i A\in \Spec A$ and let $\overline{\{\p_i\}}$ denote the closure of $\{\p_i\}$ in $\Spec A$. Let $R_i=A/\p_i$ and let $\kappa(\p_i)=\Frac(R_i)$ denote the residue field at $\p_i$. For $j\in\{1, \dotsc, r\}$, let $D_j$ denote the divisor $\sum_{i\neq j}(\overline{\{\p_i\}}\cap \overline{\{\p_j\}})$ on $\overline{\{\p_j\}}=\Spec R_j$. We use $\Omega_{R_j}
^\cdot(\log D_j)$ to denote the $\F_p$-subalgebra of $\Omega_{\kappa(\p_j)}^\cdot$ generated by $\Omega_{R_j}^\cdot$ and elements of the form $\dlog f$ such that $f\in R_j$ and $\Supp(R_j/f)\subset D_j$. Let $q\geq 0$ and let $\chi\in \H^q(K)$ be unramified on $\Spec A\setminus \bigcup_{i=1}^r \overline{\{\p_i\}}$. Let $n_i=\sw_{\p_i}\chi$ for $1\leq i\leq r$. 

We now state some results of Kato that will be used here and in the proof of Theorem~\ref{thm:onelemma_intro} in Section~\ref{sec:endgames}.
The first is an integrality statement.  As defined, the refined Swan conductor at $\p_j$ is given by
\begin{equation}\label{eq:rsw_j}
\rsw_{\p_j}(\chi) = \pi_j^{-n_j}(\alpha + \beta \wedge \dlog \pi_j)
\end{equation}
with $\alpha \in \Omega^q_{\kappa(\p_j)}$ and $\beta \in \Omega^{q-1}_{\kappa(\p_j)}$.  The following theorem states that $\alpha,\beta$ are in fact integral outside $D_j$ and describes their possible poles along $D_j$. (Note that the $\alpha, \beta$ occurring in Theorem~\ref{thm:7.1} coincide with those in \eqref{eq:rsw_j} when $r=1$, and in other cases they are rescaled by $\prod_{i\neq j}\pi_i^{n_i}$.)

\begin{theorem}[Kato, {\cite[Theorem~7.1]{K}}]\label{thm:7.1}
Let $j \in \{1, \dotsc, r\}$ and assume $n_j \geq 1$. Write
\[\rsw_{\p_j}(\chi)=\Big(\prod_{i=1}^r\pi_i^{-n_i}\Big)\cdot (\alpha+\beta\wedge\dlog\pi_j)\]
with $\alpha\in \Omega^q_{\kappa(\p_j)}$ and $\beta\in\Omega^{q-1}_{\kappa(\p_j)}$. Then we have
\begin{align*}
\alpha\in \Omega^q_{R_j}(\log D_j), \qquad \beta\in \Omega^{q-1}_{R_j}(\log D_j).
\end{align*}
\end{theorem}

The following proposition describes how the refined Swan conductors associated to the different $\p_j$ are related to each other.
Retain the notation of Theorem~\ref{thm:7.1}. For a subset $s\subset\{1,\dotsc, r\}$, let $|s|=\card(s)$ and let $s(1),\dots , s(|s|)$ denote the elements of $s$ ordered so that $s(1)<\dots <s(|s|)$. Theorem~\ref{thm:7.1} allows one to write
\[\rsw_{\p_j}(\chi)=\Big(\prod_{i=1}^r\pi_i^{-n_i}\Big)\cdot \sum_{s}\omega_s(j)\wedge\dlog \pi_{s(1)}\wedge\dots \wedge \dlog\pi_{s(|s|)}\]
with $\omega_s(j)\in\Omega^{q-|s|}_{R_j}$, where $s$ ranges over all subsets of $\{1,\dotsc, r\}$. 

\begin{proposition}[Kato, {\cite[Proposition~7.3]{K}}]\label{prop:7.3}
Let $R=A/(\p_1+\dotsc +\p_r)$ and let $j_1,j_2\in\{1, \dotsc , r\}$ be such that $n_{j_1}\geq 1$ and $n_{j_2}\geq 1$. Then
for each $s\subset\{1,\dotsc, r\}$ the images of $\omega_s(j_1)$ and $\omega_s(j_2)$ in $\Omega^{q-|s|}_{R}$ coincide.
\end{proposition}

We now state a definition and a theorem concerning blowups.

\begin{definition}[Kato, {\cite[Definition~7.4]{K}}]\label{def:7.4} Recall that $\ell$ is the residue field of $A$.
We say that $\chi$ is strongly clean with respect to $A$ if for any
$j$ such that $n_j \geq 1$, the image of $\omega_s(j)$ in $\Omega^{q-|s|}_\ell$ under $\Omega^{q-|s|}_{R_j}\to \Omega^{q-|s|}_\ell$ is not
zero for some $s$.
\end{definition}

\begin{theorem}[Kato, {\cite[Theorem~8.1]{K}}]\label{thm:8.1}
Let $f:X\to\Spec A$ be the blowup at the closed point of $\Spec A$, and let $\nu\in X$ be the generic point of the exceptional divisor. Then 
\[\sw_\nu (f^*\chi)\leq \sum_{i=1}^r n_i\]
with equality if and only if $\chi$ is strongly clean with respect to $A$.
\end{theorem}
In fact, Theorem~\ref{thm:8.1} is only the first part of Kato's statement, but it will suffice for our purposes. 

An important case that is missing from Theorem~\ref{thm:7.1} and Proposition~\ref{prop:7.3} is the case $n_j=0$.
In that case, there is no refined Swan conductor at $\p_j$.
However, in its place we have the isomorphism $\H^q(\kappa(\p_j)) \oplus \H^{q-1}(\kappa(\p_j)) \to \fil_0\H^q(K_{\p_j})$ given by $\lambda_{\pi_j}$ (see Proposition~\ref{prop:kato61}(1)).
If we write $\chi = \lambda_{\pi_j}(\chi_1,\chi_2)$, then we can ask about the Swan conductors and refined Swan conductors of $\chi_1,\chi_2$ at the prime ideals of $R_j$ corresponding to the $\p_i$.
The following lemma, which we see as an analogue of Theorem~\ref{thm:7.1} in the case $n_j=0$, deals with the Swan conductors.

\begin{lemma}\label{lem:new7.1}
Let $j \in \{1, \dotsc, r\}$ and assume $n_j = 0$, whereby there exist unique $\chi_1 \in \H^q(\kappa(\p_j))$ and $\chi_2 \in \H^{q-1}(\kappa(\p_j))$
such that $\chi = \lambda_{\pi_j}(\chi_1, \chi_2)$ in $\H^q(K_{\p_j})$.
Denote by $\bar{\p}_i$ the image of $\p_i$ in $R_j=A/\p_j$.
Then, for $i \neq j$, we have $\sw_{\bar{\p}_i}(\chi_1) \le n_i$ and $\sw_{\bar{\p}_i}(\chi_2) \le n_i$. \end{lemma}
\begin{proof}
As in Kato's proof of~Theorem~\ref{thm:7.1}, we reduce to the following situation: $\dim(A)=2$; $A$ is complete; $\H_p^{c+1}(\ell) \neq 0$ where $p^c=[\ell:\ell^p]$; and the order of $\chi$ is a power of $p$. The reduction to $\dim(A)=2$ comes from replacing $A$ by $A_\q$, where $\q$ corresponds to a height one prime ideal of $R_j$. By~\cite[Lemma~7.7]{K}, there exists a field extension $\ell'/\ell$ such that $\H^{c+1}_p(\ell')\neq 0$ (where $c$ is defined by $p^c=[\ell':\ell'^p]$) and a $p$-basis of $\ell$ remains a $p$-basis of $\ell'$, implying in particular that $\ell'/\ell$ is separable. Applying~\cite[Lemma 1]{K4} to $A/\m^s$ for $s\geq 1$ and taking the inverse limit shows the existence of a complete two-dimensional regular local ring $A'$ which is flat over $A$ and has residue field $A' \otimes \ell= \ell'$. Now the stability of the Swan conductor under well-behaved extensions of Henselian discrete valuation fields described in~\cite[Lemma~6.2 and Proposition~6.3]{K} shows we can replace $A$ by $A'$. By Proposition~\ref{prop:kato61}(1) we may replace $\chi$ by its $p$-primary component and so assume that the order of $\chi$ is a power of $p$. 

Our assumption that $\dim(A)=2$ implies $r\leq 2$. The statement for $r=1$ is empty so assume $r=2$, $\sw_{\p_1}(\chi)=0$, $\sw_{\p_2}(\chi)=n \ge 0$.
 The discrete valuation corresponding to $\bar{\p}_2$ on $K_{\p_1}$ has residue field $\ell$ with $[\ell:\ell^p]=p^c$, so by Proposition~\ref{prop:kato61}(2) the result is automatically true for $q>c+2$; we therefore assume $q \le c+2$.
 Kato~\cite[\S 7.5]{K} defines a complex that in our situation becomes
\begin{align}\label{eq:complex}
\H^{c+3}(K)\xrightarrow{(\partial_\q)_\q} \bigoplus_{\q\in\Spec( A)^1} \H^{c+2}(\kappa(\q))\xrightarrow{(\partial_\q')_\q} \H^{c+1}(\ell).
\end{align}
 Note that $\fil_0 \H^{c+3}(K_\q)= \H^{c+3}(K_\q)$; this follows from~\cite[Corollary~2.5]{K} in the case $\Char\kappa(\q)=0$ and from Proposition~\ref{prop:kato61}(2) if $ \Char \kappa(\q)=p$. Thus the residue map of Definition~\ref{def:residue} is defined on the whole of $\H^{c+3}(K_\q)$. Composing it with the natural map $\H^{c+3}(K)\to \H^{c+3}(K_\q)$ gives the map $\partial_\q$ in~\eqref{eq:complex}. The map $\partial_\q'$ is defined similarly.

Consider the element
\[
t = \{ \chi, 1 + a \pi_2^{n+1}, b_1, \dotsc, b_{c+2-q} \} \in \H^{c+3}(K),
\]
for arbitrary $a \in A$ and $b_1, \dotsc, b_{c+2-q} \in A[1/\pi_1\pi_2]^\times$. Since $\sw_{\p_2}(\chi)=n$, Definition~\ref{def:Swan} shows that $t$
becomes $0$ in $\H^{c+3}(K_{\p_2})$. Now let $\q\in\Spec( A)^1\setminus\{ \p_1,\p_2\}$. Since $t$ is unramified on $\Spec A \setminus ( \overline{\{\p_1\}} \cup  \overline{\{\p_2\}})$, we have $\partial_{\q}(t)=0$ by \cite[\S 7.5]{K}. Therefore, $\partial_{\q}(t)=0$ for all height one prime ideals $\q \neq \p_1$. Now Kato's complex~\eqref{eq:complex} gives $\partial'_{\p_1} \partial_{\p_1}(t)=0$. Moreover, the $p$-primary part of $\partial'_{\p_1} \colon \H^{c+2}(\kappa(\p_1)) \to \H^{c+1}(\ell)$ is an isomorphism by~\cite[Theorem~3 (3)]{KatoGC} and so it follows that $\partial_{\p_1}(t)=0$.

For $\p_1$, the definition of $\lambda_\pi$ in Section~\ref{sec:lambda-dvr} gives
\begin{align*}
t &= \{ \lambda_{\pi_1}(\chi_1, \chi_2), 1 + a \pi_2^{n+1}, b_1, \dotsc, b_{c+2-q} \} \\
&= \{ \iota^q(\chi_1), 1 + a \pi_2^{n+1}, b_1, \dotsc, b_{c+2-q} \}
+ \{ \iota^{q-1}(\chi_2), \pi_1, 1 + a \pi_2^{n+1}, b_1, \dotsc, b_{c+2-q} \}
\end{align*}
in $\H^{c+3}(K_{\p_1})$.  If the $b_i$ are all units in $A$ then the properties of $\iota^q$ (see Section~\ref{sec:iota-local}) give
\begin{align*}
t &=\iota^{c+3}\bigl(\{\chi_1,1 + \bar{a} \bar{\pi}_2^{n+1},  \bar{b}_1, \dotsc,  \bar{b}_{c+2-q}\}\bigr)\\
&\qquad+\{\iota^{c+2}\bigl((-1)^{c+1-q}\{\chi_2, 1 +  \bar{a} \bar{ \pi}_2^{n+1},  \bar{b}_1, \dotsc,  \bar{b}_{c+2-q}\}\bigr),\pi_1\}\\
&=\lambda_{\pi_1}\bigl(\{\chi_1,1 + \bar{a} \bar{\pi}_2^{n+1},  \bar{b}_1, \dotsc,  \bar{b}_{c+2-q}\},(-1)^{c+1-q} \{\chi_2, 1 +  \bar{a} \bar{ \pi}_2^{n+1},  \bar{b}_1, \dotsc,  \bar{b}_{c+2-q}\}\bigr),
\end{align*}
where the bars denote images in $A/\p_1$.
Now, recalling that the residue map $\partial_{\p_1}$ is given by the inverse of $\lambda_{\pi_1}$ followed by projection on the second factor, we obtain
\[
0=\partial_{\p_1}(t) = (-1)^{c+1-q} \{\chi_2, 1 + \bar{a} \bar{\pi}_2^{n+1}, \bar{b}_1, \dotsc, \bar{b}_{c+2-q} \}
\]
in $\H^{c+2}(\kappa(\p_1))$. In other words, $\chi_2$ annihilates $U^{n+1} \K_{c+3-q}(\kappa(\p_1))$, where the unit filtration is defined using the discrete valuation corresponding to $\bar{\p}_2$. 
By~\cite[Proposition~6.5]{K}, this means that $\sw_{\bar{\p}_2}(\chi_2) \le n$.

A similar argument with $b_1=\pi_1$ gives $\sw_{\bar{\p}_2}(\chi_1) \le n$.
\end{proof}

Our next result is an analogue of Proposition~\ref{prop:7.3} in the case $n_j=0$.  For simplicity we state it only in the case $r=2$.
Suppose $\sw_{\p_1}(\chi)=0$ and $\sw_{\p_2}(\chi)=n \geq 1$.
Write $\chi = \lambda_{\pi_1}(\chi_1, \chi_2)$ in $\H^q(K_{\p_1})$, with $\chi_1 \in \H^q(\kappa(\p_1))$ and $\chi_2 \in \H^{q-1}(\kappa(\p_1))$. Applying~Theorem~\ref{thm:7.1} with $j=2$, we can write
\[\rsw_{\p_2,n}(\chi) =  \big[ \eta + \omega \wedge \dlog\bar{\pi}_1 , \eta' + \omega' \wedge \dlog\bar{\pi}_1 \big]_{\pi_2,n}\]
where $\eta\in\Omega^q_{R_2}; \omega,\eta'\in \Omega^{q-1}_{R_2}; \omega'\in  \Omega^{q-2}_{R_2}$ and $\bar{\pi}_1$ denotes the image of $\pi_1$ in $R_2=A/\p_2$.

\begin{lemma}\label{lem:new7.3}
We have
\begin{gather*}
\rsw_{\bar{\p}_2,n}(\chi_1) = [\bar{\eta} , -\bar{\eta'}]_{\bar{\pi}_2,n}, \\
\rsw_{\bar{\p}_2,n}(\chi_2) =  [-\bar{\omega}, \bar{\omega'}]_{\bar{\pi}_2,n},
\end{gather*}
where the bars denote reductions modulo $\p_1$.
\end{lemma}
\begin{proof}
As in the proof of Lemma~\ref{lem:new7.1}, we reduce to the case where $\dim(A)=2$; $A$ is complete; $\H_p^{c+1}(\ell) \neq 0$ where $c$ is defined by $p^c=[\ell:\ell^p]$; and the order of $\chi$ is a power of $p$. Our assumption that $\sw_{\p_2}(\chi)=n \geq 1$ implies that $q\leq c+2$. This follows from Proposition~\ref{prop:kato61}(2) because $K_{\p_2}$ has residue field $\kappa(\p_2)$ with $[\kappa(\p_2):\kappa(\p_2)^p]=p^{c+1}$ (see~\cite[Lemma~7.2]{K}, for example).

Consider the element
\[
t = \{ \chi, 1 + a \pi_2^n, b_1, \dotsc, b_{c+2-q} \} \in \H^{c+3}(K)
\]
for arbitrary $a \in A$ and $b_1, \dotsc, b_{c+2-q} \in A[1/\pi_1\pi_2]^\times$. Since $t$ is unramified on $\Spec A \setminus ( \overline{\{\p_1\}} \cup  \overline{\{\p_2\}})$, we have $\partial_{\q}(t)=0$ for all $\q\in \Spec(A)^1\setminus\{\p_1,\p_2\}$ by \cite[\S 7.5]{K}. Therefore, Kato's complex~\eqref{eq:complex} gives 
\begin{equation}\label{eq:complex2}
\partial'_{\p_1}\partial_{\p_1}(t) + \partial'_{\p_2}\partial_{\p_2}(t) = 0.
\end{equation}

Beware that in what follows, for $x\in A$, the notation $\bar{x}$ will sometimes denote the reduction of $x$ modulo $\p_1$ and sometimes it will denote the reduction of $x$ modulo $\p_2$. In each new context, we will state which reduction is meant by the notation. The notation $\bar{\bar{x}}$ will always mean the image of $x$ in $\ell$, in other words its reduction modulo $\p_1+\p_2$.

In $\H^{c+3}(K_{\p_2})$, by definition of the refined Swan conductor (Definition~\ref{def:rsw}) we have
\[\{\chi,1 + a \pi_2^n\}=\lambda_{\pi_2}(\bar{a} (\eta + \omega \wedge \dlog\bar{\pi}_1), \bar{a} (\eta' + \omega' \wedge \dlog\bar{\pi}_1)),\] where the bars denote reduction modulo $\p_2$. Therefore, by definitions of $t$ and of $\lambda_\pi$ (Section~\ref{sec:lambda-dvr} and \eqref{eq:lambda-comp}),
\begin{align*}
t &= \{ \lambda_{\pi_2}(\bar{a} (\eta + \omega \wedge \dlog\bar{\pi}_1), \bar{a} (\eta' + \omega' \wedge \dlog\bar{\pi}_1)), b_1, \dotsc, b_{c+2-q} \}\\
&= \{\iota^{q+1}(\delta_1(\bar{a} (\eta + \omega \wedge \dlog\bar{\pi}_1))), b_1, \dotsc, b_{c+2-q}\}\\
&\qquad+\{\iota^q(\delta_1(\bar{a} (\eta' + \omega' \wedge \dlog\bar{\pi}_1))),\pi_2, b_1, \dotsc, b_{c+2-q}\}
\end{align*}
in $\H^{c+3}(K_{\p_2})$, where again the bars denote reduction modulo $\p_2$. 

Suppose first that $b_1,\dots,b_{c+2-q}\in A^\times$. Then Lemma~\ref{lem:delta-wedge} gives
\begin{align*}
\partial_{\p_2}(t) &= (-1)^{c-q}\{\delta_1(\bar{a} (\eta' + \omega' \wedge \dlog\bar{\pi}_1)), \bar{b}_1, \dotsc, \bar{b}_{c+2-q} \}\\
&=(-1)^{c-q}\{\delta_1(\bar{a} \eta'),\bar{b}_1, \dotsc, \bar{b}_{c+2-q}\} \\
& \qquad \qquad + (-1)^{c-q}\{\delta_1(\bar{a}\omega' \wedge \dlog\bar{\pi}_1), \bar{b}_1, \dotsc, \bar{b}_{c+2-q} \}\\
&=(-1)^{c-q}\{\delta_1(\bar{a} \eta'),\bar{b}_1, \dotsc, \bar{b}_{c+2-q}\} + (-1)^{c-q}\{\delta_1(\bar{a}\omega'), \bar{\pi}_1, \bar{b}_1, \dotsc, \bar{b}_{c+2-q} \}\\
&=(-1)^{c-q}\{\delta_1(\bar{a} \eta'),\bar{b}_1, \dotsc, \bar{b}_{c+2-q}\} +\{\delta_1(\bar{a}\omega'), \bar{b}_1, \dotsc, \bar{b}_{c+2-q},\bar{\pi}_1 \}
\end{align*}
in $\H^{c+2}(\kappa(\p_2))$, where the bars denote reduction modulo $\p_2$. Therefore, 
\begin{equation}\label{eq:p2}
\partial'_{\p_2}\partial_{\p_2}(t)=\{\delta_1(\bar{\bar{a}}\bar{\omega'}), \bar{\bar{b}}_1, \dotsc, \bar{\bar{b}}_{c+2-q}\}=\delta_1(\bar{\bar{a}}\bar{\omega'}\wedge  \dlog\bar{\bar{b}}_1\wedge \dotsc\wedge \dlog\bar{\bar{b}}_{c+2-q})
\end{equation}
in $\H^{c+1}(\ell)$, where $\bar{\bar{x}}$ denotes the image of $x$ in $\ell$ and $\bar{\omega'}$ denotes the image of $\omega'$ in $\Omega^{q-2}_\ell$. 

On the other hand, under the assumption that $b_1,\dots,b_{c+2-q}\in A^\times$, a calculation similar to the one in the proof of Lemma~\ref{lem:new7.1} yields
\[
\partial_{\p_1}(t) = (-1)^{c+1-q} \{\chi_2, 1 + \bar{a} \bar{\pi}_2^n, \bar{b}_1, \dotsc, \bar{b}_{c+2-q} \}
\]
in $\H^{c+2}(\kappa(\p_1))$, where this time the bars denote reduction modulo $\p_1$.
Now write $\rsw_{\bar{\p}_2,n}(\chi_2) = [\alpha,\beta]_{\bar{\pi}_2,n}$ for some $\alpha\in \Omega^{q-1}_\ell, \beta\in\Omega^{q-2}_\ell$, where again the bars denote reduction modulo $\p_1$. By definition of the refined Swan conductor, this means $\{ \chi_2, 1+\bar{a} \bar{\pi}_2^n \} = \lambda_{\bar{\pi}_2}(\bar{\bar{a}} \alpha,\bar{\bar{a}} \beta)$, where $\bar{x}$ denotes the reduction of $x$ modulo $\p_1$ and $\bar{\bar{x}}$ denotes the image of $x$ in $\ell$. Thus, with the same convention for the bar notation,
\begin{align*}
\partial_{\p_1}(t) &= (-1)^{c+1-q} \{\lambda_{\bar{\pi}_2}(\bar{\bar{a}} \alpha,\bar{\bar{a}} \beta), \bar{b}_1, \dotsc, \bar{b}_{c+2-q} \}\\
&=(-1)^{c+1-q}\{\iota^{q}(\delta_1(\bar{\bar{a}} \alpha)), \bar{b}_1, \dotsc, \bar{b}_{c+2-q} \} \\
& \qquad \qquad +(-1)^{c+1-q} \{\iota^{q-1}(\delta_1(\bar{\bar{a}} \beta)), \bar{\pi}_2, \bar{b}_1, \dotsc, \bar{b}_{c+2-q} \}\\
&=(-1)^{c+1-q}\{\iota^{q}(\delta_1(\bar{\bar{a}} \alpha)), \bar{b}_1, \dotsc, \bar{b}_{c+2-q} \}-\{\iota^{q-1}(\delta_1(\bar{\bar{a}} \beta)), \bar{b}_1, \dotsc, \bar{b}_{c+2-q}, \bar{\pi}_2 \}
\end{align*}
in $\H^{c+2}(\kappa(\p_1))$.
 We obtain
\begin{equation}\label{eq:p1}
\partial'_{\p_1}\partial_{\p_1}(t) = -\{ \delta_1(\bar{\bar{a}}\beta), \bar{\bar{b}}_1, \dotsc, \bar{\bar{b}}_{c+2-q} \}
= -\delta_1(\bar{\bar{a}} \beta\wedge \dlog\bar{\bar{b}}_1 \wedge \dotsb \wedge \dlog\bar{\bar{b}}_{c+2-q} )
\end{equation}
in $\H^{c+1}(\ell)$. Now combine \eqref{eq:complex2}, \eqref{eq:p2}, \eqref{eq:p1} and Lemma~\ref{lem:c+1} below to obtain $\beta=\bar{\omega'}$.

The other equalities are obtained by taking some of the $b_i$ equal to $\pi_1$, $\pi_2$.
\end{proof}

\begin{lemma}\label{lem:c+1}
Let $\ell$ be a field of characteristic $p>0$ with $[\ell:\ell^p]=p^c$ and $\H^{c+1}_p(\ell) \neq 0$.  Let $\alpha \in \Omega^q_\ell$ be such that, for all $a \in \ell$ and all $b_1, \dotsc, b_{c-q} \in \ell^\times$, 
\[
\delta_1(a \alpha \wedge \dlog b_1 \wedge \dotsb \wedge \dlog b_{c-q})=0 
\]
in $\H^{c+1}_p(\ell)$. Then $\alpha=0$.
\end{lemma}
\begin{proof}
By~\cite[\href{https://stacks.math.columbia.edu/tag/07P2}{Lemma 07P2}]{stacks-project}, $\Omega^1_\ell$ is a $\ell$-vector space of dimension $c$.
By linear algebra, $\Omega^c_\ell$ has dimension $1$, and the wedge product $\Omega^q_\ell \times \Omega^{c-q}_\ell \to \Omega^c_\ell$ is a perfect $\ell$-linear pairing.
The elements $a \dlog b_1 \wedge \dotsb \wedge \dlog b_{c-q}$ for $a \in \ell$ and $b_1, \dotsc, b_{c-q} \in \ell^\times$ generate $\Omega^{c-q}_\ell$; so, if $\alpha \neq 0$, then the hypothesis implies that $\delta_1(\Omega^c_\ell)=0$.
But $\delta_1 \colon \Omega^c_\ell \to \H^{c+1}_p(\ell)$ is surjective, so this contradicts $\H^{c+1}_p(\ell) \neq 0$. 
\end{proof}

\section{Blowing up}\label{sec:blowups}

In this section, we show how information about the refined Swan conductor $[\alpha, \beta]_{\pi,n}$ is retained under blowups. Namely, in Lemma~\ref{lem:blowup}, we show that after a blowup at a point $P_0$ on the special fibre, one can read off $\alpha_{P_0}$ and $\beta_{P_0}$ from the residues at logarithmic poles of some relevant differentials.

First, we introduce some notation.


\subsection{Residues}
Let $q\geq 1$, let $X\to S$ be a morphism of schemes and let $D\subset X$ be an effective Cartier divisor.
Suppose that $X$ \'{e}tale locally along $D$ looks like $D\times \mathbb{A}^1$, in the sense that the de Rham complex of log poles is defined for $D\subset X$ over $S$ (see~\cite[\href{https://stacks.math.columbia.edu/tag/0FMV}{Definition 0FMV}]{stacks-project}); this is for example true if $D$ is a smooth prime divisor in a smooth variety $X$ over a field.
Recall (from~\cite[\href{https://stacks.math.columbia.edu/tag/0FMU}{Section 0FMU}]{stacks-project}, for example)
the exact sequence
\begin{equation}\label{eq:residuediff}
0 \to \Omega^q_{X/S} \to \Omega^q_{X/S}(\log D) \xrightarrow{\rho} \Omega^{q-1}_D \to 0
\end{equation}
of sheaves on $X$, where $\rho$ is the residue map:
on a sufficiently small open set $U$, we have $D\cap U=(f)$ for some rational function $f$ and any section $\omega$ of $\Omega^q(\log D)$ on $U$ can be written as $\omega = \eta + g\, \dlog f$ with $\eta \in \Omega^q(U)$ and $g \in \Omega^{q-1}(U)$; then $\rho(\omega)$ is defined to be $g|_D$, which is independent of the choices made.

\subsection{Residues after a blowup}
For the remainder of this section, let $Y$ be a smooth variety of dimension $m$ over a field $L$, $P_0 \in Y(L)$ a point, $R=\O_{Y,P_0}$ the local ring at $P_0$ with maximal ideal $\m$.  Let $f \colon \widetilde{Y} \to \Spec R$ be the blowup at $P_0$;
by definition, we have $\widetilde{Y} = \Proj S$, where $S$ is the graded ring $R \oplus \bigoplus_{i \ge 1} \m^i$, with the grading putting the factor $\m^i$ in degree $i$.
Let $Z = f^{-1}(P_0)$ be the exceptional fibre, isomorphic to $\P^{m-1}_L$.
There is a natural homomorphism $\m \to S$ sending $x \in \m$ to the element $x^{(1)} \in S$, which is $x$ considered as an element of the degree-1 factor of $S$.
This induces a natural homomorphism $\m/\m^2 \to \H^0(Z,\O(1))$, which we also denote by $x \mapsto x^{(1)}$.

We begin by identifying some natural isomorphisms.

The map $d \colon \m/\m^2 \to \Omega^1_{Y/L}|_{P_0} = \Omega^1_{R/L} \otimes_R L$ is an isomorphism; let 
\begin{equation}\label{eq:psi}
\psi \colon \Omega^1_{Y/L}|_{P_0} \to \H^0(Z,\O_Z(1))
\end{equation}
be the inverse of this isomorphism composed with the map $x \mapsto x^{(1)}$ defined above (so $\psi(dx) = x^{(1)}$).

By~\cite[Theorem~II.8.13]{Hartshorne}, there is a short exact sequence of sheaves
\[
0 \to \Omega^1_{Z/L}(2) \to \O_Z(1)^m \to \O_Z(2) \to 0
\]
on $Z$, giving a short exact sequence
\[
0 \to \H^0(Z,\Omega^1_{Z/L}(2)) \to \H^0(Z,\O_Z(1))^m \to \H^0(Z,\O_Z(2)).
\]
Observe that $\H^0(Z,\O_Z(1))^m$ can be identified with $\H^0(Z,\O_Z(1))^{\otimes 2}$: one concrete way to see this is by choosing an $L$-basis (comprising $m$ basis vectors) for $\H^0(Z,\O_Z(1))$.
Having made this identification, the map to $\H^0(Z,\O_Z(2))$ is given by $x \otimes y \mapsto xy$, and so its kernel is naturally (up to a choice of sign) identified with $\bigwedge^2 \H^0(Z,\O_Z(1))$ under the embedding $x \wedge y \mapsto (x \otimes y) - (y \otimes x)$.
Combining this with the isomorphism $\psi$ gives an isomorphism
\begin{equation}\label{eq:varphi}
\varphi \colon \Omega^2_{Y/L}|_{P_0} \to \H^0(Z,\Omega^1_{Z/L}(2)).
\end{equation}
The proof of~\cite[Theorem~II.8.13]{Hartshorne} leads to the following explicit description of $\varphi$: if $x_1, \dotsc, x_m$ form a basis for $\m/\m^2$, then 
\[
\varphi(dx_i \wedge dx_j) = (x_j^{(1)})^2 d(x_i^{(1)} / x_j^{(1)}).
\]

We are now ready to state the main result of this section.

\begin{lemma}\label{lem:blowup}
Fix $x \in \m \setminus \m^2$ and let $\pi_Z = f^* x$.
\begin{enumerate}
\item\label{it:piZ} The element $\pi_Z$ is a local parameter in $\O_{\widetilde{Y},Z}$.
\item\label{it:rswZ} Take $\beta \in \Omega^1_{R/L}$ and denote by $\beta_{P_0}$ the image of $\beta$ in $\Omega^1_{Y/L}|_{P_0}$.
Then $(f^* \beta)/\pi_Z$ lies in $\Omega^1_{\O_{\widetilde{Y},Z}}(\log Z)$ and has residue $\psi(\beta_{P_0}) / x^{(1)}$.
\item\label{it:swZ} If furthermore $\beta_{P_0} \neq 0$, then $(f^* \beta)/\pi_Z^r$ does not lie in $\Omega^1_{\O_{\widetilde{Y},Z}}(\log Z)$ for $r>1$.
\item\label{it:rswZ2} Take $\alpha \in \Omega^2_{R/L}$ and denote by $\alpha_{P_0}$ the image of $\alpha$ in $\Omega^2_{Y/L}|_{P_0}$.
Then $(f^* \alpha)/\pi_Z^2$ lies in $\Omega^2_{\O_{\widetilde{Y},Z}}(\log Z)$ and has residue $\varphi(\alpha_{P_0}) / (x^{(1)})^2$. 
\item\label{it:swZ2} If furthermore $\alpha_{P_0} \neq 0$, then $(f^* \alpha)/\pi_Z^r$ does not lie in $\Omega^2_{\O_{\widetilde{Y},Z}}(\log Z)$ for $r>2$.
\end{enumerate}
\end{lemma}

\begin{proof}
Since the statements concern $\O_{\widetilde{Y},Z}$, we may work on the affine piece $U$ of $\widetilde{Y}$ obtained by inverting $x^{(1)}$.  Denote by $f_U$ the restriction of $f$ to $U$.
We use standard facts about blowups: see, for example, \cite[\href{https://stacks.math.columbia.edu/tag/052P}{Section 052P}]{stacks-project}.
Extend $x$ to a basis $x=x_1, x_2, \dotsc, x_m$ for $\m/\m^2$; then we have $U = \Spec R'$, where
\[
R' = R[u_2, \dotsc, u_m] / (xu_2 - x_2, \dotsc, xu_m - x_m)
\]
and $u_i = x_i^{(1)}/x^{(1)}$.
The ideal $\m R'$ is principal, generated by the image of $x$, proving~(\ref{it:piZ}).  In what follows we will often abuse notation and identify $R$ with its image in $R'$.

We can write $\beta = a_1 dx_1 + \dotsb + a_m dx_m$ with $a_1, \dotsc, a_m \in R$.
Then
\[
f_U^*(\beta) = a_1 dx_1 + a_2 (u_2 dx_1 + x_1 du_2) + \dotsb + a_m (u_m dx_1 + x_1 du_m).
\]
and so (writing $\pi_Z = x = x_1$ in $R'$)
\[
f_U^*(\beta)/\pi_Z = (a_2 du_2 + \dotsb + a_m du_m) + (a_1 + a_2 u_2 + \dotsb + a_m u_m) d\pi_Z/\pi_Z.
\]
By definition, this lies in $\Omega^1_{\O_{\widetilde{Y},Z}}(\log Z)$.  Since each $a_i$ lies in $R$, the restriction of $a_i$ to $Z$ is the constant function $a_i(P_0)$.
The residue of $f_U^*(\beta)/\pi_Z$ is therefore
\begin{align*}
a_1(P_0) + a_2(P_0) u_2|_Z + \dotsb a_m(P_0) u_m|_Z
&= (a_1(P_0) x_1^{(1)} + \dotsc a_m(P_0) x_m^{(1)} ) / x_1^{(1)} \\
&= \psi(\beta_{P_0}) / x_1^{(1)}
\end{align*}
proving~(\ref{it:rswZ}).
Statement~(\ref{it:swZ}) follows easily: if $(f^* \beta)/\pi_Z^r$ were to lie in $\Omega^1_{\O_{\widetilde{Y},Z}}(\log Z)$ with $r>1$, then $(f^* \beta)/\pi_Z$ would lie in $\Omega^1_{\O_{\widetilde{Y},Z}}$; but this is not the case, since by~\eqref{it:rswZ} its residue is non-zero.

To prove~(\ref{it:rswZ2}), write $\alpha = \sum_{i > j} a_{ij} dx_i \wedge dx_j$ with $a_{ij} \in R$.
We treat the terms separately.

For terms with $j=1$, we have
\[
f_U^*(a_{i1} dx_i \wedge dx_1) = a_{i1} d(x_1 u_i) \wedge dx_1 = a_{i1} x_1 du_i \wedge dx_1
\]
and so
\[
\pi_Z^{-2} f_U^*(a_{i1} dx_i \wedge dx_1) = a_{i1} du_i \wedge d\pi_Z/\pi_Z,
\]
which has residue
\[
a_{i1}(P_0) d(x_i^{(1)}/x_1^{(1)}) = \varphi(a_{i1}(P_0) dx_i \wedge dx_1) / (x_1^{(1)})^2.
\]
as required.  For terms with $i,j \neq 1$, we have
\begin{align*}
f_U^*(a_{ij} dx_i \wedge dx_j) &= a_{ij} d(x_1 u_i) \wedge d(x_1 u_j) \\
&= a_{ij}( x_1^2 du_i \wedge du_j + x_1 u_i dx_1 \wedge du_j + x_1 u_j du_i \wedge dx_1 )  
\end{align*}
and so
\[
\pi_Z^{-2} f_U^*(a_{ij} dx_i \wedge dx_j)
= a_{ij} du_i \wedge du_j + a_{ij}( u_j du_i - u_i du_j) \wedge d\pi_Z/\pi_Z,
\]
with residue
\[
a_{ij}(P_0) u_j^2 d(x_i^{(1)}/x_j^{(1)}) = \varphi(a_{ij}(P_0) dx_i \wedge dx_j) / (x_1^{(1)})^2.
\]
This proves~(\ref{it:rswZ2}), and~(\ref{it:swZ2}) follows as in the first case.
\end{proof}

\section{Some calculations for $\mathbb{P}^n$}\label{sec:Pn}
This section collects some calculations for projective space. In Section~\ref{sec:diffPn} we study some spaces of differentials on projective space with poles of bounded (logarithmic) order along a hyperplane, and show to what extent they are determined by their residues. This will be used in Section~\ref{sec:endgames} in the proof of Theorem~\ref{thm:onelemma_intro} in conjunction with a result of Kato (Proposition~\ref{prop:7.3}) to transfer information about the refined Swan conductor along a chain of blowups. The relevant projective spaces will be the exceptional divisors of these blowups. In Section~\ref{sec:SwanPn} we describe the graded pieces of Kato's filtration by Swan conductor on $\H^1(E\setminus Z, \Q/\Z)$, where $E$ is projective space over a field of characteristic $p$ and $Z$ is a hyperplane. This will be used in Section~\ref{sec:endgames} in the proof of Theorem~\ref{thm:onelemma_intro} once our successive blowups have reduced the Swan conductor to zero, and Proposition~\ref{prop:tame} has reduced our task to computing the residue map $\partial$ via Lemma~\ref{lem:inv_trace}.

\subsection{Differentials with logarithmic poles along a hyperplane in $\P^n$}\label{sec:diffPn}
Let $L$ be a field and let $\P^n=\P^n_L$ have coordinates $X_0, \dotsc, X_n$.
Let $H \subset \P^n$ be the hyperplane defined by $X_0=0$.

\begin{lemma}\label{lem:logdiff}
\begin{enumerate}
\item\label{logdiff1} The homomorphism
\begin{equation}\label{eq:rho}
\H^0(\P^n, \Omega^1(\log H)(H)) \xrightarrow{\rho} \H^0(H,\O(1)),
\end{equation}
obtained by forming the tensor product with $\O(H)$ in~\eqref{eq:residuediff} and applying $\O(H)|_H \cong \O_H(1)$, is an isomorphism.
$\H^0(\P^n, \Omega^1(\log H)(H))$ has a basis $\{ d(X_i/X_0) : 1 \le i \le n \}$
satisfying $\rho(d(X_i/X_0)) = -X_i|_H$.
\item\label{logdiff2} $\H^0(\P^n, \Omega^1(2H))$ has a basis consisting of the $n(n+1)/2$ elements
\[
\{ (X_i^2/X_0^2)d(X_j/X_i) = -(X_j^2/X_0^2)d(X_i/X_j) : 0\le i<j \le n \}.
\]
The $n$ of these with $i=0$ make up the aforementioned basis of the subspace $\H^0(\P^n, \Omega^1(\log H)(H))$.
The natural morphism of sheaves $\Omega^1_{\P^n} \to \iota_* \Omega^1_H$ (where $\iota$ is the inclusion of $H$ in $\P^n$) gives rise to a short exact sequence
\[
0 \to \H^0(\P^n, \Omega^1(\log H)(H)) \to \H^0(\P^n,\Omega^1(2H)) \xrightarrow{\omega \mapsto \omega|_H} \H^0(H,\Omega^1(2)) \to 0
\]
in which, for $0<i<j \le n$, the basis element $(X_i^2/X_0^2) d(X_j/X_i)$ maps to $X_i^2 d(X_j/X_i) \in \H^0(H,\Omega^1(2))$.
\item\label{logdiff3} The homomorphism
\[
\H^0(\P^n,\Omega^2(\log H)(2H)) \xrightarrow{\rho} \H^0(H,\Omega^1(2)),
\]
obtained by forming the tensor product with $\O(2H)$ in~\eqref{eq:residuediff} and using $\O(2H)|_H \cong \O_H(2)$, is an isomorphism.
$\H^0(\P^n,\Omega^2(\log H)(2H))$ has a basis consisting of the $n(n-1)/2$ elements 
\[
\{ d(X_i/X_0) \wedge d(X_j/X_0) : 1 \le i<j \le n \}
\]
satisfying $\rho(d(X_i/X_0) \wedge d(X_j/X_0)) = X_i^2 d(X_j/X_i)$.
\end{enumerate}
\end{lemma}

\begin{proof}
We will repeatedly use the short exact sequence
\begin{equation}\label{eq:omega-seq}
0 \to \Omega^1_{\P^n} \to \O_{\P^n}(-1)^{n+1} \to \O_{\P^n} \to 0
\end{equation}
of~\cite[Theorem~II.8.13]{Hartshorne} which, together with the standard calculation of the cohomology groups $\H^i(\P^n,\O(r))$, lets us calculate the cohomology of $\Omega^1(r)$ (or, equivalently, $\Omega^1(rH)$).
For $r>0$,
taking the tensor product with $\O(r)$ and taking cohomology gives the exact sequence
\begin{multline}\label{eq:omega-long-seq}
0 \to \H^0(\P^n,\Omega^1(r)) \to \H^0(\P^n, \O(r-1))^{n+1} \xrightarrow{\alpha}
\H^0(\P^n, \O(r)) \to \\ \to \H^1(\P^n, \Omega^1(r)) \to 0,
\end{multline}
where we have used $\H^1(\P^n,\O(r-1))=0$.
Identifying $\H^0(\P^n,\O(r-1))$ with the space of homogeneous polynomials of degree $r-1$,
the map $\alpha$ is given by
\[
\alpha(f_0, \dotsc, f_n) = f_0 X_0 + \dotsb + f_n X_n.
\]

To prove~(\ref{logdiff1}), 
take $r=1$; then $\alpha$ is an isomorphism, showing that $\H^0(\P^n,\Omega^1(1))$ and $\H^1(\P^n,\Omega^1(1))$ are both zero.
It then follows easily that 
\[
\rho \colon \H^0(\P^n, \Omega^1(\log H)(H)) \to \H^0(H,\O(1))
\]
is an isomorphism.
We have
\[
d(X_i/X_0) = (X_i/X_0) \dlog (X_i/X_0) = -(X_i/X_0) \dlog (X_0/X_i)
\]
showing $\rho(d(X_i/X_0)) = -X_i|_H$. The elements $-X_i|_H$ form a basis of $\H^0(H,\O(1))$, showing that the $d(X_i/X_0)$ form a basis for $\H^0(\P^n, \Omega^1(\log H)(H))$.

To prove~(\ref{logdiff2}), take $r=2$; then a basis for the kernel of $\alpha$ is given by the $n(n+1)/2$ elements having $X_i$ in the $j$th position and $-X_j$ in the $i$th position, for $i < j$.  Therefore $\H^0(\P^n,\Omega^1(2H))$ has dimension $n(n+1)/2$.

We can now prove by induction on $n$ that the claimed elements do indeed form a basis.  For $n=1$ the dimension is $1$ and it is clear. Assuming the statement to be true for $\P^{n-1}$, the elements $X_i^2 d(X_j/X_i)$ with $0<i<j\le n$ form a basis for $\H^0(H,\Omega^1(2))$.
Now, of the claimed basis elements $(X_i^2/X_0^2)d(X_j/X_i)$, those with $i=0$ form a basis for $\H^0(\P^n,\Omega^1(\log H)(H))$ by~(\ref{logdiff1}) and restrict to zero on $H$.
Those with $i>0$ map bijectively onto our basis for $\H^0(H,\Omega^1(2))$, showing that all together they form a basis for $\H^0(\P^n,\Omega^1(2H))$.

To prove~(\ref{logdiff3}), we first verify 
$\rho(d(X_i/X_0) \wedge d(X_j/X_0)) = X_i^2 d(X_j/X_i)$
by writing
\begin{align*}
d(X_i/X_0) \wedge d(X_j/X_0) &= (X_iX_j/X_0^2) \dlog(X_i/X_0) \wedge \dlog(X_j/X_0) \\
&= (X_iX_j/X_0^2) \dlog(X_j/X_0) \wedge \dlog(X_0/X_i) \\
&= (X_iX_j/X_0^2) (\dlog(X_j/X_i) - \dlog(X_0/X_i))\wedge \dlog(X_0/X_i) \\
&= (X_i^2/X_0^2) d(X_j/X_i)\wedge \dlog(X_0/X_i)
\end{align*}
which, since $X_0/X_i$ is a local parameter at $H$, gives the claimed result.

Given the exact sequence
\[
0 \to \H^0(\P^n,\Omega^2(2H)) \to \H^0(\P^n,\Omega^2(\log H)(2H))
\xrightarrow{\rho} \H^0(H,\Omega^1(2)),
\]
and that the elements $X_i^2d(X_j/X_j)$ form a basis for $\H^0(H,\Omega^1(2))$ by~(\ref{logdiff2}), it is now enough to show $\H^0(\P^n,\Omega^2(2H))=0$.
To see this, consider the short exact sequence
\[
0 \to \Omega^2_{\P^n} \to \wedge^2 (\O_{\P^n}(-1)^{n+1}) \to \Omega^1_{\P^n} \to 0
\]
arising from~\eqref{eq:omega-seq}; for details see~\cite[\href{https://stacks.math.columbia.edu/tag/0FUK}{Lemma 0FUK}]{stacks-project} and its proof.
As remarked there, the middle term is isomorphic to the direct sum of $n(n+1)/2$ copies of $\O(-2)$.
Twisting and taking global sections gives a short exact sequence
\[
0 \to \H^0(\P^n,\Omega^2(2)) \to \H^0(\P^n,\O)^{n(n+1)/2} \to \H^0(\P^n,\Omega^1(2)) \to \H^1(\P^n,\Omega^2(2))
\]
in which the last term vanishes by \emph{loc.\ cit.}.
Now comparing dimensions shows $\H^0(\P^n,\Omega^2(2))=0$, completing the proof.
\end{proof}

\subsection{(Refined) Swan conductors for $\H^1(E\setminus Z, \Q/\Z)$}\label{sec:SwanPn}

Let $\F$ be a field of characteristic $p>0$, let $E=\P_\F^m$ and let $Z\subset E$ be a hyperplane. Let $\kappa(E)_Z$ be the field of fractions of the Henselisation of $\O_{E,Z}$. Let $\frac{1}{p}$ denote the canonical map $\H^1_p(\kappa(E)_Z)\to \H^1(\kappa(E)_Z, \Q/\Z)$ induced by identifying $\Z/p$ with the $p$-torsion in $\Q/\Z$.

\begin{lemma}\label{lem:H1EminusZ}
Let $E=\P^m_\F$ with homogeneous coordinates $X_0, \dotsc, X_m$ and let $Z$ be the hyperplane $X_0=0$.
Write $\pi = X_0/X_1$.
Let $\kappa(E)_Z$ be the field of fractions of the Henselisation of $\O_{E,Z}$ and
consider the filtration on $\H^1(E\setminus Z, \Q/\Z)$ obtained by pulling back $\{ \fil_n \H^1(\kappa(E)_Z) \}$.
\begin{enumerate}
\item\label{it:fil0EZ} $\fil_0 \H^1(E\setminus Z, \Q/\Z)$ consists of the constant classes $\H^1(\F,\Q/\Z)$.
\item For $p \nmid n$, there is an isomorphism
\[
\H^0(Z,\O(n)) \to \gr_n \H^1(E\setminus Z, \Q/\Z)
\]
given by $F \mapsto \frac{1}{p}\delta_1(\widetilde{F}/X_0^n)$, where $\widetilde{F} \in \H^0(E,\O(n))$ is any lift of $F$ and $\delta_1 \colon \kappa(E)_Z \to \H^1_p(\kappa(E)_Z)$ is the Artin--Schreier map. 
The refined Swan conductor is given by
\[
\rsw_n\left(\frac{1}{p}\delta_1(\widetilde{F}/X_0^n)\right) = \big[ -d(F/X_1^n) , n F/X_1^n  \big]_{\pi, n}.
\]
\end{enumerate}
\end{lemma}
\begin{proof}
There is an exact sequence
\[
0 \to \H^1(\F,\Q/\Z) \to \H^1(E \setminus Z, \Q/\Z) \to \H^1(\bar{E} \setminus \bar{Z}, \Q/\Z)
\]
where $\bar{E},\bar{Z}$ are base changes of $E,Z$ to a separable closure of $\F$.
By Proposition~\ref{prop:kato61}, the piece $\fil_0 \H^1(E\setminus Z, \Q/\Z)$ is generated by the prime-to-$p$ torsion in $\H^1(E \setminus Z, \Q/\Z)$ together with the elements unramified at $Z$, that is, the image of $\H^1(E, \Q/\Z)$.
We have $E= \P^m_\F$ and $E \setminus Z \cong \mathbb{A}^m_\F$.
Therefore $\H^1(\bar{E} \setminus \bar{Z}, \Q/\Z)$ has no prime-to-$p$ torsion, showing that the prime-to-$p$ torsion in $\H^1(E \setminus Z, \Q/\Z)$ all comes from $\H^1(\F,\Q/\Z)$.
Moreover, $\H^1(\bar{E},\Q/\Z)$ is trivial, showing that the image of $\H^1(E,\Q/\Z)$ also coincides with $\H^1(\F,\Q/\Z)$.
This proves~(\ref{it:fil0EZ}).

Since the filtration on $\H^1(E \setminus Z, \Q/\Z)$ is obtained by pulling back that on $\H^1(\kappa(E)_Z)$, the resulting map
\[
\gr_n \H^1(E \setminus Z, \Q/\Z) \to \gr_n \H^1(\kappa(E)_Z)
\]
is injective.
By~\cite[Lemma~3.6]{K}, for $p\nmid n$ there is a surjection
\[
h: \kappa(Z) \to \gr_n \H^1(\kappa(E)_Z), \qquad x \mapsto\frac{1}{p} \delta_1(\tilde{x} \pi^{-n}).
\]
By~\cite[Lemma~3.7]{K}, the resulting element has refined Swan conductor
\[
\rsw_n\left(\frac{1}{p}\delta_1(\tilde{x} \pi^{-n})\right) = [-dx , nx ]_{\pi, n}.
\]
In particular, this shows that $h$ is an isomorphism for $p\nmid n$.
We claim that the image of $\H^0(Z,\O(n))$ under the injective map $F \mapsto F/X_1^n$ corresponds under the isomorphism $h$ to $\gr_n \H^1(E \setminus Z, \Q/\Z)$.

Indeed, we have $h(F/X_1^n) = \frac{1}{p}\delta_1(\widetilde{F}/X_0^n)$,  which is unramified outside $Z$.
On the other hand, if $\chi$ is an element of $\gr_n \H^1(E \setminus Z, \Q/\Z)$ then we write $\chi = \frac{1}{p}\delta_1(\tilde{x}\pi^{-n})$ and consider its refined Swan conductor.
By~Theorem~\ref{thm:7.1} applied to the local rings of all points in $Z$, we see that $x$ is regular on $Z$ apart from a pole of order at most $n$ along $X_1=0$.  (For points where $\pi$ is a local equation for $Z$ this follows immediately; at other points of $Z$ a simple change of variables is needed.)
Thus $x$ is of the form $F/X_1^n$ for some $F \in \H^0(Z,\O(n))$, as claimed.
\end{proof}

\section{Tangent vectors}\label{sec:tangent}
We return to the setting and notation of Theorem~\ref{thm:onelemma_intro}, wherein $k$ is a $p$-adic field with ring of integers $\O_k$, uniformiser $\pi$ and residue field $\F$, and $X/k$ is a smooth geometrically irreducible variety with smooth model $\X/\O_k$, and geometrically irreducible special fibre $Y$.
For $r\geq 1$ we write $q_r$ for the reduction map $\X(\O_k) \to \X(\O_k / \pi^r \O_k)$, where $\pi$ denotes a uniformiser of $k$. For $P \in \X(\O_{k})$ we use $B(P,r)$ to denote the set $q_r^{-1}(q_r(P))$ of points $Q\in \X(\O_{k})$ such that $Q$ has the same image as $P$ in $\X(\O_{k}/\pi^r)$. We write $P_0$ for the image of $P$ in $Y(\F)$.

In Lemma~\ref{lem:tangent} we collect some well-known facts relating lifts of points to tangent vectors. In Lemma~\ref{lem:blowuptv} we show how to keep track of these tangent vectors when blowing up our scheme $\X$ at a point on the special fibre.

\begin{lemma}\label{lem:tangent}
There is a function $B(P,r) \to T_{P_0}(Y)$, which we denote as $Q \mapsto \tv{r}{P}{Q}$, depending on the choice of uniformiser $\pi$ and with the following properties.
\begin{enumerate}
\item\label{tr1} The function factors as $q_{r+1}$ followed by a bijection from $q_{r+1}(B(P,r))$ to $T_{P_0}(Y)$.
\item\label{tr2} For a point $Q \in B(P,r)$ and a regular function $f \in \O_{\X,P_0}$, we have
\begin{equation}\label{eq:tv}
f(Q) \equiv f(P) + \pi^r df_{P_0}(\tv{r}{P}{Q}) \pmod{\pi^{r+1}},
\end{equation}
where $df \in \Omega^1_{\O_{Y,P_0}/\F}$ is the derivative of $f$ restricted to $Y$, $df_{P_0} \in \Omega^1_Y|_{P_0}$ is
its image in the stalk at $P_0$, and $\Omega^1_{Y}|_{P_0}$ is naturally identified with the $\F$-linear dual of $T_{P_0}(Y)$.
\item\label{tr3} Let $k'/k$ be a finite extension, with $\F'/\F$ the extension of residue fields, and let $X'$ and $Y'$ be the base changes of $X$ to $k'$ and $Y$ to $\F'$, respectively.
Let $P' \in X'(k')$ be the base change of $P$.
Fix a uniformiser $\pi'$ in $k'$ and write $\pi = c (\pi')^\varepsilon$ with $c \in \O_{k'}^\times$, so that $\varepsilon$ is the ramification index of $k'/k$.
Let $\bar{c}$ denote the image of $c$ in $\F^\times$.
Then the diagram
\[
\begin{CD}
q_{r+1}(B(P,r)) @>{\tv{r}{P}{\cdot}}>> T_{P_0}(Y) \\
@VVV @VV{\times \bar{c}^r}V \\
q_{\varepsilon r+1}(B(P',\varepsilon r)) @>{\tv{\varepsilon r}{P'}{\cdot}}>> T_{P_0}(Y') \\
\end{CD}
\]
commutes, where we identify $T_{P_0}(Y')$ with $T_{P_0}(Y) \otimes_\F \F'$. 
\end{enumerate}
\end{lemma}
\begin{proof}
One explicit way to see this is as follows.  
Write $d = \dim X$.  Since $\X \to \O_k$ is smooth at $P_0$, there is a neighbourhood of $P_0$ that embeds into $\Aff^n_{\O_k}$ as the zero set of $n-d$ polynomials $f_1, \dots, f_{n-d}$.
Such an embedding induces an embedding of the tangent space $T_P(\X)$ into $T_P(\Aff^n) \cong \O_k^n$. 
Consider a point $Q \in \Aff^n(\O_k)$ that is congruent to $P$ modulo $\pi^r$; we can write $Q = P + \pi^r \mathbf{v}$, where $\mathbf{v} \in T_P(\Aff^n_{\O_k})$ is a vector.
Using the Taylor expansion, the condition that $q_{r+1}(Q)$ lie in $\X$ can be written as
\begin{equation}\label{eq:taylor}
(f_1(Q), \dotsc, f_{n-d}(Q))
= (f_1(P), \dotsc, f_{n-d}(P)) + \pi^r \mathbf{J}(P) \mathbf{v} \equiv \mathbf{0} \pmod{\pi^{r+1}},
\end{equation}
where $\mathbf{J}$ is the $(n-d)\times n$ Jacobian matrix of partial derivatives of the $f_i$.
Let $\bar{\mathbf{v}} \in \F^n$ be the reduction of $\mathbf{v}$ modulo $\pi$; the reduction of $\mathbf{J}(P)$ modulo $\pi$ is $\mathbf{J}(P_0)$.
The condition~\eqref{eq:taylor} is equivalent to $\mathbf{J}(P_0) \bar{\mathbf{v}} = 0$, which simply says that $\bar{\mathbf{v}}$ lies in the tangent space $T_{P_0}(Y)$; because $Y$ is smooth at $P_0$, this is an $\F$-vector space of dimension $d$.
So every point $Q \in B(P,r)$ gives rise to a vector $\bar{\mathbf{v}} \in T_{P_0}(Y)$, and we define $\tv{r}{P}{Q} = \bar{\mathbf{v}}$.
Conversely, every $\bar{\mathbf{v}} \in T_{P_0}(Y)$ gives a solution to~\eqref{eq:taylor}, which by Hensel's Lemma lifts to a point of $B(P,r)$.
This defines the bijection of~(\ref{tr1}).

For~(\ref{tr2}), take $Q \in B(P,r)$ and write as before $Q = P + \pi^r \mathbf{v}$, where $\mathbf{v} \in \O_k^n$ has reduction $\bar{\mathbf{v}}$ lying in $T_{P_0}(Y)$.  The function $f$ extends to a regular function on a neighbourhood of $P_0$ in $\Aff^n_{\O_k}$, and we denote the extension also by $f$.  Taylor expansion gives
\[
f(Q) \equiv f(P) + \pi^r \nabla f(P) \cdot \mathbf{v} \pmod{\pi^{r+1}}.
\]
This depends only on $\bar{\mathbf{v}}$, and the restriction of $\nabla f(P)$ to $T_{P_0}(Y)$ is $df_{P_0}$, proving~(\ref{tr2}).
Also, property~(\ref{tr2}) characterises the bijection and does not depend on the embedding used to define it, showing that the bijection itself does not depend on the embedding.

The statement~(\ref{tr3}) follows easily from the definitions using
\[
P + \pi^r \mathbf{v} = P + c^r (\pi')^{\varepsilon r} \mathbf{v}. \qedhere
\]
\end{proof}

\begin{remark}
The canonical bijection is between $q_{r+1}(B(P,r))$ and the vector space $T_{P_0}(Y) \otimes_\F (\m^r / \m^{r+1})$, instead of $T_{P_0}(Y)$, as can be seen by applying~(\ref{tr3}) with $k=k'$.  See, for example, \cite[\S III.5]{SGA1}.
That gives a bijection independent of the choice of $\pi$.  However, since we will use the formula~(\ref{tr2}), 
we opt for the explicit rather than the canonical choice.
\end{remark}

\subsection{Tangent vectors and blowups}\label{sec:tangentblowup}
Let $\m$ denote the maximal ideal of $\O_{Y, P_0}$.
Let $f \colon \widetilde{\X} \to \X$ be the blowup of $\X$ at $P_0$ and let $E$ be the exceptional divisor, isomorphic to $\P^m_\F$.  Let $\widetilde{Y}$ denote the strict transform of $Y$. The linear form $\pi^{(1)} \in \H^0(E,\O(1))$ cuts out a hyperplane in $E$ which is $E \cap \widetilde{Y}$. Its complement $U$ is naturally isomorphic to $\Spec \Sym(\m/\m^2)$, the affine space corresponding to the vector space $\m/\m^2$. To make this explicit, choose a system of local parameters $\pi, x_1, \dotsc, x_m$ in $\O_{\X,P_0}$; then $\pi^{(1)}, x_1^{(1)} \dotsc, x_m^{(1)}$ give a system of projective coordinates on $E$, so the $u_i = x_i^{(1)}/\pi^{(1)}$ for $1\leq i\leq m$ restrict to a system of affine coordinates on $U$.
This means that, for every point $O \in U(\F)$, there is a natural isomorphism $\theta$ from $\m/\m^2$ to the cotangent space $\m_O / \m_O^2$.  (In explicit coordinates, $\theta(dx_i)=du_i$.)

\begin{lemma}\label{lem:blowuptv}
Let $Q \in B(P,r)$, let $\widetilde{P},\widetilde{Q}$ be the strict transforms of $P,Q$, respectively, and let $\widetilde{P}_0,\widetilde{Q}_0\in E(\F)$ be their respective reductions.
\begin{enumerate}
\item \label{blowupdu}  If $r>1$, then $\widetilde{Q} \in B(\widetilde{P},r-1)$ and the dual map $\theta^\vee \colon T_{\widetilde{P}_0} U \to T_{P_0} Y$ sends $\tv{r-1}{\widetilde{P}}{\widetilde{Q}}$ to $\tv{r}{P}{Q}$. In explicit coordinates, $du_i(\tv{r-1}{\widetilde{P}}{\widetilde{Q}})=dx_i(\tv{r}{P}{Q})$ for $1\leq i\leq m$.
\item \label{blowuptv1} If $r=1$ then $u_i(\widetilde{Q}_0)=dx_i(\tv{1}{P}{Q})$ for $1\leq i\leq m$.
\end{enumerate}
\end{lemma}

\begin{proof}
Work on the affine piece of $\tilde{\X}$ corresponding to $\pi^{(1)}$; then $f$ is defined by $x_i = \pi u_i$ for $1\leq i\leq m$ and~\eqref{eq:tv} yields
\begin{equation}\label{eq:tvx}
\pi u_i(\widetilde{Q})\equiv \pi u_i(\widetilde{P})+\pi^r dx_i(\tv{r}{P}{Q}) \pmod{\pi^{r+1}}.
\end{equation}
Hence, $u_i(\widetilde{Q})\equiv u_i(\widetilde{P})  \pmod{\pi^{r-1}}$ for $1\leq i\leq m$, whereby $\widetilde{Q} \in B(\widetilde{P},r-1)$. For $r>1$,~\eqref{eq:tv} yields
\begin{equation}\label{eq:tvu}
u_i(\widetilde{Q})\equiv u_i(\widetilde{P})+\pi^{r-1} du_i(\tv{r-1}{\widetilde{P}}{\widetilde{Q}}) \pmod{\pi^{r}}.
\end{equation}
and comparing~\eqref{eq:tvx} and~\eqref{eq:tvu} proves~\eqref{blowupdu}.

To prove~\eqref{blowuptv1}, observe that for $r=1$,~\eqref{eq:tvx} gives
\begin{align*}
dx_i(\tv{1}{P}{Q})&\equiv u_i(\widetilde{Q})-u_i(\widetilde{P})\pmod{\pi}\\
&=u_i(\widetilde{Q}_0)-u_i(\widetilde{P}_0)=u_i(\widetilde{Q}_0)
\end{align*}
since $u_i(\widetilde{P}_0)=0$.
\end{proof}

\section{Proof of Theorem~\ref{thm:onelemma_intro}}\label{sec:endgames}
We continue with the setting and notation of Theorem~\ref{thm:onelemma_intro}. In this section, we will prove the following strengthening of Theorem~\ref{thm:onelemma_intro}.

\begin{theorem}\label{thm:onelemma}
Let $k$ be a finite extension of $\Q_p$.
Let $X$ be a smooth, geometrically irreducible variety over $k$, and let $\X \to\Spec \O_k$ be a smooth model of $X$. 
Suppose that the special fibre $Y$ of $\X$ is geometrically irreducible.
Let $n>0$, 
let $\A\in\fil_n \Br X$, and let $\rsw_{n}(\A) = [\alpha, \beta]_{\pi,n}$ for some $(\alpha,\beta)\in \Omega^2_F \oplus \Omega^1_F$.
Let $P\in \X(\O_k)$, and let $P_0 \in Y(\F)$ be the reduction of $P$.
Then \begin{enumerate}
\item\label{onelemma_reg} $\alpha$ and $\beta$ are regular at $P_0$. 
\end{enumerate}
Moreover, we have the following description of the evaluation map $\evmap{\A} \colon \X(\O_k) \to \Br k$.  
\begin{enumerate}[resume]
\item \label{onelemma_1} For $Q \in B(P,n)$,
\[
\inv \A(Q) = \inv \A(P) + \frac{1}{p}\Tr_{\F/\F_p} \beta_{P_0}(\tv{n}{P}{Q}).
\]
In particular, if $\beta_{P_0}\neq 0$ then $\evmap{\A}$ takes $p$ distinct values on $B(P,n)$. 

\item\label{rla} If $\beta=0$ and $n>2$ then there exists $\gamma\in\Omega^1_Y|_{P_0}$ such that, for $Q\in B(P,1)$ and $R\in B(Q,n-1)$, 
\[
\inv \A(R) = \inv \A(Q) - \frac{1}{p} \Tr_{\F/\F_p} \alpha_{P_0}(\tv{1}{P}{Q},\tv{n-1}{Q}{R}) + \frac{1}{p}\Tr_{\F/\F_p}(\gamma(\tv{n-1}{Q}{R})).
\]
If, furthermore, $p$ is odd and $2< n< e'+2+te$ for some $t\in\Z_{\geq 0}$ with $[k( \mmu_{p^{t+1}}):k]$ coprime to $p$ then the following holds. For any integer $s$ with $1 \le s  < n/2$, $Q\in B(P,s)$ 
and $R \in B(Q,r)$ with $r=n-s$, we have
\[
\inv \A(R) = \inv \A(Q) - \frac{s}{p} \Tr_{\F/\F_p} \alpha_{P_0}(\tv{s}{P}{Q},\tv{r}{Q}{R}) + g_{\A,Q}(R),
\]
where $g_{\A,Q}: B(Q,r)\to p^{-(t+1)}\Z/\Z$ is a continuous function satisfying $g_{\A,Q}(R)=\frac{1}{p}\Tr_{\F/\F_p}(\gamma(\tv{n-1}{Q}{R}))$ 
for all $R\in B(Q,n-1)$.

\item \label{surj2} If $\beta=0$ and $\alpha_{P_0}\neq 0$ then there exists $Q \in B(P,1)$ such that $\evmap{\A}$ takes $p$ distinct values on $B(Q,n-1)$. 

\item \label{surj3} Suppose $\A$ has order $p^{t+1}$ in $\Br X$ for some $t\in\Z_{\geq 0}$. Suppose either that $n> e'+(t-1)e$, or that $n=e'+(t-1)e$ and $C(\alpha)=C(\beta)
=0$, where $C$ denotes the Cartier operator. 
\begin{enumerate}[(i)]
\item If $\beta_{P_0}\neq 0$ then $\evmap{\A}\colon B(P,n-te)\to \Br k[p^{t+1}] \cong p^{-(t+1)}\Z/\Z$ is surjective. 
\item \label{surj3ii} If $\beta=0$ and $\alpha_{P_0}\neq 0$ then, under the additional assumptions that $n>e'+2$ and $n\geq te+3$, there exists a point $Q \in B(P,1)$ such that $\evmap{\A}:B(Q,n-te-1)\to \Br k[p^{t+1}] \cong p^{-(t+1)}\Z/\Z$ is surjective.
\end{enumerate}
\end{enumerate}
\end{theorem}

\begin{remark} 
\begin{enumerate}
\item As remarked in the introduction, case~\eqref{surj2} is only possible if $p\mid n$. In fact, Lemma~\ref{lem:dbna} shows that if $p\nmid n$ and $\beta=0$ then $\alpha=0$. 
\item Lemma~\ref{lem:p=n=2} provides a complement to~\eqref{rla} in the case where $\beta=0$ and $p=n=2$.
\item The function $g_{\A,Q}(R)$ in~\eqref{rla} satisfies $g_{\A,Q}(R)=\inv \mathcal{D}(R)-\inv \mathcal{D}(Q)$ for some $\mathcal{D}\in \fil_{n-1} \Br(\O_{\X,P_0}[p^{-1}])$: see the proof on page~\pageref{onelemma-proof}. 
\item The condition in~\eqref{rla} that the Swan conductor of $\A$ be smaller than $e'+2+te$ for some $t\in\Z_{\geq 0}$ with $[k( \mmu_{p^{t+1}}):k]$ 
coprime to $p$ is automatically satisfied (with $t=0$) if $p^2$ does not divide the order of $\A$ in $\Br X$, see~\cite[Proposition~4.1(1)]{K} and Proposition~\ref{prop:kato61}(1).
\item The condition in~\eqref{surj3}\ref{surj3ii} that $n\geq te+3$ is only needed if $p=e=2$. In all other cases where $t\geq 1$, it follows from $n\geq e'+(t-1)e$ combined with $p\mid n$ (Lemma~\ref{lem:dbna}) and $p\mid e$ (Remark~\ref{remark:ea0}). The result for $t=0$ is~\eqref{surj2}.
\end{enumerate}
\end{remark}

We begin the proof by noting that Theorem~\ref{thm:onelemma}\eqref{onelemma_reg} follows from Theorem~\ref{thm:7.1} applied to the local ring $\O_{\X,P_0}$. (Note that $\A$ lying in $\Br X\subset \H^2(K)$ is equivalent to $\A$ being unramified on $X$ in the sense of \S~\ref{sec:unram} by \cite[Corollary~2.5]{K} and purity, see \cite[Theorem~3.7.3]{CTSbook}, for example.)

Now let $f \colon \widetilde{\X} \to \X$ be the blowup of $\X$ at $P_0$ and let $E$ be the exceptional divisor, isomorphic to $\P^m_\F$. Let $\widetilde{Y},\widetilde{P},\widetilde{Q},\widetilde{R}$ denote the strict transforms of $Y,P,Q,R$ and let $Z = E \cap \widetilde{Y}$. Let $\widetilde{P}_0,\widetilde{Q}_0,\widetilde{R}_0\in E(\F)$ be the reductions of $\widetilde{P},\widetilde{Q},\widetilde{R}$, respectively. Let $K$ denote the function field of $X$. 
Let $\m$ be the maximal ideal of $\O_{\X,P_0}$.  Extend $\pi$ to a basis $\pi, x_1, \dotsc, x_m$ for $\m/\m^2$.
Then the restrictions of $\pi^{(1)}, x_1^{(1)}, \dotsc, x_m^{(1)} \in \H^0(\widetilde{\X}, \O(1))$ give a system of homogeneous coordinates on $E \cong \P^m_\F$, in which $Z \subset E$ is cut out by $\pi^{(1)}=0$.
On $\widetilde{\X}$ we have $\pi x_1^{(1)} = x_1 \pi^{(1)}$, giving $\pi = \pi_{\widetilde{Y}} \pi_E$ where
\[
\pi_{\widetilde{Y}} = \pi^{(1)} / x_1^{(1)}, \qquad \pi_E = f^*(x_1)
\]
are local parameters at the divisors $\widetilde{Y}$ and $E$ respectively.

Our first two results in this section relate the Swan conductor and refined Swan conductor at $E$ of $f^*\A$ to those of $\A$.  We recall the definitions of $\sw_E(f^* \A)$ and $\rsw_{E,n}(f^* \A)$ from Section~\ref{sec:geomswan}. We make use of the homomorphisms 
\begin{align*}
\psi&\colon \Omega^1_{Y/\F}|_{P_0}\to\H^0(Z, \O(1)), \\
\varphi&\colon \Omega^2_{Y/\F}|_{P_0}\to\H^0(Z, \Omega^1_{Z/\F}(2)), \\
\rho&\colon \H^0(E,\Omega^1(\log Z)(Z))\to\H^0(Z, \O(1)),
\end{align*}
defined in \eqref{eq:psi}, \eqref{eq:varphi} and \eqref{eq:rho}, respectively.

\begin{lemma}\label{lem:downby1}
Let $\A \in \Br X$ with $\sw(\A)=n>1$ and $\rsw_{n}(\A)=[\alpha,\beta]_{\pi,n}$.
Then $\sw_E(f^* \A) \le n-1$ and $\rsw_{E,n-1}(f^* \A) = [\alpha_E, \beta_E]_{\pi,n-1}$
where $\beta_E$ is the unique element of $\H^0(E,\Omega^1(\log Z)(Z))$ satisfying $\rho(\beta_E) = -\psi(\beta_{P_0})$. Explicitly, if $\beta_{P_0}=\sum_i a_i dx_i$ for some $a_i\in\F$ then $\beta_E = \sum_i a_i du_i$ where $u_i = x_i^{(1)} / \pi^{(1)}$. In particular, if $\beta_{P_0} \neq 0$ then $\sw_E(f^* \A) = n-1$.
\end{lemma}

\begin{proof}
The blowup $f$ induces an isomorphism $\O_{\X,Y} \to \O_{\widetilde{\X},\widetilde{Y}}$, whereby $\sw_{\widetilde{Y}}(f^*\A) = \sw_Y \A$ and
\[
\rsw_{\widetilde{Y},n}(f^*\A) = \pi^{-n} \Big( f^* \alpha + f^* \beta \wedge \dlog\pi\Big).
\]

Since $\F$ is perfect, $\Omega^1_{\F}=\Omega^2_\F=0$. In particular, this implies that the images of $\alpha$ and $\beta$ in $\Omega^2_\F$ and $\Omega^1_\F$, respectively, are zero. Thus,
$\A$ is not strongly clean with respect to $\O_{\X,P_0}$ (see Definition~\ref{def:7.4}).
Consequently, Theorem~\ref{thm:8.1} shows $\sw_E(f^* \A)<n$.  Write $\sw_E(f^* \A)=n-r$ for some $r\ge 1$.

Writing $\pi = \pi_{\widetilde{Y}} \pi_E$ gives
\begin{equation}\label{eq:EY1}
\rsw_{\widetilde{Y},n}(f^*\A) = (\pi_{\widetilde{Y}}^n \pi_E^{n-r})^{-1} \pi_E^{-r} \Big( f^* \alpha + f^* \beta \wedge \dlog\pi_E + f^* \beta \wedge \dlog\pi_{\widetilde{Y}} \Big).
\end{equation}
Let $R_{\widetilde{Y}}=\O_{\widetilde{\X},Z}/\pi_{\widetilde{Y}}=\O_{\widetilde{Y},Z}$ and $R_E=\O_{\widetilde{\X},Z}/\pi_E=\O_{E,Z}$.
Now Theorem~\ref{thm:7.1} gives $\pi_E^{-r}f^*\beta\in \Omega^1_{R_{\widetilde{Y}}}(\log Z)$. 

Suppose $\beta_{P_0} \neq 0$.
Applying Lemma~\ref{lem:blowup}(\ref{it:swZ}) shows that $r=1$, that is, $\sw_E(f^*\A) = n-1$.
On the other hand, using $\pi$ as a uniformiser on $E$, we have 
\[
\rsw_{E,n-1}(f^*\A) = (\pi^{n-1})^{-1}\Big(\alpha_E+\beta_E\wedge \dlog\pi \Big)
\]
with $(\alpha_E, \beta_E)\in\Omega^2_{\kappa(E)}\oplus\Omega^1_{\kappa(E)}$. 
Writing $\pi=\pi_{\widetilde{Y}}\pi_{E}$ gives
\begin{equation}\label{eq:EY2}
\rsw_{E,n-1}(f^*\A) = (\pi_{\widetilde{Y}}^n\pi_E^{n-1})^{-1}\pi_{\widetilde{Y}}\Big(\alpha_E+\beta_E\wedge \dlog\pi_{\widetilde{Y}}+\beta_E \wedge \dlog\pi_E \Big).
\end{equation}
Applying~Theorem~\ref{thm:7.1} to all local rings of $E$ gives $\beta_E \in \Omega^1_E(\log Z)(Z)$.

We now apply~Proposition~\ref{prop:7.3}, which states that the terms in $\dlog\pi_{\widetilde{Y}} \wedge \dlog\pi_E$ in~\eqref{eq:EY1} and~\eqref{eq:EY2} coincide, after restricting to $\kappa(Z)$.
In other words, the residue of $-\pi_E^{-1} f^*\beta$ along $Z \subset \widetilde{Y}$ and the residue of $\pi_{\widetilde{Y}} \beta_E$ along $Z \subset E$ are equal.
By Lemma~\ref{lem:blowup}(\ref{it:rswZ}) and~Lemma~\ref{lem:logdiff}(\ref{logdiff1}), this is equivalent to $\rho(\beta_E) = -\psi(\beta_{P_0})$.

If, on the other hand, we have $\beta_{P_0}=0$, then either $r=1$ and the above calculation gives $\beta_E = 0$; or $r>1$ and $\alpha_E,\beta_E$ both vanish; in either case, the claimed equation for $\beta_E$ holds.
\end{proof}

Taking $s=0$ in Lemma~\ref{lem:downby2} below shows that when $\beta=0$ the Swan conductor drops further and the refined Swan conductor at $E$ of $f^*\A$ is related to $\alpha_{P_0}$.

\begin{lemma}\label{lem:downby2} Suppose $n>0$ and $p\mid n$. Let $s\geq 0$ and suppose
$\sw(\A) = n-2s$ and $\rsw_{n-2s}(\A)=[\alpha_s,\beta_s]_{\pi,n-2s}$ with $\alpha_s=\sum_{i>j} a_{ij}dx_{i}\wedge dx_{j}$ and $\beta_s=s \sum_{i>j} a_{ij}x_{j}^2 d(x_{i}/x_{j})$ for some $a_{ij}\in\O_{Y,P_0}$.
 Then $\sw_{E}( f^* \A)\leq n-2(s+1)$ and $\rsw_{E,n-2(s+1)}(f^*\A)=[\alpha_E,\beta_E]_{\pi,n-2(s+1)}$ with \begin{align*}
\alpha_E&=\sum_{i>j} a_{ij}(P_0)du_{i}\wedge du_{j}\\
\beta_E&\equiv (s+1)\sum_{i>j} a_{ij}(P_0)u_{j}^2 d(u_{i}/u_{j})\bmod \H^0(E, \Omega^1(\log Z)(Z)).
\end{align*}
In particular, if $\alpha_{s,P_0}\neq 0$ then $\sw_{E}( f^* \A)= n-2(s+1)$.
\end{lemma}

\begin{remark}
By Lemma~\ref{lem:logdiff}\eqref{logdiff2}, the statement concerning $\beta_E$ is equivalent to $\beta_E|_Z=(s+1)\varphi(\alpha_{P_0})$, where $\varphi$ is defined as in~\eqref{eq:varphi}.
\end{remark}

\begin{proof}
By Lemma~\ref{lem:downby1} and its proof, $\sw_{E}(f^* \A)\leq n-2s-1$ and $\rsw_{E,n-2s-1}(f^*\A)=[\alpha',0]_{\pi_E, n-2s-1}$ for some $\alpha'\in\Omega^2_{F}$, since all the $x_{i}$ vanish at $P_{0}$ and hence $\beta_{s,P_{0}}=0$. Suppose for contradiction that $\sw_E f^*\A= n-2s-1$. Then~\eqref{eq:EY1} gives
\begin{equation}\label{eq:EY1again}
\rsw_{\widetilde{Y},n-2s}(f^*\A) = (\pi_{\widetilde{Y}}^{n-2s} \pi_E^{n-2s-1})^{-1} \pi_E^{-1} \Big( f^* \alpha_s + f^* \beta_s \wedge \dlog\pi_E + f^* \beta_s \wedge \dlog\pi_{\widetilde{Y}} \Big).
\end{equation}
Since $\beta_{s,P_{0}}=0$, Lemma~\ref{lem:blowup} shows that $\pi_{E}^{-1}f^*\beta_s\in \Omega^1(\widetilde{Y})$ and $\pi_E^{-1}f^*\alpha_s\in \Omega^2(\widetilde{Y})$, with $(\pi_E^{-1}f^*\alpha_s)|_Z=0$. 
Moreover, \eqref{eq:EY2} becomes 
\begin{equation}\label{eq:EY2again}
\rsw_{E,n-2s-1}(f^*\A) = (\pi_{\widetilde{Y}}^{n-2s}\pi_E^{n-2s-1})^{-1}\pi_{\widetilde{Y}}\alpha'.
\end{equation}
Applying~Theorem~\ref{thm:7.1} to all local rings of $E$ gives $\alpha' \in \Omega^2_E(\log Z)(Z)$. Now we apply~Proposition~\ref{prop:7.3} twice: comparing the $\dlog\pi_E$ terms in~\eqref{eq:EY1again} and~\eqref{eq:EY2again} gives $(\pi_E^{-1}f^*\beta_s)|_Z=0$; comparing the remaining terms gives $\alpha' \in \Omega^2(E)$. Since $E\cong \P^m_\F$, we have $\Omega^2(E)=0$ and hence $\alpha'=0$. Therefore, $\rsw_{E,n-2s-1}(f^*\A) =0$, which contradicts the assumption that $\sw_E f^*\A= n-2s-1$. Hence $\sw_Ef^*\A\leq n-2(s+1)$. 

It remains to prove the claims concerning the refined Swan conductor at level $n-2(s+1)$. Write $\sw_Ef^*\A= n-2(s+1)-r$ for some $r\geq 0$. Then~\eqref{eq:EY1} gives
\begin{multline}\label{eq:EY1again2}
\rsw_{\widetilde{Y},n-2s}(f^*\A) \\ = (\pi_{\widetilde{Y}}^{n-2s} \pi_E^{n-2(s+1)-r})^{-1} \pi_E^{-2-r} \Big( f^* \alpha_s + f^* \beta_s \wedge \dlog\pi_E + f^* \beta_s \wedge \dlog\pi_{\widetilde{Y}} \Big).
\end{multline}
By~Theorem~\ref{thm:7.1}, $\pi_E^{-2-r}  f^* \beta_s \in \Omega^1_{\O_{\widetilde{Y},Z}}(\log Z)$ and $\pi_E^{-2-r}  f^* \alpha_s \in \Omega^2_{\O_{\widetilde{Y},Z}}(\log Z)$, whereupon Lemma~\ref{lem:blowup} shows that if $r>0$ then $\alpha_{s, P_0}= 0$. This implies that all the $a_{ij}(P_0)$ are zero, which proves the claimed equality in the case $r>0$, since then $\alpha_E$ and $\beta_E$ are zero by definition of the refined Swan conductor at level $n-2(s+1)$.

Let us assume henceforth that $\alpha_{s, P_0}\neq 0$ and hence $\sw_E(f^*\A)= n-2(s+1)$. Now~\eqref{eq:EY2} becomes
\begin{multline}\label{eq:EY2again2}
\rsw_{E,n-2(s+1)}(f^*\A) \\ = (\pi_{\widetilde{Y}}^{n-2s}\pi_E^{n-2(s+1)})^{-1}\pi_{\widetilde{Y}}^2\Big(\alpha_E+\beta_E\wedge \dlog\pi_{\widetilde{Y}}+\beta_E \wedge \dlog\pi_E \Big).
\end{multline}
Applying~Theorem~\ref{thm:7.1} to all local rings of $E$ gives $\beta_E\in \Omega^1_E(\log Z)(2Z)$ and $\alpha_E \in \Omega^2_E(\log Z)(2Z)$. Now we apply~Proposition~\ref{prop:7.3} several times. Comparing the $\dlog\pi_E\wedge \dlog\pi_{\widetilde{Y}}$ terms in~\eqref{eq:EY1again2} and~\eqref{eq:EY2again2} shows that the residue of $\pi_{\widetilde{Y}}^2\beta_E$ along $Z\subset E$ equals the residue of $-\pi_E^{-2}f^*\beta_s$ along $Z\subset \widetilde{Y}$. We now show that this residue is zero. Recalling that $f^*x_1=\pi_E$, we have
\begin{equation}
-\pi_E^{-2}f^*\beta_s=s\sum_{i>j}f^*a_{ij}u_i^2d(u_j/u_i)=s\sum_{i>j}f^*a_{ij}(u_idu_j-u_jdu_i),
\end{equation}
which has zero residue along $Z\subset \widetilde{Y}$. Therefore, $\pi_E^{-2}f^*\beta_s\in \Omega^1_{\O_{\widetilde{Y},Z}}$ and $\beta_E\in \Omega^1_E(2Z)$. Comparing the $\dlog\pi_E$ terms in~\eqref{eq:EY1again2} and~\eqref{eq:EY2again2} shows that $(\pi_{\widetilde{Y}}^2\beta_E)|_Z$ equals the sum of the residue of $\pi_E^{-2}  f^* \alpha_s$ along $Z\subset \widetilde{Y}$
and $\pi_E^{-2}f^*\beta_s|_Z$. Lemma~\ref{lem:blowup} shows that the residue of $\pi_E^{-2}  f^* \alpha_s$ along $Z\subset \widetilde{Y}$ is 
\[\varphi(\alpha_{s,P_0})/(x^{(1)})^2=\varphi\Bigl(\sum_{i>j}a_{ij}(P_0)dx_i\wedge dx_j\Bigr)/(x^{(1)})^2.\] 
Moreover, 
\[\pi_E^{-2}f^*\beta_s|_Z=s\varphi\Bigl(\sum_{i>j}a_{ij}(P_0)dx_i\wedge dx_j\Bigr)/(x^{(1)})^2.\]
Hence, $(\pi_{\widetilde{Y}}^2\beta_E)|_Z=(s+1)\varphi\Bigl(\sum_{i>j}a_{ij}(P_0)dx_i\wedge dx_j\Bigr)/(x^{(1)})^2$. Multiplying by $(x^{(1)})^2$ shows that 
\[\beta_E|_Z=(s+1)\varphi\Bigl(\sum_{i>j}a_{ij}(P_0)dx_i\wedge dx_j\Bigr).\] 
By Lemma~\ref{lem:logdiff}\eqref{logdiff2}, this proves the claim regarding $\beta_E$.

Now compare the $\dlog\pi_{\widetilde{Y}}$ terms in~\eqref{eq:EY1again2} and~\eqref{eq:EY2again2} to see that $\pi_E^{-2}f^*\beta_s|_Z$ is equal to the sum of the residue of $\pi_{\widetilde{Y}}^2\alpha_E$ along $Z\subset E$ and $(\pi_{\widetilde{Y}}^2\beta_E)|_Z$. By our calculations above, this implies that the residue of $\pi_{\widetilde{Y}}^2\alpha_E$ along $Z\subset E$ equals $-\varphi\Bigl(\sum_{i>j}a_{ij}(P_0)dx_i\wedge dx_j\Bigr)/(x^{(1)})^2$. Now Lemma~\ref{lem:logdiff}\eqref{logdiff3} shows that this residue is equal to $\rho(\alpha_E)/(x^{(1)})^2$ and hence $\alpha_E=\sum_{i>j}a_{ij}(P_0)du_i\wedge du_j$.
\end{proof}

The next two results deal with the endgame, where the Swan conductor at $E$ is zero and our task is to compute the residue $\partial_E(f^*\A)$.

\begin{lemma}\label{lem:sw0}
Let $\A \in \Br X$ with $\sw(\A)=1$ and $\rsw_{1}(\A)=[\alpha,\beta]_{\pi,1}$.
Then $\sw_E(f^*\A)=0$.
Let $F \in \H^0(E,\O(1))$ be any linear form restricting to $\psi(\beta_{P_0})$ on $Z$.
The residue $\partial_E(f^*\A) \in \H^1(\kappa(E),\Q/\Z)$ lies in $\H^1(E \setminus Z, \Q/\Z)$ and coincides, modulo $\H^1(\F,\Q/\Z)$, with the image of $F/\pi^{(1)}$ under the map \[\kappa(E) \to \H^1(\kappa(E),\Z/p)\xrightarrow{\frac{1}{p}} \H^1(\kappa(E),\Q/\Z)\] induced by the Artin--Schreier sequence.
\end{lemma}

\begin{proof}
We begin as in the proof of Lemma~\ref{lem:downby1} to obtain $\sw_E(f^*\A)=0$ and
\begin{equation*}
\rsw_{\widetilde{Y},1}(f^*\A) = (\pi_{\widetilde{Y}})^{-1} \pi_E^{-1} \Big( f^* \alpha + f^* \beta \wedge \dlog\pi_E + f^* \beta \wedge \dlog\pi_{\widetilde{Y}}\Big).
\end{equation*}
Write $\chi_2 = \partial_E(f^*\A) \in \H^1(\kappa(E))$.
By Proposition~\ref{prop:tame}, $\chi_2$ lies in the subgroup $\H^1(E\setminus Z, \Q/\Z)$. 
Denote by $\kappa(E)_Z$ the field of fractions of the Henselisation of $\O_{E,Z}$.
To determine $\chi_2$, we consider Kato's filtration on $\H^1(\kappa(E)_Z)$.
Lemma~\ref{lem:new7.1} gives $\sw_Z(\chi_2) \le 1$, and Lemma~\ref{lem:new7.3} gives
\[
\rsw_{Z,1}(\chi_2) = [\alpha' , \beta' ]_{\bar{\pi}_{\widetilde{Y}},1}
\]
in which $\beta'$ is equal to the residue of $\pi_E^{-1} f^* \beta$ at $Z$, that is, 
\[
\beta' = \psi(\beta_{P_0})/x_1^{(1)}
\]
by Lemma~\ref{lem:blowup}.
The proof is completed by Lemma~\ref{lem:H1EminusZ}.
\end{proof}

\begin{lemma}\label{lem:sw02}
Suppose $p=2$. Let $\A \in \Br X$ with $\sw(\A)=2$ and $\rsw_{2}(\A)=[\alpha,0]_{\pi,2}$.
Then $\sw_E(f^*\A)=0$. Furthermore, let $\kappa(E)_Z$ be the field of fractions of the Henselisation of $\O_{E,Z}$ and and consider the filtration on $\H^1(E\setminus Z,\Q/\Z)$ obtained by pulling back $\{\fil_n \H^1(\kappa(E)_Z )\}$. 
Let $F \in \H^0(E,\O(2))$ be any quadratic form such that $dF$ restricts to $\varphi(\alpha_{P_0})$ on $Z$. 
The residue $\partial_E(f^*\A) \in \H^1(\kappa(E),\Q/\Z)$ lies in $\H^1(E \setminus Z, \Q/\Z)$ and coincides, modulo $\fil_1\H^1(E\setminus Z,\Q/\Z)$, with the image of $F/(\pi^{(1)})^2$ under the map $\kappa(E) \to \H^1(\kappa(E),\Z/p)\xrightarrow{\frac{1}{p}} \H^1(\kappa(E),\Q/\Z)$ induced by the Artin--Schreier sequence.
\end{lemma}

\begin{proof}
We begin as in the proof of Lemma~\ref{lem:downby2} to obtain $\sw_E(f^*\A)=0$ and
\begin{equation*}
\rsw_{\widetilde{Y},2}(f^*\A) =  \pi_{\widetilde{Y}}^{-2}  \pi_E^{-2} f^* \alpha=[\pi_E^{-2} f^* \alpha, 0]_{\pi_{\widetilde{Y}},2}.
\end{equation*}
Write $\chi_2 = \partial_E(f^*\A) \in \H^1(\kappa(E))$. By Proposition~\ref{prop:tame}, $\chi_2\in \H^1(E\setminus Z, \Q/\Z)$. Lemma~\ref{lem:new7.1} gives $\sw_Z(\chi_2) \le 2$, and Lemma~\ref{lem:new7.3} gives
\begin{align}\label{eq:rswchi22}
\rsw_{Z,2}(\chi_2) = [-\bar{\omega}, 0 ]_{\bar{\pi}_{\widetilde{Y}},2}
\end{align}
in which $\bar{\omega}$ is equal to the residue of $\pi_E^{-2} f^* \alpha$ at $Z$, that is
\begin{align}\label{eq:omegares}
\bar{\omega}=\varphi(\alpha_{P_0})/(x_1^{(1)})^2
\end{align}
by Lemma~\ref{lem:blowup}. 

Let $F \in \H^0(E,\O(2))$ be any quadratic form such that $dF$ restricts to $\varphi(\alpha_{P_0})$ on $Z$. 
To see that such an $F$ exists, write $\alpha = \sum_{i > j} a_{ij} dx_i \wedge dx_j$ with $a_{ij} \in \O_{Y,P_0}$. Then 
\begin{align*}
\varphi(\alpha_{P_0})&=\sum_{i > j} a_{ij}(P_0)(x_j^{(1)})^2 d(x_i^{(1)} / x_j^{(1)})\\
&=\sum_{i > j} a_{ij}(P_0) d(x_i^{(1)} x_j^{(1)})
\end{align*}
since we are in characteristic $p=2$. 

Now let $\delta_1: \kappa(E)_Z \to \H^1_p(\kappa(E)_Z)$ denote the Artin--Schreier map. Note that $\delta_1(F/(\pi^{(1)})^2)$ is unramified outside $Z$ and hence $\frac{1}{p}\delta_1(F/(\pi^{(1)})^2)\in \H^1(E\setminus Z,\Q/\Z)$.
By \cite[Lemmas~3.6 and~3.7]{K}, $\frac{1}{p}\delta_1(F/(\pi^{(1)})^2)\in \H^1(\kappa(E)_Z)$ has refined Swan conductor
\begin{align}\label{eq:rsw3.7}
\rsw_{Z,2}\bigl(\delta_1(F/(\pi^{(1)})^2)\bigr)=[-dF/(x_1^{(1)})^2,0]_{\pi_{\widetilde{Y}},2}=[-\varphi(\alpha_{P_0})/(x_1^{(1)})^2,0]_{\pi_{\widetilde{Y}},2}.
\end{align}
The result now follows from \eqref{eq:rswchi22}, \eqref{eq:omegares}, \eqref{eq:rsw3.7} and the injectivity of the refined Swan conductor on the graded pieces of Kato's filtration by Swan conductor.
\end{proof}

We now prove Theorem~\ref{thm:onelemma}\eqref{onelemma_1} in the case $n=1$. This will form the basis for a proof of Theorem~\ref{thm:onelemma}\eqref{onelemma_1} by induction.

\begin{lemma}\label{lem:B1}
Let $\A\in\fil_1 \Br X$, and let $\rsw_{1}(\A) = [\alpha, \beta]_{\pi, 1}$.
Let $P\in \X(\O_k)$, and let $P_0 \in Y(\F)$ be the reduction of $P$. For $Q \in B(P,1)$,
\[
\inv \A(Q) = \inv \A(P) + \frac{1}{p}\Tr_{\F/\F_p} \beta_{P_0}(\tv{1}{P}{Q}).
\]
In particular, if $\beta_{P_0}\neq 0$ then $\evmap{\A}$ takes $p$ distinct values on $B(P,1)$.
\end{lemma}

\begin{proof}
We have
\begin{align*}
\inv \A(Q) - \inv \A(P) = \inv A,
\end{align*}
where $A=f^*\A(\widetilde{Q}) - f^*\A(\widetilde{P})\in \Br k$. By Lemma~\ref{lem:sw0}, Proposition~\ref{prop:tame} and Lemma~\ref{lem:inv_trace}, $A\in\Br k[p]$ and
\[\inv A=\frac{1}{p}\Tr_{\F/\F_p}\Bigl(F/\pi^{(1)}(\widetilde{Q}_0)-F/\pi^{(1)}(\widetilde{P}_0)\Bigr),\]
where $F \in \H^0(E,\O(1))$ is any linear form restricting to $\psi(\beta_{P_0})$ on $Z$. Write $\beta_{P_0}=\sum_i a_i dx_i$. Then we can take $F=\sum_i a_i x_i^{(1)}$ and hence $F/\pi^{(1)}=\sum_i a_i u_i$, where $u_i=x_i^{(1)}/\pi^{(1)}$.
Since $u_i(\widetilde{P}_0)=0$ for $1\leq i\leq m$, we have $F/\pi^{(1)}(\widetilde{P}_0)=0$ and 
\[\inv A=\frac{1}{p}\Tr_{\F/\F_p}\Bigl(F/\pi^{(1)}(\widetilde{Q}_0)\Bigr).\]
Now the result follows from Lemma~\ref{lem:blowuptv}\eqref{blowuptv1}.
\end{proof}

We now prove Theorem~\ref{thm:onelemma}\eqref{onelemma_1} by induction. 

\begin{proof}[Proof of Theorem~\ref{thm:onelemma}\eqref{onelemma_1}]
Let $N\geq 1$ and suppose that we have proved Theorem~\ref{thm:onelemma}\eqref{onelemma_1} for all $n\leq N$. Our task is to prove it for $n=N+1$. Suppose that $\A\in\fil_{N+1}\Br X$, $\rsw_{N+1}\A=[\alpha,\beta]_{\pi, N+1}$, $Q\in B(P, N+1)$. 
By Lemma~\ref{lem:downby1}, $\sw_E(f^* \A) \le N$ and $\rsw_{E,N}(f^* \A) = [\alpha_E, \beta_E]_{\pi,N}$
where $\beta_E$ is the unique element of $\H^0(E,\Omega^1(\log Z)(Z))$ satisfying $\rho(\beta_E) = -\psi(\beta_{P_0})$. Now we apply Theorem~\ref{thm:onelemma}\eqref{onelemma_1} to $f^*\A\in \Br X$, $\widetilde{P}\in\widetilde{\X}(\O_k), \widetilde{Q}\in B(\widetilde{P}, N)$ to obtain
\begin{align}
\inv f^*\A(\widetilde{Q}) = \inv f^* \A(\widetilde{P}) + \frac{1}{p}\Tr_{\F/\F_p} \Bigl(\beta_{E,\widetilde{P}_0}\Bigl(\tv{N}{\widetilde{P}}{\widetilde{Q}}\Bigr)\Bigr).
\end{align}
Lemma~\ref{lem:blowuptv}\eqref{blowupdu} shows that $\beta_{E,\widetilde{P}_0}\Bigl(\tv{N}{\widetilde{P}}{\widetilde{Q}}\Bigr)=\beta_{P_0}\Bigl(\tv{N+1}{P}{Q}\Bigr)$.
Noting that $f^*\A(\widetilde{Q})=\A(Q)$ and $f^*\A(\widetilde{P})=\A(P)$ completes the proof of Theorem~\ref{thm:onelemma}\eqref{onelemma_1}.
\end{proof}

Now we turn our attention to the proof of Theorem~\ref{thm:onelemma}\eqref{rla} and~\eqref{surj2}. We begin with the first statement of Theorem~\ref{thm:onelemma}\eqref{rla}, which is the content of the next lemma.

\begin{lemma}\label{lem:B2s1}
Let $n>2$ and let $\A\in\fil_n\Br X$ with  $\rsw_{n}\A=[\alpha,0]_{\pi,n}$. There exists $\gamma\in\Omega^1_Y|_{P_0}$ such that the following holds: for any $Q\in B(P,1)$ and $R \in B(Q,n-1)$, we have
\[
\inv \A(R) = \inv \A(Q) - \frac{1}{p} \Tr_{\F/\F_p} \Bigl(\alpha_{P_0}(\tv{1}{P}{Q},\tv{n-1}{Q}{R}) + \gamma(\tv{n-1}{Q}{R})\Bigr).
\]
Furthermore, if $\alpha_{P_0}\neq 0$ then there exists $Q\in B(P,1)$ such that $\evmap{\A}$ takes $p$ distinct values on $B(Q,n-1)$.
\end{lemma}

\begin{proof}
By Lemma~\ref{lem:downby2},  $\sw_E(f^* \A) \leq n-2$ and $\rsw_{E,n-2}(f^* \A) = [\alpha_E, \beta_E]_{\pi,n-2}$
where $\beta_E \in \H^0(E,\Omega^1(2Z))$ satisfies $\beta_E|_Z = \varphi(\alpha_{P_0})$. Writing $\alpha_{P_0} = \sum a_{ij} dx_i \wedge dx_j$ for some $a_{ij}\in\F$ and letting $u_i = x_i^{(1)} / \pi^{(1)}$, this means $\beta_E = \beta_1+\beta_2$ where $\beta_1=\sum a_{ij} u_j^2 d(u_i/u_j)$ and $\beta_2\in \H^0(E,\Omega^1(\log Z)(Z))$. Lemma~\ref{lem:logdiff}\eqref{logdiff1} shows we can write $\beta_2=\sum b_i du_i$ for some $b_i\in\F$. Note that $\widetilde{R}\in B(\widetilde{Q},n-2)$. Now Theorem~\ref{thm:onelemma}\eqref{onelemma_1} gives
\[
\inv \A(R) - \inv \A(Q) = \inv f^*\A(\widetilde{R}) - \inv f^*\A(\widetilde{Q}) = \frac{1}{p} \Tr_{\F/\F_p} \beta_{E,\widetilde{Q}_0}(\tv{n-2}{\widetilde{Q}}{\widetilde{R}}).
\]
It remains to check that $\beta_{1,\widetilde{Q}_0}(\tv{n-2}{\widetilde{Q}}{\widetilde{R}})=-\alpha_{P_0}(\tv{1}{P}{Q},\tv{n-1}{Q}{R})$ and set $\gamma=-\sum b_i dx_i\in \Omega^1_Y|_{P_0}$. We have
\[\beta_{1,\widetilde{Q}_0}(\tv{n-2}{\widetilde{Q}}{\widetilde{R}})=\sum a_{ij} \big( u_j(\widetilde{Q}_0) du_i(\tv{n-2}{\widetilde{Q}}{\widetilde{R}})-u_i(\widetilde{Q}_0)du_j(\tv{n-2}{\widetilde{Q}}{\widetilde{R}}) \big) \]
Now the claimed formula follows from Lemma~\ref{lem:blowuptv}.

For the second claim, it suffices to show the existence of tangent vectors $v,w\in T_{P_0}(Y)$ such that $\Tr_{\F/\F_p} \Bigl(\alpha_{P_0}(v,w) + \gamma(w)\Bigr)\neq 0$, since one can then multiply $w$ by scalars in $\F_p$. Since $\alpha_{P_0}\neq 0$, it is easily seen that there exist $v',w'\in T_{P_0}(Y)$ such that $\alpha_{P_0}(v',w') + \gamma(w')\neq 0$. The non-degeneracy of the trace form implies that there exists $\lambda\in \F$ such that 
\begin{align*}
\Tr_{\F/\F_p} \Bigl(\lambda\bigl(\alpha_{P_0}(v',w') + \gamma(w')\bigr)\Bigr)\neq 0.
\end{align*}
Now take $v=v'$ and $w=\lambda w'$.
\end{proof}

The following lemma for the case $p=n=2$ completes the proof of Theorem~\ref{thm:onelemma}\eqref{surj2}.

\begin{lemma}\label{lem:p=n=2}
Suppose $p=2$.
Let $\A\in\fil_2 \Br X$, and suppose that $\rsw_{2}(\A) = [\alpha, 0]_{\pi, 2}$. Write $\alpha_{P_0}=\sum a_{ij} dx_i\wedge dx_j$ for some $a_{ij}\in\F$.
Let $P\in \X(\O_k)$, and let $P_0 \in Y(\F)$ be the reduction of $P$. There exists $\gamma\in\Omega^1_Y|_{P_0}$ such that for $Q \in B(P,1)$,
\[
\inv \A(Q) = \inv \A(P) + \frac{1}{2}\Tr_{\F/\F_2}\Bigl(\sum a_{ij}dx_i(\tv{1}{P}{Q})dx_j(\tv{1}{P}{Q}) +\gamma(\tv{1}{P}{Q}) \Bigr).
\]
Furthermore, if $\alpha_{P_0}\neq 0$ then $\evmap{\A}$ takes $2$ distinct values on $B(P,1)$.
\end{lemma}

\begin{proof}
We have
\begin{align*}
\inv \A(Q) - \inv \A(P) = \inv A,
\end{align*}
where $A=f^*\A(\widetilde{Q}) - f^*\A(\widetilde{P})\in \Br k$. By Proposition~\ref{prop:tame}, Lemma~\ref{lem:sw02} and Lemma~\ref{lem:H1EminusZ}, $A\in\Br k[2]$ and
\begin{align*}
\partial (A) &= \partial_E (f^*\A)(\widetilde{Q}_0) -  \partial_E (f^*\A)(\widetilde{P}_0)\\
&= \frac{1}{2}\delta_1\bigl(F/(\pi^{(1)})^2+G/\pi^{(1)}\bigr)(\widetilde{Q}_0)- \frac{1}{2}\delta_1\bigl(F/(\pi^{(1)})^2+G/\pi^{(1)}\bigr)(\widetilde{P}_0)
\end{align*}
where $F \in \H^0(E,\O(2))$ is any quadratic form such that $dF$ restricts to $\varphi(\alpha_{P_0})$ on $Z$, $G\in  \H^0(E,\O(1))$ is a linear form, and $\delta_1:\kappa(E)_Z\to \H^1_p(\kappa(E)_Z)$ is the Artin--Schreier map. Recalling that the characteristic is $2$, we can take $F=\sum a_{ij} x_i^{(1)}x_j^{(1)}$, whereby $F/(\pi^{(1)})^2=\sum a_{ij} u_iu_j$.
Now Lemma~\ref{lem:inv_trace} gives
\begin{align*}
\inv (A) &= \frac{1}{2}\Tr_{\F/\F_2}\Bigl(\bigl(F/(\pi^{(1)})^2+G/\pi^{(1)}\bigr)(\widetilde{Q}_0)- \bigl(F/(\pi^{(1)})^2+G/\pi^{(1)}\bigr)(\widetilde{P}_0)\Bigr)\\
&=\frac{1}{2}\Tr_{\F/\F_2}\Bigl(\bigl(F/(\pi^{(1)})^2+G/\pi^{(1)}\bigr)(\widetilde{Q}_0)\Bigr),
\end{align*}
since $u_i(\widetilde{P}_0)=0$ for $1\leq i\leq m$ and therefore $(F/(\pi^{(1)})^2)(\widetilde{P}_0)=(G/\pi^{(1)})(\widetilde{P}_0)=0$.
Recall from Lemma~\ref{lem:blowuptv}\eqref{blowuptv1} that $u_i(\widetilde{Q}_0)=dx_i(\tv{1}{P}{Q})$. By Lemma~\ref{lem:H1EminusZ}, we can write $G=\sum b_i x_i^{(1)}$ for some $b_i\in\F$. Hence we set $\gamma=\sum b_i dx_i\in \Omega^1_Y|_{P_0}$.

For the second claim, let $Q(v)=\sum a_{ij}dx_i(v)dx_j(v)$ and let the associated bilinear form be 
\[B(v,w)=Q(v+w)-Q(v)-Q(w)=\sum a_{ij}(dx_i(v)dx_j(w)+dx_i(w)dx_j(v)).\]
Note that $B(v,w)=(Q(v+w)+\gamma(v+w))-(Q(v)+\gamma(v))-(Q(w)+\gamma(w))$. It suffices to show the existence of $v,w\in T_{P_0}(Y)$ such that $\Tr_{\F/\F_2}(B(v,w))\neq 0$, since then at least one of $\Tr_{\F/\F_2}(Q(v+w)+\gamma(v+w)), \Tr_{\F/\F_2}(Q(v)+\gamma(v)), \Tr_{\F/\F_2}(Q(w)+\gamma(w))$ is non-zero. Since $\alpha_{P_0}\neq 0$, there exist $v',w'\in T_{P_0}(Y)$ such that $B(v',w')\neq 0$. Now the non-degeneracy of the trace form shows the existence of $\lambda\in\F$ such that $\Tr_{\F/\F_2}(\lambda B(v',w'))\neq 0$, whence the result.
\end{proof}

Now we complete the proof of Theorem~\ref{thm:onelemma}\eqref{rla}.

\begin{proof}[Proof of Theorem~\ref{thm:onelemma}\eqref{rla}]\label{onelemma-proof}
The first statement of Theorem~\ref{thm:onelemma}\eqref{rla} has been proved in Lemma~\ref{lem:B2s1}. We will now prove the second statement. Note that if $\alpha_{P_0}=0$ then the second statement follows from the first by setting $g_{\A, Q}(R)$ equal to $\inv\A(R)-\inv\A(Q)$. Henceforth we will assume $\alpha_{P_0}\neq 0$.

We first introduce notation for successive blowups. Let $\X_1\xrightarrow{f_1} \X$ be the blowup of $\X$ at $P_0$ with exceptional divisor $E_1$ and let $Z_1=E_1\cap\widetilde{Y}$. Let $P_1\in\X_1(\O_k)$ be the section lifting $P$ and let $\X_2\xrightarrow{f_2} \X_1$ be the blowup at the closed point $P_{1,0}$ of $P_1$. Iterating this construction gives a sequence of blowups
\[\dots \xrightarrow{f_4} \X_3\xrightarrow{f_3}\X_2\xrightarrow{f_2} \X_1\xrightarrow{f_1} \X.\]
Write $E_i$ for the exceptional divisor of the $i$th blowup and let $P_i, Q_i, R_i\in \X_i(\O_k)$ be the sections lifting $P,Q,R$, respectively. Let $\m$ be the maximal ideal of $\O_{\X, P_{0}}$ and let $\pi, x_1,\dots, x_m$ be a basis for $\m/\m^2$. Let $u_{1,j}=x_j^{(1)}/\pi^{(1)}$ so that $u_{1,1},\dots ,u_{1, m}$ restrict to a system of affine coordinates on $E_1\setminus Z_1\cong \mathbb{A}^m_\F$. Now let $u_{2,j}=u_{1,j}^{(1)}/\pi^{(1)}$ so that $u_{2,1},\dots ,u_{2,m}$ restrict to a system of affine coordinates on $E_2\setminus Z_2$, and so on and so forth.

Write $\alpha=\sum_{i>j} a_{ij} dx_i\wedge dx_j$ for some $a_{ij}\in \O_{Y,P_0}$.
By Lemma~\ref{lem:downby2}, we have $\sw_{E_1}(f_1^*\A)=n-2$ 
and, furthermore, $\rsw_{E_1, n-2}(f_1^*\A)=[\alpha_{E_1},\beta_{E_1}+\beta'_{E_1}]_{\pi, n-2}$ where $\alpha_{E_1}=\sum_{i>j} a_{ij}(P_0)du_{1,i}\wedge du_{1,j}$, $\beta_{E_1}=\sum_{i>j} a_{ij}(P_0)u_{1,j}^2 d(u_{1,i}/u_{1,j})$ and $\beta'_{E_1}\in \H^0(E_1, \Omega^1(\log Z_1)(Z_1))$. 

Let $\mathcal{C}_1\in\Br V_1$ (where $V_1$ is the generic fibre of some open neighbourhood $\V_1$ of $P_0$ in $\X$) be the explicit Brauer group element constructed in Lemma~\ref{lem:rsw_exist} below with $\sw_{E_1}(\mathcal{C}_1)\leq n-2$ and $\rsw_{E_1,n-2}(\mathcal{C}_1)=[0,\beta'_{E_1}]_{\pi,n-2}$. 
(To apply Lemma~\ref{lem:rsw_exist} we must check that $n-2$ is not of the form $e'+\ell e$ for any integer $\ell$.  If $n-2$ were of this form, then we should have $e' \in \Z$ and hence $p \mid e'$.  But we also have $p \mid n$ by Lemma~\ref{lem:dbna} and $p \mid e$ by Remark~\ref{remark:ea0} so this would give $p \mid 2$, a contradiction.)
Let $\mathcal{B}_1=f_1^* \A-\mathcal{C}_1\in\Br V_1$ and note that $\rsw_{E_1,n-2}(\mathcal{B}_1)=[\alpha_{E_1}, \beta_{E_1}]_{\pi,n-2}$.

Now we blow up at $P_{1,0}=Q_{1,0}$. Applying Lemma~\ref{lem:downby2} to $\mathcal{B}_1$ and repeating the argument above, we can write $f_2^*\mathcal{B}_1=\mathcal{B}_2+\mathcal{C}_2$ in $\Br V_2$, where $V_2$ is the generic fibre of some open neighbourhood $\V_2$ of $P_{1,0}$ in $\X_1$,
$\sw_{E_2}(\mathcal{B}_2)=n-4$, $\sw_{E_2}(\mathcal{C}_2)\leq n-4$, $\rsw_{E_2,n-4}(\mathcal{B}_2)=[\alpha_{E_2},\beta_{E_2}]_{\pi,n-4}$ with $\alpha_{E_2}=\sum_{i>j} a_{ij}(P_0)du_{2,i}\wedge du_{2,j}$ and $\beta_{E_2}=2\sum_{i>j} a_{ij}(P_0)u_{2,j}^2 d(u_{2,i}/u_{2,j})$, and $\rsw_{E_2,n-4}(\mathcal{C}_2)=[0,\beta'_{E_2}]_{\pi,n-4}$ for some $\beta'_{E_2}\in  \H^0(E_2, \Omega^1(\log Z_2)(Z_2))$. Thus, 
\[f_2^*f_1^*\A=\mathcal{B}_2+\mathcal{C}_2+f_2^*\mathcal{C}_1.\]
Continuing in this way, after $s$ blowups we obtain an equality
\[f_s^*\dots f_2^*f_1^*\A=\mathcal{B}_s+\mathcal{D}_s\]
of elements in $\Br V_s$, where $V_s$ is the generic fibre of some open neighbourhood $\mathcal{V}_s$ of $P_{s,0}$ in $\X_s$, $\sw_{E_s}(\mathcal{B}_s)= n-2s$ and $\rsw_{E_s,n-2s}(\mathcal{B}_s)=[\alpha_{E_s},\beta_{E_s}]_{\pi,n-2s}$ with $\alpha_{E_s}=\sum_{i>j} a_{ij}(P_0)du_{s,i}\wedge du_{s,j}$ and $\beta_{E_s}=s\sum_{i>j} a_{ij}(P_0)u_{s,j}^2 d(u_{s,i}/u_{s,j})$. We have gathered all the terms coming from the $\mathcal{C}_i$'s into $\mathcal{D}_s$. 

Note that $R_s\in B(Q_s, n-2s)$ so Theorem~\ref{thm:onelemma}\eqref{onelemma_1} gives
\begin{equation}
\inv \mathcal{B}_s(R_s) = \inv \mathcal{B}_s(Q_s) + \frac{1}{p}\Tr_{\F/\F_p} \beta_{E_s, Q_{s,0}}(\tv{n-2s}{Q_s}{R_s}).
\end{equation}
A calculation shows that $\beta_{E_s, Q_{s,0}}(\tv{n-2s}{Q_s}{R_s})=-s\alpha_{P_0}(\tv{s}{P}{Q},\tv{n-s}{Q}{R})$.
Finally, 
\begin{align*}
\inv\A(R)-\inv\A(Q)&=\inv f_s^*\dots f_2^*f_1^*\A(R_s)-\inv f_s^*\dots f_2^*f_1^*\A(Q_s)\\
&=\inv \mathcal{B}_s(R_s) - \inv \mathcal{B}_s(Q_s) +\inv \mathcal{D}_s(R_s) - \inv \mathcal{D}_s(Q_s)\\
&=\frac{-s}{p}\Tr_{\F/\F_p}\alpha_{P_0}(\tv{s}{P}{Q},\tv{n-s}{Q}{R})+\inv \mathcal{D}_s(R_s) - \inv \mathcal{D}_s(Q_s).
\end{align*}
Define $g_{\A,Q}(R)$ to be $\inv \mathcal{D}_s(R_s) - \inv \mathcal{D}_s(Q_s)$. We claim that this does not depend on the choice of a suitable $s$. In other words, suppose that $Q\in B(P,s)$ for some $s$ with $1<s<n/2$. Then we can also consider $Q$ as lying in $B(P,s-1)$. Our claim is that for $R\in B(Q,n-s+1)$ we have $\inv \mathcal{D}_s(R_s) - \inv \mathcal{D}_s(Q_s)=\inv \mathcal{D}_{s-1}(R_{s-1}) - \inv \mathcal{D}_{s-1}(Q_{s-1})$. To prove the claim, write $\mathcal{D}_s=f_s^*\mathcal{D}_{s-1}+\mathcal{C}_s$, with $\sw_{E_s}\mathcal{C}_s\leq n-2s$. For $R\in B(Q,n-s+1)$ we have 
\begin{align*}
\inv \mathcal{D}_s(&R_s) - \inv \mathcal{D}_s(Q_s) \\ &=\inv f_s^*\mathcal{D}_{s-1}(R_s)+\inv \mathcal{C}_s(R_s)-\inv f_s^*\mathcal{D}_{s-1}(Q_s)-\inv \mathcal{C}_s(Q_s)\\
&=\inv\mathcal{D}_{s-1}(R_{s-1})-\inv \mathcal{D}_{s-1}(Q_{s-1})+\inv \mathcal{C}_s(R_s)-\inv \mathcal{C}_s(Q_s).
\end{align*}
For $R\in B(Q,n-s+1)$, we have $R_s\in B(Q_s,n-2s+1)$, whereby $\tv{n-2s}{Q_s}{R_s}=0$. Therefore, Theorem~\ref{thm:onelemma}\eqref{onelemma_1} shows that $\inv \mathcal{C}_s(R_s)=\inv \mathcal{C}_s(Q_s)$, whence the claim.

For $s=1$, we have $\mathcal{D}_1=\mathcal{C}_1$ and Theorem~\ref{thm:onelemma}\eqref{onelemma_1} gives
\begin{align*}
\inv \mathcal{C}_1(R_1) - \inv \mathcal{C}_1(Q_1)=\frac{1}{p}\Tr_{\F/\F_p} \beta'_{E_1, Q_{1,0}}(\tv{n-2}{Q_1}{R_1}).
\end{align*}
Now let $U_1=E_1\setminus Z_1$ and let $\gamma\in\Omega^1_Y|_{P_0}$ be the image of $\beta'_{E_1, Q_{1,0}}$ under the isomorphism $\Omega^1_{U_1}|_{Q_{1,0}}\to \Omega^1_Y|_{P_0}$ (see Section~\ref{sec:tangentblowup}). Then $\beta'_{E_1, Q_{1,0}}(\tv{n-2}{Q_1}{R_1})=\gamma(\tv{n-1}{Q}{R})$, which completes the proof.
\end{proof}

\begin{lemma}\label{lem:rsw_exist}
Let $e$ be the absolute ramification index of $k$, and set $e'=ep/(p-1)$.
Let $\beta\in \H^0(E, \Omega^1(\log Z)(Z))$ and suppose that either $0<n<e'$ or $e'+(t-1)e<n< e'+te$ for some $t\in\Z_{\geq 1}$ with $[k( \mmu_{p^{t+1}}):k]$ coprime to $p$. There exists an open neighbourhood $\V$ of $P_0$ in $\X$ with generic fibre $V$ and an element $\mathcal{C}\in \fil_n\Br V[p^{t+1}]$ such that $\rsw_{E, n}(\mathcal{C})=[0,\beta]_{\pi,n}$. Explicitly, writing $K$ for the function field of $X$, we can take
\[\mathcal{C}=\begin{cases}
\cores_{K(\mmu_{p^{t+1}})/K}\sum_i\Bigl(1+\frac{p^t\tilde{b}_iu_i}{\varepsilon (n-te)\pi^n }(\zeta_p-1)^{p}, u_i\Bigr)_{p^{t+1}} & \textrm{ if } p\nmid n-te,\\
\cores_{K(\mmu_{p^{t+1}})/K}\sum_i\Bigl( 1+\frac{p^t\tilde{b}_i u_i}{\varepsilon \pi^n}(\zeta_p-1)^{p}, \pi \Bigr)_{p^{t+1}} & \textrm{ if } p\mid n-te,
\end{cases}\]
where $\varepsilon = [k(\mmu_{p^{t+1}}):k]$, $\beta=\sum_i b_i du_i$ with $b_i \in \F$, and $\tilde{b}_i$ is an arbitrary lift of $b_i$ to $\O_k$.
\end{lemma}

\begin{proof}
By Lemma~\ref{lem:logdiff}\eqref{logdiff1}, $\beta=\sum_i b_i du_i=d(\sum_i b_i u_i)$ for some $b_i\in\F$, whereby $d\beta=0$. Take arbitrary lifts $\tilde{b}_i$ of the $b_i$ to $k$. Let $k'=k( \mmu_{p^{t+1}})$, let $\varepsilon =[k':k]$ and let $\pi'$ be a uniformiser of $k'$. Let $K'=Kk'$. Let $F$ denote the common residue field of $K$ and $K'$ and let $K^h$ and $(K')^h$ be their respective Henselisations. Let $w=(\pi')^{-\varepsilon }N_{k'/k}\pi'$ and let $\bar{w}\in \F^\times$ be its reduction. The Brauer group element $\mathcal{C}$ will be constructed via corestriction. We begin by proving the following claims.
\begin{enumerate}
\item \label{eq:coresformula}
We have $\cores_{(K')^h/K^h}(\lambda_{\pi'}(\cdot,\cdot))=\lambda_{N_{k'/k}\pi'}(\varepsilon \cdot,\cdot).$
\item \label{eq:corrsw} Suppose $\mathcal{C'}\in\fil_{\varepsilon n}\Br (K')^h$ for $n>0$ and $\rsw_{\varepsilon n}(\mathcal{C}')=[\alpha',\beta']_{\pi',\varepsilon n}$. Then $\cores_{(K')^h/K^h}(\mathcal{C'})\in \fil_n\Br K^h$ and 
\begin{equation}\label{eq:rswC'}
\rsw_{ n}(\cores_{(K')^h/K^h}\mathcal{C}')=[\varepsilon\bar{w}^n\alpha',\bar{w}^n\beta']_{N_{k'/k}\pi',n}.
\end{equation}
\end{enumerate}

To prove~\eqref{eq:coresformula}, consider the isomorphism $\nu':\gr^0(M_{r,\O_{(K')^h}}^{q-1}) \to W_r\Omega^{q-1}_{F,\log} \oplus  W_r\Omega^{q-2}_{F,\log}$ of~\eqref{eq:gr0}, the inverse of which is used to define $\lambda_{\pi'}$. By~\cite[p.~113]{BK}, the sheaf $M_{r,\O_{(K')^h}}^{q-1}$ is generated by symbols and $\nu$ can be defined as follows. Let $x_1, \dots , x_{q-1}\in F^\times$ and let $\tilde{x}_1, \dots , \tilde{x}_{q-1}$ be lifts of the $x_i$ to $\O_{K^h}\subset \O_{(K')^h}$. 
Then 
\begin{align*}
\nu'(\{\tilde{x}_1,\dots, \tilde{x}_{q-1}\})&=(\dlog x_1\wedge\dots\wedge\dlog x_{q-1},0)\ \textrm{and}\\
 \nu'(\{\tilde{x}_1,\dots, \tilde{x}_{q-2},\pi'\})&=(0,\dlog x_1\wedge\dots\wedge\dlog x_{q-2}).
 \end{align*}
   Now let $\nu:\gr^0(M_{r,\O_{K^h}}^{q-1}) \to W_r\Omega^{q-1}_{F,\log} \oplus  W_r\Omega^{q-2}_{F,\log}$ be the isomorphism from~\eqref{eq:gr0} whose inverse is used to define $\lambda_{N_{k'/k}\pi'}$. It is defined by 
   \begin{align*}
   \nu(\{\tilde{x}_1,\dots, \tilde{x}_{q-1}\})&=(\dlog x_1\wedge\dots\wedge\dlog x_{q-1},0)\ \textrm{and}\\
   \nu(\{\tilde{x}_1,\dots, \tilde{x}_{q-2},N_{k'/k}\pi'\})&=(0,\dlog x_1\wedge\dots\wedge\dlog x_{q-2}).
   \end{align*}
     Therefore, $\nu\circ\cores_{(K')^h/K^h}(\{\tilde{x}_1,\dots, \tilde{x}_{q-2},\pi'\})=\nu'(\{\tilde{x}_1,\dots, \tilde{x}_{q-2},\pi'\})$. Since $\tilde{x}_1,\dots, \tilde{x}_{q-1}\in \O_{K^h}$, we have $\cores_{(K')^h/K^h}(\{\tilde{x}_1,\dots, \tilde{x}_{q-1}\})=\varepsilon\{\tilde{x}_1,\dots,\tilde{x}_{q-1}\}$ and hence $\nu\circ\cores_{(K')^h/K^h}(\{\tilde{x}_1,\dots, \tilde{x}_{q-1}\})=\varepsilon\nu'(\{\tilde{x}_1,\dots,\tilde{x}_{q-1}\})$. Therefore, the isomorphisms $\nu,\nu'$ satisfy the following commutative diagram.
     
\[     
\begin{CD}
\gr^0(M^{q-1}_{r,\O_{(K')^h}}) @>{\nu'}>> W_r \Omega^{q-1}_{F,\log} \oplus W_r \Omega^{q-2}_{F,\log} \\
@V{\cores_{(K')^h/K^h}}VV @VV{(\alpha,\beta) \mapsto (\varepsilon\alpha,\beta)}V \\
\gr^0(M^{q-1}_{r,\O_{K^h}}) @>{\nu}>> W_r \Omega^{q-1}_{F,\log} \oplus W_r \Omega^{q-2}_{F,\log} \\
\end{CD}
\]

Now the definitions of $\lambda_\pi$ and $\delta_r$ (Definitions~\ref{def:lambda} and~\ref{def:delta}) yield the commutative diagram below for all $q\geq 2$ and $r\geq 1$
\begin{equation}\label{eq:coresdiag}
\begin{tikzcd}[column sep=huge]
W_r \Omega^{q-1}_{F} \oplus W_r \Omega^{q-2}_{F} \arrow[r,"\delta_r"] \ar[d,"{(\alpha,\beta)\mapsto(\varepsilon\alpha,\beta)}"]
&
\H^q_{p^r}(F) \oplus \H^{q-1}_{p^r}(F) \arrow[r,"\lambda_{\pi'}"]\arrow[d,"{(\alpha,\beta)\mapsto (\varepsilon\alpha,\beta)}"]
& V^q_{p^r}(\O_{(K')^h})\arrow[d, "\cores_{(K')^h/K^h}"] \\
W_r \Omega^{q-1}_{F} \oplus W_r \Omega^{q-2}_{F} \arrow[r,"\delta_r"]
& \H^q_{p^r}(F) \oplus \H^{q-1}_{p^r}(F) \arrow[r,"\lambda_{N_{k'/k}\pi'}"] 
& V^q_{p^r}(\O_{K^h}) 
\end{tikzcd}
\end{equation}
proving~\eqref{eq:coresformula}. 

Now we prove~\eqref{eq:corrsw}. Suppose $\mathcal{C'}\in\fil_{\varepsilon n}\Br (K')^h$ for $n>0$ and $\rsw_{\varepsilon n}(\mathcal{C}')=[\alpha',\beta']_{\pi',\varepsilon n}$. That $\cores_{(K')^h/K^h}(\mathcal{C'})\in \fil_n\Br K^h$ follows from the compatibility of corestriction with cup products and~\cite[Proposition~6.3(ii)]{K}. Recall that the definition of the refined Swan conductor (Definition~\ref{def:rsw}) involves an abuse of notation wherein we write $\lambda_\pi$ for the composition $\lambda_\pi\circ\delta_1$. With this convention, the compatibility of corestriction with cup products and~\eqref{eq:coresformula} yield
\begin{align*}
\{\cores_{(K')^h/K^h}\mathcal{C'},1+(N_{k'/k}\pi')^nT\}&=\cores_{(K')^h/K^h}(\{\mathcal{C}',1+(N_{k'/k}\pi')^nT\})\\
&=\cores_{(K')^h/K^h}(\{\mathcal{C}', 1+(\pi')^{\varepsilon n}w^nT\})\\
&=\cores_{(K')^h/K^h}(\lambda_{\pi'}(\alpha' \bar{w}^nT ,\beta'\bar{w}^nT ))\\
&=\lambda_{N_{k'/k}\pi'}(\varepsilon\alpha' \bar{w}^nT ,\beta'\bar{w}^nT),
\end{align*}
proving~\eqref{eq:corrsw}.

Now let $a=\pi/(\pi')^\varepsilon$ and let $\bar{a}$ denote the image of $a$ in $F^\times$. The definition of the refined Swan conductor (Definition~\ref{def:rsw}) shows $[0,\beta]_{\pi,n}=[0,\bar{w}^n\bar{a}^{-n}\beta]_{N_{k'/k}\pi',n}$. Therefore, by~\eqref{eq:corrsw}, our task is to construct $\mathcal{C}'\in\fil_{\varepsilon n}\Br K'$ with 
\[\rsw_{ \varepsilon n}(\mathcal{C}')=[0,\bar{a}^{-n}\beta]_{\pi',\varepsilon n}.\]
Let $\zeta_{p^{t+1}}\in k'$ be a fixed choice of primitive $p^{t+1}$th root of unity in $k'$. The choice of $\zeta_{p^{t+1}}$ yields choices of primitive $p^s$th roots of unity for $1\leq s\leq t+1$ by setting $\zeta_{p^s}=\zeta_{p^{t+1}}^{p^{t+1-s}}$. For $1\leq s\leq t+1$ and $x,y \in (K')^\times$, let $(x,y)_{p^s} \in \Br K'[p^s]$ denote the class of the corresponding cyclic algebra, which depends on our chosen primitive $p^s$th root of unity.
Alternatively, as in~\cite[\S XIV.2]{Serre}, $(x,y)_{p^s}$ can be constructed as a cup product as follows: let $\delta \colon (K')^\times / (K')^{\times p^s} \to \H^1(K',\mmu_{p^s})$ be the Kummer isomorphism, and take the image of $(\delta(x),\delta(y))$ under the composition
\begin{multline*}
\H^1(K',\mmu_{p^s}) \times \H^1(K',\mmu_{p^s}) \\
\xrightarrow{\cup} \H^2(K',\mmu_{p^s}^{\otimes 2}) \xrightarrow{\cong} \H^2(K', \mmu_{p^s}) \to \H^2(K',(\bar{K'})^\times) = \Br K',
\end{multline*}
where we have used the choice of $\zeta_{p^s}$ to give an isomorphism $\mmu_{p^s}^{\otimes 2} \cong \mmu_{p^s}$. Note that $(x,y)_{p^s}={p^{t+1-s}}(x,y)_{p^{t+1}}$.

First suppose $0<n< e'$ and $p\nmid n$. Let 
\[\mathcal{C'}=\sum_i\Bigl(1+\frac{\tilde{b}_iu_i}{\varepsilon n\pi^n }(\zeta_p-1)^{p}, u_i\Bigr)_p\in \Br K'[p].\]
 By~\cite[Proposition~4.1 and Lemma~4.3]{K} we have $\mathcal{C'}\in\fil_{\varepsilon n}\Br K'$ and 
$\rsw_{ \varepsilon n}(\mathcal{C}')=[0,\bar{a}^{-n}\beta]_{\pi',\varepsilon n},$
as desired.

Now suppose $0<n< e'$ and $p\mid n$. Then~\cite[Proposition~4.1 and Lemma~4.3]{K} show that
\[\mathcal{C'}=\sum_i\Bigl( 1+\frac{\tilde{b}_i u_i}{\varepsilon \pi^n}(\zeta_p-1)^{p}, \pi \Bigr)_p\in \Br K'[p]\]
has the desired Swan conductor and refined Swan conductor.

So far, we have proved the lemma for $0<n<e'$. Now suppose $e'+(t-1)e< n <e'+te$ for some $t\geq 1$. Recall that our task is to construct $\mathcal{C}'\in\fil_{\varepsilon n}\Br K'$ with 
\[\rsw_{\varepsilon n}(\mathcal{C}')=[0,\bar{a}^{-n}\beta]_{\pi',\varepsilon n}.\]
By Lemma~\ref{lem:rsw-same}, it suffices to construct $\mathcal{C}'\in\Br K'$ such that $p^t\mathcal{C}'\in \fil_{\varepsilon(n-te)}\Br K'$ and
\[\rsw_{\varepsilon (n-te)}(p^t\mathcal{C}')=\Bigl[0,\frac{p^t}{(\pi')^{\varepsilon et}}\bar{a}^{-n}\beta\Bigr]_{\pi',\varepsilon (n-te)},\]
where we are abusing notation by writing $p^t/(\pi')^{\varepsilon et}$ for its image in $F^\times$.

If $p\nmid n-te$, let
\[\mathcal{C'}=\sum_i\Bigl(1+\frac{p^t\tilde{b}_iu_i}{\varepsilon (n-te)\pi^n }(\zeta_p-1)^{p}, u_i\Bigr)_{p^{t+1}}\in \Br K'[p^{t+1}].\]
Then
\[p^t\mathcal{C'}=\sum_i\Bigl(1+\frac{p^t\tilde{b}_iu_i}{\varepsilon (n-te)\pi^n }(\zeta_p-1)^{p}, u_i\Bigr)_{p}\in \Br K'[p]\]
and~\cite[Proposition~4.1 and Lemma~4.3]{K} show that $p^t\mathcal{C}'$ has the desired properties.

If $p\mid n-te$, let 
\[\mathcal{C'}=\sum_i\Bigl( 1+\frac{p^t\tilde{b}_i u_i}{\varepsilon \pi^n}(\zeta_p-1)^{p}, \pi \Bigr)_{p^{t+1}}\in \Br K'[p^{t+1}]\]
and apply ~\cite[Proposition~4.1 and Lemma~4.3]{K} to $p^t\mathcal{C}'$.
 \end{proof}

Finally, we prove the surjectivity results of Theorem~\ref{thm:onelemma}\eqref{surj3}.

\begin{proof}[Proof of Theorem~\ref{thm:onelemma}\eqref{surj3}]
First suppose that $\beta_{P_0}\neq 0$. We proceed by induction on $t$. The case $t=0$ follows from Theorem~\ref{thm:onelemma}\eqref{onelemma_1} and Theorem~\ref{thm:onelemma}\eqref{surj2}.
Now suppose the statement of Theorem~\ref{thm:onelemma}\eqref{surj3} holds for some $t_0\in\Z_{\geq 0}$. Let $\A$ have order $p^{t_0+2}$ in $\Br X$ and suppose that either $n> e'+t_0e$ or $n=e'+t_0e$ and $C(\alpha)=C(\beta)=0$. 
By Lemma~\ref{lem:rsw-same}, $\sw(p\A)=n-e$ and $\rsw_{n-e}(p\A)=[\bar{u}\alpha,\bar{u}\beta]_{\pi, n-e}$, where $\bar{u}$ is the image of $p/\pi^e$ in $\F^\times$. 
Consider the following exact sequence coming from multiplication by $p$:
\[
\begin{tikzcd}[column sep=small]
0\arrow[r]&p^{-1}\Z/\Z\arrow[r]& p^{-(t_0+2)}\Z/\Z\arrow[r,"p"]&p^{-(t_0+1)}\Z/\Z\arrow[r]& 0.
\end{tikzcd}
\]

Let $I$ denote the image of $\inv\circ\evmap{\A}:B(P, n-(t_0+1)e)\to p^{-(t_0+2)}\Z/\Z$. By the induction hypothesis, $\inv\circ\evmap{p\A}: B(P, n-(t_0+1)e)\to p^{-(t_0+1)}\Z/\Z$ is surjective. Therefore, for each $x\in p^{-(t_0+1)}\Z/\Z$, $I$ contains at least one preimage of $x$ under multiplication by $p$. Now Theorem~\ref{thm:onelemma}\eqref{onelemma_1} shows that in fact $I$ contains all preimages of elements of $p^{-(t_0+1)}\Z/\Z$ under multiplication by $p$, i.e.~$I=p^{-(t_0+2)}\Z/\Z$, as required.

Now suppose that $\beta=0$, $\alpha_{P_0}\neq 0$, $n>e'+2$ and $n\geq te+3$. Lemma~\ref{lem:B2s1} shows that there exists $\gamma\in \Omega^1_Y|_{P_0}$ such that for $Q\in B(P,1)$ and $R \in B(Q,n-1)$,
\begin{equation}\label{eq:gamma}
\inv \A(R) = \inv \A(Q) - \frac{1}{p} \Tr_{\F/\F_p} \Bigl(\alpha_{P_0}(\tv{1}{P}{Q},\tv{n-1}{Q}{R}) + \gamma(\tv{n-1}{Q}{R})\Bigr).
\end{equation}
The proof of Lemma~\ref{lem:B2s1} shows that $\gamma$ is constructed as follows: let $\rsw_{E,n-2}(f^* \A) = [\alpha_E, \beta_E]_{\pi,n-2}$ and write $\alpha_{P_0}=\sum a_{ij}dx_i\wedge dx_j$ for some $a_{ij}\in\F$. Then 
\[\beta_E-\sum a_{ij}u_j^2d(u_i/u_j)=\sum b_idu_i\] for some $b_i\in\F$ and we set $\gamma=-\sum b_idx_i$. The proof of Lemma~\ref{lem:B2s1} goes on to show the existence of $v,w\in T_{P_0}Y$ such that 
\begin{equation}\label{eq:nonzero}
 \Tr_{\F/\F_p} \Bigl(\alpha_{P_0}(v,w) + \gamma(w)\Bigr)\neq 0.
 \end{equation}
Let $Q_\A\in B(P,1)$ be such that $\tv{1}{P}{Q_\A}=v$. Then $\evmap{\A}$ takes $p$ distinct values on $B(Q_\A,n-1)$. We will show that $\inv\circ\evmap{\A}\colon B(Q_\A, n-te-1)\to p^{-(t+1)}\Z/\Z$ is surjective. The case $t=0$ follows immediately from~\eqref{eq:gamma} and~\eqref{eq:nonzero}. Now suppose we have proved the result for $t_0$ and we want to prove it for $t_0+1$. By Lemma~\ref{lem:rsw-same}, $\sw(p\A)=n-e$ and $\rsw_{n-e}(p\A)=[\bar{u}\alpha,\bar{u}\beta]_{\pi, n-e}$. Applying Lemma~\ref{lem:rsw-same} to $f^* \A$ shows that $\sw(f^* (p\A))=n-2-e$ and $\rsw_{E,n-2-e}(f^* (p\A))=[\bar{u}\alpha_E, \bar{u}\beta_E]_{\pi,n-2-e}$. The construction of $\gamma$ above shows that for $Q\in B(P,1)$ and $R \in B(Q,n-e-1)$
\begin{align*}\label{eq:pgamma}
\inv (p\A)(R) &= \inv (p\A)(Q) - \frac{1}{p} \Tr_{\F/\F_p} \Bigl(\bar{u}\alpha_{P_0}(\tv{1}{P}{Q},\tv{n-e-1}{Q}{R}) +\bar{u} \gamma(\tv{n-e-1}{Q}{R})\Bigr).
\end{align*}
The induction hypothesis states that if $v',w'\in T_{P_0}Y$ are such that 
\begin{equation}\label{eq:pnonzero}
 \Tr_{\F/\F_p} \Bigl(\bar{u}\alpha_{P_0}(v',w') + \bar{u}\gamma(w')\Bigr)\neq 0
 \end{equation}
and $Q_{p\A}\in B(P,1)$ is such that $\tv{1}{P}{Q_{p\A}}=v'$ then $\inv\circ\evmap{p\A}\colon B(Q_{p\A}, n-(t_0+1)e-1)\to p^{-(t_0+1)}\Z/\Z$ is surjective. By~\eqref{eq:nonzero}, we can take $w'=w/\bar{u}$, $v'=v$ and $Q_{p\A}=Q_\A$. Let $I$ denote the image of $\inv\circ\evmap{\A}\colon B(Q_{\A}, n-(t_0+1)e-1)\to p^{-(t_0+2)}\Z/\Z$. The result for $p\A$ shows that for each $x\in p^{-(t_0+1)}\Z/\Z$, $I$ contains at least one preimage of $x$ under multiplication by $p$. To show that $I$ contains all preimages of $x$ under multiplication by $p$, we will show that for any $S\in B(Q_{\A}, n-(t_0+1)e-1)$, the map $\inv\circ\evmap{\A}$ gives a surjection from $B(S, n-1)$ to $p^{-1}\Z/\Z$. Since $n\geq (t_0+1)e+3$, we have $\tv{1}{P}{S}=\tv{1}{P}{Q_\A}=v$, whereby the result follows from~\eqref{eq:gamma} and~\eqref{eq:nonzero}. 
\end{proof}

\section{Proof of Theorem~\ref{thm:main}}\label{sec:main}

We now prove Theorem~\ref{thm:main}. The notation and assumptions of that theorem will be in force throughout this section, so $k$ denotes a finite extension of $\Q_p$, $X/k$ is a smooth, geometrically irreducible variety with smooth model $\X \to\Spec \O_k$, and the special fibre $Y$ of $\X$ is assumed to be geometrically irreducible.

For ease of notation we define a modified version of Kato's filtration as follows.
\begin{align*}
\film_{-2} \Br  \Kh  &=\{\A\in\fil_0\Br  \Kh\mid \partial\A=0 \}; \\
\film_{-1} \Br  \Kh &=\{\A\in\fil_0\Br  \Kh\mid \partial\A\in \H^1(\F,\Q/\Z) \}; \\
\film_{0} \Br  \Kh &=\fil_0\Br K^h; \\
\film_n \Br  \Kh &= \{\A\in \fil_{n+1}\Br  \Kh \mid \rsw_{n+1}(\A)\in [\Omega^2_F, 0]_{\pi,n+1}\} \text{ for } n \ge 1.
\end{align*}

For the purposes of the definition, $ \Kh$ could be replaced by any Henselian discrete valuation field of characteristic zero. Pulling back from $\Br \Kh$ to $\Br X$ gives a filtration on $\Br X$ whose pieces we denote by $\film_n \Br X$.

\begin{lemma}\label{lemma:main1}
For $n\geq -2$, we have  $\film_n \Br X \subset \file_n \Br X$.   
\end{lemma}
\begin{proof}
This follows from Proposition~\ref{prop:tame} for $n=-2,-1,0$ and from Theorem~\ref{thm:onelemma_intro}\eqref{onelemma_intro_1} and Lemma~\ref{lem:rsw-basechange} for $n \ge 1$.
\end{proof}

The reverse inclusions will be given by the following lemmas.

\begin{lemma}\label{lem:main2}
Let $n \ge 1$, and let $\A$ be an element of $\film_n \Br X \setminus \film_{n-1} \Br X$.
Then there exists a finite unramified extension $k'/k$ and $P \in \mathcal{X}(\O_{k'})$ such that $\evmap{\A}$ takes $p$ distinct values on $B(P,n)$. In particular, $\A\notin \file_{n-1} \Br X$. 
\end{lemma}
\begin{proof}
Since $\A\in \film_n \Br X$, we have $\A\in \fil_{n+1}\Br X $ and $\rsw_{n+1}(\A)=[\alpha,0]_{\pi, n+1}$ for some $\alpha\in\Omega^2_F$. By Theorem~\ref{thm:onelemma}\eqref{onelemma_reg}, $\alpha$ lies in $\Omega^2(Y)$. Suppose first that $\alpha\neq 0$.
Let $Z \subset Y$ be the zero locus of $\alpha$, which by assumption is a strict closed subset of $Y$, and set $U = Y \setminus Z$.
By the Lang--Weil estimates~\cite{LW}, there is a finite extension $\F'/\F$ such that $U(\F')$ is non-empty.
Let $k'/k$ be the unramified extension of $k$ having residue field $\F'$.
Choose any $P_0 \in U(\F')$ and lift it (by Hensel's Lemma) to a point $P \in \mathcal{X}(\O_{k'})$. By Lemma~\ref{lem:rsw-basechange} we have $\res_{k'/k} \A \in \fil_n \Br X_{k'} $ and $\rsw_{n+1}(\res_{k'/k} \A) = \rsw_{n+1}(\A)$. Since $\alpha_{P_0}\neq 0$, Theorem~\ref{thm:onelemma_intro}\eqref{surj2-intro}  shows that there exists $Q \in B(P,1)$ such that $\evmap{\A}$ takes $p$ distinct values on $B(Q,n)$. It follows that $\A\notin \file_{n-1} \Br X$.

Now suppose that $\alpha=0$. Then $\A\in \fil_{n}\Br X $. Let $\rsw_{n}(\A)=[\alpha',\beta']_{\pi,n}$ for some $(\alpha',\beta')\in\Omega^2(Y)\oplus \Omega^1(Y) $. Note that $\beta'\neq 0$ since $\A\notin \film_{n-1} \Br X$. Then by the same argument as above, there exists a finite extension $\F'/\F$ and a point $P_0 \in Y(\F')$ satisfying $\beta'_{P_0} \neq 0$. Let $k'/k$ be the unramified extension with residue field $\F'$ and lift $P_0$ to a point $P \in \X(\O_{k'})$. Now Theorem~\ref{thm:onelemma_intro}\eqref{onelemma_intro_1} shows that $\evmap{\A}$ takes $p$ distinct values on $B(P,n)$, whereby $\A\notin \file_{n-1} \Br X$. 
\end{proof}

\begin{lemma}\label{lem:main3}
For $n \ge 0$, we have 
\[
\begin{tikzcd}
\file_n \Br X \arrow[d, phantom, sloped, "\subset"]\\
\{\A\in \Br X\mid \forall\; k'/k \textrm{ finite unramified, } \forall\; P\in \X(\O_{k'}),  \evmap{\A} \textrm{ is constant on } B(P,n+1)\}\arrow[d, phantom, sloped, "\subset"]\\
    \film_n \Br X.
\end{tikzcd}    
  \]
\end{lemma}
\begin{proof}
The first containment is obvious. Now take $\A \in\Br X$ such that  for all finite unramified extensions $k'/k$ and all $P\in \X(\O_{k'})$,  $\evmap{\A}$ is constant on $B(P,n+1)$. Let $r$ be the smallest non-negative integer such that $\A \in \film_r \Br X$. By Lemma~\ref{lem:main2}, $r\leq n$, whence the result.
\end{proof}

\begin{corollary}\label{cor:Evnr}
For $n \ge 0$, we have 
\[
\begin{tikzcd}
\file_n \Br X \arrow[d, phantom, sloped, "="]\\
\{\A\in \Br X\mid \forall\; k'/k \textrm{ finite unramified, } \forall\; P\in \X(\O_{k'}),  \evmap{\A} \textrm{ is constant on } B(P,n+1)\}\arrow[d, phantom, sloped, "="]\\
    \film_n \Br X.
\end{tikzcd}    
  \]
\end{corollary}

\begin{proof}
Immediate from Lemmas~\ref{lemma:main1} and~\ref{lem:main3}.
\end{proof}

\begin{lemma}\label{lem:-1p}
For every $r \ge 1$, there are inclusions 
\begin{align}\label{eq:-2p}
\file_{-2} \Br X &\subset  \{\A\in \Br X\mid \forall\; k'/k \textrm{ finite unramified, } \evmap{\A} \textrm{ is zero on }\X(\O_{k'})\}\\
&\subset \film_{-2} \Br X \nonumber
 \end{align}
 and 
 \begin{align}\label{eq:-1p}
 \file_{-1} \Br X &\subset  \{\A\in \Br X\mid \forall\; k'/k \textrm{ finite unramified, } \evmap{\A} \textrm{ is constant on }\X(\O_{k'})\}\\
 &\subset \film_{-1} \Br X.\nonumber
 \end{align}
\end{lemma}
\begin{proof}
The first inclusions in~\eqref{eq:-2p} and~\eqref{eq:-1p} are clear.
By Lemma~\ref{lem:main3}, 
\[\{\A\in \Br X\mid \forall\; k'/k \textrm{ finite unramified, } \evmap{\A} \textrm{ is constant on }\X(\O_{k'})\}\subset \film_0 \Br X \]
and so the statement only concerns elements of $\film_0 \Br X= \fil_0 \Br X$.
Suppose that $\A \in \fil_0 \Br X$ satisfies $\partial(\A) \neq 0$.
We will prove the existence of a finite unramified extension $k'/k$ such that $ \evmap{\A}$ is non-zero on $\X(\O_{k'})$ and, furthermore, if $\partial(\A)$ does not lie in $\H^1(\F,\Q/\Z)$ then we will show that $ \evmap{\A}$ is non-constant on $\X(\O_{k'})$.
The argument we use is the same as that used in~\cite[\S 5]{bad} for elements of order prime to $p$. 

Write $\bar{Y}$ for the base change of $Y$ to an algebraic closure of $\F$.
Since $Y$ is geometrically connected, the Hochschild--Serre spectral sequence gives a short exact sequence
\[
0 \to \H^1(\F, \Q/\Z) \to \H^1(Y, \Q/\Z) \to \H^1(\bar{Y}, \Q/\Z).
\]
If $\partial(\A)$ lies in $\H^1(\F, \Q/\Z)$ and $\X(\O_k)$ is non-empty, then Proposition~\ref{prop:tame} shows that the corresponding evaluation map $\X(\O_k) \to \Br k$ is constant and non-zero, as desired.
If $\partial(\A)$ lies in $\H^1(\F, \Q/\Z)$ and $\X(\O_k)$ is empty, we can use Lang--Weil to pass to an unramified extension $k'/k$ of degree prime to the order of $\A$ where $\X(\O_{k'})$ is non-empty, and we obtain the same result.

On the other hand, suppose that $\partial(\A)$ does not lie in $\H^1(\F,\Q/\Z)$.
To prove that $\evmap{\A}$ is non-constant, we may change $\A$ by a constant algebra.
Write $\bar{\partial}(\A)$ for the image of $\partial(\A)$ in $\H^1(\bar{Y}, \Q/\Z)$.
Let $m$ be the order of $\bar{\partial}(\A)$; then $m \partial(\A)$ lies in $\H^1(\F, \Q/\Z)$, which is isomorphic to $\Q/\Z$.
Therefore there exists $\alpha \in \H^1(\F,\Q/\Z)$ satisfying $m \alpha = m \partial(\A)$.
The map $\partial \colon \Br k \to \H^1(\F,\Q/\Z)$ is an isomorphism; let $\A' \in \Br k$ satisfy $\partial(\A') = \alpha$.
Replacing $\A$ by $\A-\A'$, we reduce to the case where $\partial(\A)$ and $\bar{\partial}(\A)$ have the same order $m$.

The class $\partial(\A)$ lies in the subgroup $\H^1(Y,\Z/m) \subset \H^1(Y,\Q/\Z)$.
Let $T \to Y$ be a $\Z/m$-torsor representing this class; since its image in $\H^1(\bar{Y},\Z/m)$ also has order $m$, \cite[Lemma~5.15]{bad} shows that the variety $T$ is geometrically connected.  As it is smooth, it is also geometrically irreducible.
The image of $T(\F) \to Y(\F)$ consists of those points $P_0 \in Y(\F)$ such that $\partial(\A)$ maps to $0$ under the induced map $P_0^* \colon \H^1(Y,\Z/m) \to \H^1(\F,\Z/m)$.
Similarly, for any $a \in \H^1(\F,\Z/m)$, let $T_a \to Y$ be a torsor representing the class $\partial(\A) - a$; then the image of $T_a(\F) \to Y(\F)$ consists of those $P_0$ satisfying $P_0^*(\partial(\A)) = a$.
For any fixed $a$, it follows from Lang--Weil that $T_a$ has points over any sufficiently large extension of $\F$.
Therefore, for some extension $\F'/\F$, there exist $P_0,Q_0 \in Y(\F')$ satisfying $P_0^*(\partial(\A)) \neq Q_0^*(\partial(\A))$ in $\H^1(\F',\Z/m)$.
Let $k'/k$ be the unramified extension with residue field $\F'$, and let $P,Q$ be lifts of $P_0,Q_0$ to $\X(\O_{k'})$.
By Proposition~\ref{prop:tame}, we have $\A(P) \neq \A(Q)$ in $\Br k'$, and our proof is complete.
\end{proof}

\begin{remark}
Lemma~\ref{lem:-1p} was already known for elements of order prime to $p$: for instance, the statement $\file_{-1}\Br X \subset \film_{-1} \Br X$ for the prime-to-$p$ parts is the main result of~\cite{CTS}, and we believe that the proof there also gives the slightly stronger result of Lemma~\ref{lem:-1p} for the prime-to-$p$ parts.
\end{remark}

\begin{corollary}\label{cor:Evnr2}
We have
\begin{align*}\file_{-2} \Br X&= \{\A\in \Br X\mid \forall\; k'/k \textrm{ finite unramified, } \evmap{\A} \textrm{ is zero on }\X(\O_{k'})\}\\
&= \film_{-2} \Br X
\end{align*}
and
\begin{align*}
\file_{-1} \Br X&= \{\A\in \Br X\mid \forall\; k'/k \textrm{ finite unramified, } \evmap{\A} \textrm{ is constant on }\X(\O_{k'})\}\\
&= \film_{-1} \Br X.
\end{align*}
\end{corollary}

\begin{proof}
Immediate from Lemmas~\ref{lemma:main1} and~\ref{lem:-1p}.
\end{proof}

This completes the proof of Theorem~\ref{thm:main}. The following corollary of Theorem~\ref{thm:main} describes how the evaluation filtration behaves under base field extension.

\begin{corollary}\label{cor:scale}
Let $\A\in \Br X$ and let $k'/k$ be a finite extension with ramification index $e$. 
\begin{enumerate}
\item For $-2\leq n\leq 0$ we have
\[\A\in \file_n\Br X\implies \res_{k'/k}(\A)\in  \file_n\Br X_{k'}.\]
\item\label{eq:scale} Now let $n\geq 1$, suppose $\A\in \file_n\Br X\setminus\file_{n-1}\Br X$ and write $\rsw_{n+1}(\A)=[\alpha,0]_{\pi, n+1}$. If $\alpha=0$, write $\rsw_{n}(\A)=[\alpha',\beta']_{\pi, n}$. We have
\[ \res_{k'/k}(\A)\in\begin{cases}
\file_{e(n+1)-1}\Br X_{k'}\setminus\file_{e(n+1)-2}\Br X_{k'} & \textrm{ if } \alpha\neq 0;\\
\file_{en}\Br X_{k'}\setminus\file_{en-1}\Br X_{k'} & \textrm{ if } \alpha= 0 \textrm{ and } p\nmid e;\\
\file_{en-1}\Br X_{k'}\setminus\file_{en-2}\Br X_{k'} & \textrm{ if } \alpha= 0 , \alpha'\neq 0\textrm{ and } p\mid e;\\
\file_{en-1}\Br X_{k'}& \textrm{ if } \alpha= \alpha'= 0\textrm{ and } p\mid e.
\end{cases} \]
\end{enumerate}
\end{corollary}

\begin{proof}
\begin{enumerate}
\item The statements for $n=-1,-2$ are clear from the definitions. For $n=0$, use Theorem~\ref{thm:main}(\ref{it:tame0}) and Proposition~\ref{prop:kato61}(1).
\item If $\alpha\neq 0$ then $\A\in\fil_{n+1}\Br X\setminus \fil_n\Br X$. Lemma~\ref{lem:rsw-basechange} shows $\res_{k'/k}(\A)\in\fil_{e(n+1)}\Br X_{k'}$ and
\[\rsw_{e(n+1)}(\res_{k'/k}(\A)) = [\bar{a}^{-(n+1)}\alpha,0]_{\pi', e(n+1)},\] where $\pi'$ is a uniformiser of $k'$ and $a=\pi(\pi')^{-e}$. This shows $\res_{k'/k}(\A)\in\file_{e(n+1)-1}\Br X_{k'}$. Since $\bar{a}\in \F^\times$, $\bar{a}^{-(n+1)}\alpha$ is non-zero and $\res_{k'/k}(\A)\notin\fil_{e(n+1)-1}\Br X_{k'} $. Since $\file_{e(n+1)-2}\Br X_{k'}\subset \fil_{e(n+1)-1}\Br X_{k'} $, it follows that $\res_{k'/k}(\A)\notin\file_{e(n+1)-2}\Br X_{k'}$.

Henceforth, suppose that $\alpha=0$ and therefore $\A\in \fil_n \Br X$. 
Lemma~\ref{lem:rsw-basechange} shows that $\res_{k'/k}(\A)\in\fil_{en}\Br X_{k'}$ and
\begin{equation}\label{eq:rswres}
\rsw_{en}(\res_{k'/k}(\A)) = [\bar{a}^{-n}\alpha',\bar{a}^{-n}e\beta']_{\pi', en}.
\end{equation}
Since $\fil_{n-1}\Br X\subset \file_{n-1}\Br X$ and $\A\notin  \file_{n-1}\Br X$, we have $\rsw_{n}(\A)\neq 0$, whereby at least one of $\alpha',\beta'$ is non-zero. If $p\nmid e$, it follows that $\rsw_{en}(\res_{k'/k}(\A))\neq 0$ and hence $\res_{k'/k}(\A)\in \file_{en}\Br X_{k'}\setminus\file_{en-1}\Br X_{k'}$.

If $p\mid e$ then~\eqref{eq:rswres} becomes $\rsw_{en}(\res_{k'/k}(\A)) = [\bar{a}^{-n}\alpha',0]_{\pi', en}$, whereby $\res_{k'/k}(\A)\in \file_{en-1}\Br X_{k'}$. If $\alpha'\neq 0$ then $\bar{a}^{-n}\alpha'\neq 0$ and $\res_{k'/k}(\A)\notin \fil_{en-1}\Br X_{k'}$, implying that $\res_{k'/k}(\A)\notin \file_{en-2}\Br X_{k'}$. 
\end{enumerate}
\end{proof}

\begin{remark}
Recall from Lemma~\ref{lem:dbna} that if $p\nmid n+1$ then $\alpha =0$ in Corollary~\ref{cor:scale}\eqref{eq:scale}.
\end{remark}

\section{Comparison with other filtrations}\label{sec:comparisons}

Throughout this section, let $K$ denote a Henselian discrete valuation field of characteristic zero with residue field $F$.

There are several other constructions in the literature which give rise to filtrations on $\Br K$, and the question naturally arises as to whether our filtration $\{ \film_n \Br K \}$, as defined at the beginning of Section~\ref{sec:main}, coincides with any of these.
In this section we look at the relationships between several existing filtrations and ours.
We consider two sources of filtrations: existing filtrations on $\H^1(K)$, which give rise to filtrations on $\Br K$ via the cup product; and ramification filtrations on the absolute Galois group of $K$, which give rise to filtrations on $\Br K$ by considering those elements in the kernel of restriction to the subgroups in the filtration.

In what follows, we only consider filtrations on $\Br K[p]$. We often exclude the less interesting case in which the filtrations $\{ \fil_n \Br K[p] \}$ and $\{ \film_n \Br K[p] \}$ coincide; this happens if $e'<p$ or if $\Omega^2_F=0$, for example.

\subsection{Filtrations on $\H^1$}\label{sec:filH1}

The most obvious filtration to consider on $\H^1(K) = \H^1(K,\Q/\Z)$ is Kato's filtration. In the case of equal characteristic, Kato shows~\cite[Theorem~3.2(2)]{K} that his filtrations on $\H^q(K)$ for all $q \ge 1$ are induced by the cup product from that on $\H^1(K)$.
When $K$ has characteristic zero, as in our case, this is at least true for the $p$-torsion, assuming that $K$ contains a primitive $p$th root of unity~\cite[Proposition~4.1(6)]{K}.

There is also a modified or ``non-logarithmic'' version of Kato's filtration on $\H^1(K)$, introduced by Matsuda~\cite{Matsuda} in the case of equal characteristic; as shown in~\cite[Proposition~3.2.7]{Matsuda}, it can be obtained by modifying Kato's filtration on $\H^1(K)$ in exactly the same way that we modify the filtration on $\H^2(K)$.

The Proposition~\ref{prop:different} below shows our modified version of Kato's filtration on $\H^2_p(K) = \Br K[p]$ is not induced in general by any filtration on $\H^1(K)$, even if we omit $\film_{-1}$ and $\film_{-2}$. We begin with a lemma.

\begin{lemma}\label{lem:cancup}
Suppose that $K$ contains a primitive $p$th root of unity, and that the residue field $F$ of $K$ is not perfect.  Let $\chi \in \H^1_p(K)$ satisfy $\sw(\chi)=n$.  Then there exists $y \in \O_K^\times$ such that $\sw(\{\chi, y\}) = n$ and, if $n>0$, we can choose $y$ so that $\{\chi, y\}\notin \film_{n-1} \Br K[p]$.
\end{lemma}
\begin{proof} 
We use Bloch--Kato's explicit description of the graded pieces of the filtration, as described in~\cite[Theorem~4.1(6)]{K}.
Fixing a primitive $p$th root of unity in $K$ gives an isomorphism $\H^1_p(K) \cong K^\times/(K^\times)^p$, under which Kato's filtration on $\H^1_p(K)$ corresponds to the reverse of the natural filtration on $K^\times$.
There are now several cases to consider.
\begin{itemize}
\item If $n=0$, then $\chi \in \fil_0 \H^1_p(K)$ and it follows that $\sw(\{\chi, y\})=0$ for all $y \in K^\times$.
\item If $0<n<e'$, then $\chi$ corresponds to an element $(1 + x \pi^{e'-n}) \in K^\times/(K^\times)^p$ with $x \in \O_K^\times$.  Let $\bar{x} \in F^\times$ be the reduction of $x$.
First suppose that $p \nmid n$. Let $y \in \O_K^\times$ be an element satisfying $d\bar{y} \neq 0$; such an element exists since $F$ is not perfect.  Then $\bar{x} \frac{d\bar{y}}{\bar{y}} \in \Omega^1_F$ is non-zero, and by the first isomorphism of~\cite[(4.2.2)]{K} the element $\{1 + x \pi^{e'-n}, y \}$ has Swan conductor $n$. Moreover,~\cite[Proposition~4.1 and Lemma~4.3]{K} show that $\rsw_{n}(\{1 + x \pi^{e'-n}, y \})=[\bar{c}d(\bar{x}\frac{d\bar{y}}{\bar{y}}),n\bar{c}\bar{x}\frac{d\bar{y}}{\bar{y}}]_{\pi, n}$, where $c=\pi^{e'}/(\zeta-1)^p$. Since $n\bar{c}\bar{x}\frac{d\bar{y}}{\bar{y}}\neq 0$, it follows that $\{1 + x \pi^{e'-n}, y \}\notin\film_{n-1} \Br K[p]$.
Now suppose that $p \mid n$. Then the second isomorphism of~\cite[(4.2.2)]{K} shows firstly (using $q=1$) that $d\bar{x} \neq 0$, and then (using $q=2$) that $\{ 1 + x \pi^{e'-n}, \pi \}$ has Swan conductor $n$. Furthermore,~\cite[Proposition~4.1 and Lemma~4.3]{K} show that $\rsw_{n}(\{1 + x \pi^{e'-n}, \pi \})=[0,\bar{c}d\bar{x}]_{\pi, n}$ so $\{1 + x \pi^{e'-n}, \pi \}\notin \film_{n-1} \Br K[p]$.
\item If $n=e'$, then isomorphism~\cite[(4.2.1)]{K} with $q=1$ shows that $\chi$ corresponds to $x \pi^i \in K^\times/(K^\times)^p$, with $x \in \O_K^\times$ and either $p \nmid i$ or $d\bar{x} \neq 0$.  First suppose that $d\bar{x} \neq 0$. Then~\cite[(4.2.1)]{K} with $q=2$ shows $\sw(\{ x \pi^i, \pi \})=\sw(\{x,\pi\})=e'$ (the term $\{\pi^i,\pi\}$ lies in $\fil_0 \H^2(K)$ and so does not contribute). More precisely,~\cite[Proposition~4.1 and Lemma~4.3]{K} show that $\rsw_{e'}(\{ x \pi^i, \pi \})=\rsw_{e'}(\{x,\pi\})=[0,\bar{c}\frac{d\bar{x}}{\bar{x}}]_{\pi,e'}$. Now suppose that $d\bar{x}=0$ and $p\nmid i$. Let $y \in \O_K^\times$ be an element satisfying $d\bar{y} \neq 0$; we claim that $\{ x \pi^i, y \}$ has the desired properties. Write $\{ x \pi^i, y \}=\{x,y\}-i\{y,\pi\}$. Then~\cite[(4.2.1)]{K} with $q=2$ shows that $\sw(\{x,y\})\leq e'-1$ and $\sw(\{y,\pi\})=e'$. Furthermore,~\cite[Proposition~4.1 and Lemma~4.3]{K} show that $\rsw_{e'}(\{ x \pi^i, y \})=-i\rsw_{e'}(\{y,\pi\})=[0,-i\bar{c}\frac{d\bar{y}}{\bar{y}}]_{\pi,e'}$.
\end{itemize}
\end{proof}

\begin{proposition}\label{prop:different}
Suppose that $K$ contains a primitive $p$th root of unity and that the filtrations $\{ \film_n \Br K[p] \}$ and $\{ \fil_n \Br K[p] \}$ do not coincide for $n \ge 0$.  
Then there is \textbf{no} increasing filtration $\{ \Fil_n \H^1(K,\Z/p) \}$ on $\H^1(K,\Z/p)$ such that, for all $n\ge 0$, $\film_n \Br K[p]$ is generated by $\{ \Fil_n \H^1(K,\Z/p) , K^\times \}$.
\end{proposition}

\begin{proof}
Suppose for contradiction that such a filtration $\{ \Fil_n \H^1(K,\Z/p) \}$ exists. 

First, we claim that $\Fil_n \H^1_p(K,\Z/p) \subset \fil_n \H^1_p(K)$ for all $n\geq 0$. Here, $\fil_n \H^1_p(K)$ denotes Kato's filtration. To prove the claim, let $\alpha \in \Fil_n \H^1_p(K,\Z/p)$. Suppose for contradiction that $\alpha\notin \fil_n \H^1_p(K)$. Then $\sw(\alpha)>n$ and, by Lemma~\ref{lem:cancup}, there exists $b\in K^\times/K^{\times p}$ such that $\{\alpha,b\}\notin \film_n \Br K[p]$, which gives the desired contradiction. 

Now we complete the proof of the proposition. Let $\A\in\film_{n}\Br K[p]$ for $n\geq 0$. Then $\A$ is in the subgroup generated by the image of the map $\Fil_{n} \H^1_p(K,\Z/p) \times K^\times/K^{\times p}\to \Br K[p]$. Since $\Fil_{n} \H^1_p(K ,\Z/p)\subset \fil_{n} \H^1_p(K)$, we deduce that $\A\in \fil_{n}\Br K[p]$. This implies that $\film_{n}\Br K[p]=\fil_{n}\Br K[p]$. 
\end{proof}

\subsection{Ramification filtrations}

Let $\bar{K}$ be a separable closure of $K$, and let $G = \Gal(\bar{K}/K)$ be the absolute Galois group.
Given a descending filtration $(G^i)_{i \ge 0}$ on $G$, we can obtain an ascending filtration on $\H^q_n(K)$ by taking the kernels of the restriction maps $\H^q_n(K) = \H^q(G,(\Z/n)(q-1)) \to \H^q(G^i, (\Z/n)(q-1))$.

In the case of perfect residue field, the ramification groups with the upper numbering give a well-studied filtration on $G$: see~\cite[Ch.~IV]{Serre}.
In the general setting, Abbes and Saito~\cite{AS} made two definitions of ramification groups, $(G^a)_{a \in \Q_{\ge 0}}$ and $(G^a_{\log})_{a \in \Q_{\ge 0}}$, called ``non-logarithmic'' and ``logarithmic''.
In the case of perfect residue field, these coincide (up to a shift in numbering) but in general they are different.

Each of these ramification filtrations gives a filtration on $\H^1(K) = 
\Hom(G,\Q/\Z)$, and one might naturally ask whether those filtrations are related to those described in Section~\ref{sec:filH1}.
This is indeed the case: Kato and Saito~\cite{KS} have proved that Kato's filtration on $\H^1(K)$ coincides with that induced by the logarithmic ramification filtration; and Saito~\cite{Saito} has proved in the case of positive characteristic that Matsuda's non-logarithmic variant of Kato's filtration on  $\H^1(K)$ coincides with that induced by the non-logarithmic ramification filtration.
We will show that our modified Kato filtration on $\Br K[p] = \H^2_p(K)$ is not induced by either of the Abbes--Saito filtrations (where the numbering of the non-logarithmic filtration is shifted by $1$).

Given $\chi \in \H^q_p(K)$, define
\begin{align*}
f_K(\chi) &= \inf \{ a \in \Q_{> 0} \mid \chi \in \ker(\H^q_p(K) \to \H^q(G^a, (\Z/p)(q-1)) \}, \\
f_K^{\log}(\chi) &= \inf \{ a \in \Q_{> 0} \mid \chi \in \ker(\H^q_p(K) \to \H^q(G^a_{\log}, (\Z/p)(q-1)) \}.
\end{align*}
For $q=1$, this is what Abbes and Saito call the (logarithmic) \emph{conductor} of the field extension corresponding to $\chi$: see~\cite[Proposition~6.4 and Proposition~9.5]{AS}.
We have
\[
\film_0 \H^q_p(K) = \{ \chi \in \H^q_p(K) \mid f_K(\chi) \le 1 \}
= \{ \chi \in \H^q_p(K) \mid f_K^{\log}(\chi) \le 0 \}.
\]

We first prove a positive result for the case $q=1$.
In the case of positive characteristic, this follows from~\cite[Corollary~3.3]{Saito}.

\begin{proposition}\label{prop:as-calc}
Suppose that $K$ contains a primitive $p$th root of unity.  Let $\chi \in \H^1_p(K)$.  Then, for all $n \ge 0$,
\[
f_K(\chi) \le n+1 \Longleftrightarrow \Big(
\chi \in \fil_{n+1} \H^1_p(K) \text{ and } \rsw_{n+1}(\chi) = [\alpha,0]_{\pi,n+1} \text{ with } \alpha \in \Omega^1_F \Big) . 
\]
\end{proposition}
\begin{proof}
Since $K$ contains a primitive $p$th root of unity, Kato's filtration on $\H^1_p(K)$ coincides with that of Bloch--Kato (see~\cite[Proposition~4.1(6)]{K}).
This gives explicit generators for the graded pieces of the right-hand filtration, so it is just a case of calculating the conductors of the corresponding cyclic extensions.  This is accomplished in the following series of lemmas by finding the minimal polynomial of a generator for the ring of integers in each extension and applying~\cite[Lemma~6.6]{AS}.
\end{proof}

The calculations in the following lemmas are standard and probably well known.
\begin{lemma}
Suppose that $K$ contains a primitive $p$th root of unity. 
Let $\chi \in \H^1(K,\Z/p)$ correspond to the extension $K(\sqrt[p]{\pi})/K$.  Then $f_K(\chi)=e'+1$. 
\end{lemma}

\begin{proof}
Let $L=K(\sqrt[p]{\pi})$. Then $\O_L=\O_K[\sqrt[p]{\pi}]$. Now apply~\cite[Lemma~6.6]{AS}.
\end{proof}

\begin{lemma}\label{lem:pthrootofx}
Suppose that $K$ contains a primitive $p$th root of unity. 
Let $x \in \O_K^\times$ be such that $\bar{x} \in F$ is not a $p$th power.
Let $\chi \in \H^1(K,\Z/p)$ correspond to the extension $K(\sqrt[p]{x})/K$.  Then $f_K(\chi)=e'$. 
\end{lemma}

\begin{proof}
Let $L=K(\sqrt[p]{x})$. Then $L/K$ is ferociously ramified (i.e.\ the residue field of $L$ is a purely inseparable degree $p$ extension of $F$) and $\O_L=\O_K[\sqrt[p]{x}]$. Now apply~\cite[Lemma~6.6]{AS}.
\end{proof}

\begin{lemma}
Suppose that $K$ contains a primitive $p$th root of unity. 
Let $\chi \in \H^1(K,\Z/p)$ correspond to the extension $K(\sqrt[p]{1+x\pi^m})/K$, where $x \in \O_K^\times$, $0<m<e'$ and $p \nmid m$.  Then $f_K(\chi)=e'+1-m$. 
\end{lemma}

\begin{proof}
Let $L=K(\sqrt[p]{1+x\pi^m})$ and let $\varpi=\sqrt[p]{1+x\pi^{m}}-1$. Write $1=rm+sp$ for $r,s\in\mathbb{Z}$. Considering the terms of smallest valuation in the minimal polynomial of $\varpi$ shows that $\varpi^r\pi^s$ is a uniformiser for $L$ and hence $\O_L=\O_K[\varpi^r\pi^s]$. Now apply~\cite[Lemma~6.6]{AS}.
\end{proof}

\begin{lemma}
Suppose that $K$ contains a primitive $p$th root of unity. 
Let $x \in \O_K^\times$ be such that $\bar{x} \in F$ is not a $p$th power.
Let $\chi \in \H^1(K,\Z/p)$ correspond to the extension $K(\sqrt[p]{1+x\pi^{np}})/K$, where $0<np<e'$.  Then $f_K(\chi)=e'-np$. 
\end{lemma}

\begin{proof}
Let $L=K(\sqrt[p]{1+x\pi^{np}})$ and let $u=(\sqrt[p]{1+x\pi^{np}}-1)/\pi^n$. Then $L/K$ is ferociously ramified and $\O_L=\O_K[u]$.
Now apply~\cite[Lemma~6.6]{AS}.
\end{proof}

We now move to $q=2$ and show that the filtration $\{ \film_n \Br K[p] \}$ is not in general induced by either of the Abbes--Saito ramification filtrations, beginning with the non-logarithmic filtration.

\begin{proposition}
Suppose that $K$ contains a primitive $p$th root of unity and that the residue field $F$ of $K$ is not perfect.  Then it is \textbf{not} true that, for all $n \ge 0$,
\[
\film_n \Br K[p] = \{ \chi \in \Br K[p] \mid f_K(\chi) \le n+1 \}.
\]
\end{proposition}
\begin{proof}
We will show that the equality does not hold for $n=e'$.
Let $x$ be an element of $F \setminus F^p$, let $\tilde{x} \in \O_K$ be a lift of $x$ and let $\psi\in  \H^1(K,\Z/p)$ correspond to the extension $K(\sqrt[p]{\tilde{x}})/K$.
By~\cite[Proposition~4.1 and Lemma~4.3]{K}, the element $\{ \psi,\pi \}$ lies in $\film_{e'} \Br K[p]$ but not in $\film_{e'-1} \Br K[p]$.
On the other hand, by Lemma~\ref{lem:pthrootofx}, we have $f_K(\psi)=e'$, and so $f_K(\{ \psi, \pi \}) \le e'$. 
\end{proof}

Now we treat the logarithmic filtration, by showing that its behaviour under field extension differs from that of our filtration.
For each finite extension $L$ of $K$ contained in $\bar{K}$, let $\{G^a_{L,\log}\}$ be the logarithmic filtration on $G_L = \Gal(\bar{K}/L)$.

\begin{proposition}
Suppose $\Omega^2_F \neq 0$.
It is \textbf{not} true that, for all finite extensions $L/K$, we have
\[
\film_n \H^2_p(L) = \{ \chi \in \H^2_p(L) \mid f_L^{\log}(\chi) \le n \}.
\]
\end{proposition}
\begin{proof} 
Suppose for contradiction that the statement is true.
We may assume that $K$ contains a primitive $p$th root of unity. 
Let $x,y \in F$ be such that $\omega = \frac{dx}{x} \wedge \frac{dy}{y} \neq 0$, and let $\tilde{x},\tilde{y} \in \O_K$ be lifts of $x,y$ respectively.
Define $\A = (\tilde{x},\tilde{y})_p \in \Br K$.
By~\cite[Proposition~4.1 and Lemma~4.3]{K}, we have $\A \in \fil_{e'} \Br K[p]$, and $\rsw_{e'}(\A) = [\bar{c} \omega, 0]_{\pi,e'}$ where $\bar{c} \in \F$ is non-zero.
Therefore $\A$ lies in $\film_{e'-1} \Br K[p]$, and by assumption $f_K^{\log}(\A) \le e'-1$.

Let $L/K$ be any wildly ramified extension of degree $p$.
The inclusion $G_{L,\log}^{pa} \subset G_{\log}^a$ for all $a \ge 0$ (see~\cite{AS})
implies $f_L^{\log}(\res_{L/K}\A) \le p(e'-1)$,
and so the image of $\A$ in $\Br L$ lies in $\film_{p(e'-1)} \Br L[p]$.
However, the same calculation as before shows that $\rsw_{e'_L}(\res_{L/K}\A)=[\bar{c}_L \omega, 0]_{\pi_L, e'_L}$, where $\pi_L$ is a uniformiser of $L$, $e'_L = pe'$ and $\bar{c}_L$ is some non-zero element of $\F$. So the image of $\A$ lies in $\film_{pe'-1} \Br L[p]$ but not $\film_{pe'-2} \Br L[p]$, giving a contradiction.
\end{proof}

\section{Applications to the Brauer--Manin obstruction}\label{sec:BMO}

Let $V$ be a smooth, proper, geometrically irreducible variety over a number field $L$ such that $V(\mathbb{A}_L)\neq\emptyset$.
The surjectivity results described in Theorem~\ref{thm:onelemma_intro} have implications for the existence of Brauer--Manin obstructions to the Hasse principle and weak approximation
on $V$, as follows.
Suppose that $\B$ has order $n$ in $\Br V$, and that
$\p$ is a finite place of $L$ such that the evaluation map $\evmap{\B}: V(L_\p)\to\Br L_\p[n]$ is surjective. Write $V(\mathbb{A}_L)^\B$ for the subset of adelic points of $V$ that are orthogonal to $\B$ under the Brauer--Manin pairing; this contains $V(\mathbb{A}_L)^{\Br}$.
Our freedom to adjust the value taken by $\B$ at the place $\p$ shows that
\[\emptyset\neq V(\mathbb{A}_L)^\B\subsetneq V(\mathbb{A}_L).\]
In other words, $\B$ does not obstruct the Hasse principle on $V$, but it does obstruct weak approximation on $V$. Note that in order to show that $\B$ obstructs weak approximation on $V$, it suffices that $\evmap{\B}: V(L_\p)\to\Br L_\p$ be non-constant. The existence of Brauer group elements with non-constant evaluations at primes of good ordinary reduction is the subject of Theorem~\ref{thm:potential}, which we now prove.

\begin{proof}[Proof of Theorem~\ref{thm:potential}]
Let $V_\p$ be the base change of $V$ to $L_\p$, and
choose a smooth model $\mathcal{V}$ of $V_\p$ over the ring of integers of $L_\p$ such that the special fibre $Y$ is ordinary.
Let $\bar{V}_\p$ denote the base change of $V_\p$ to an algebraic closure of $L_\p$.
The spectral sequences~\cite[0.2]{BK}
\[
E^{s,t}_2 = \H^s(\bar{Y}, \bar{i}^* \R^t \bar{j}_* \Z/p^r) \implies \H^{s+t}(\bar{V}_\p, \Z/p^r)
\]
define decreasing filtrations on $\H^q(\bar{V}_\p, \Z/p^r)$ for all $r$, and also on $\H^q(\bar{V}_\p, \Z_p)$ and $\H^q(\bar{V}_\p, \Q_p)$.
For any of these filtrations, let $\gr^i$ denote the graded pieces.

Since $Y$ is ordinary, $\gr^0 \H^2(\bar{V}_\p, \Q_p) \neq 0$ by~\cite[Theorem~0.7(iii)]{BK}.
Therefore $\gr^0 \H^2(\bar{V}_\p, \Z_p)$ is also non-zero, and so $\gr^0 \H^2(\bar{V}_\p, \Z/p^r)$ is non-zero for some $r \ge 1$.

Let $\bar{L}$ be the algebraic closure of $L$ inside our chosen algebraic closure of $L_\p$, and let $\bar{V}$ be the base change of $V$ to $\bar{L}$.
By proper base change~\cite[Corollary~VI.2.6]{M}, the natural map $\H^2(\bar{V}, \Z/p^r) \to \H^2(\bar{V}_\p, \Z/p^r)$ is an isomorphism.
Let $\alpha \in \H^2(\bar{V}, \Z/p^r)$ have non-zero image in $\gr^0 \H^2(\bar{V}_\p, \Z/p^r)$.
Replacing $L$ by a finite extension, we may assume that $\alpha$ is defined over $L$ and that $L$ contains the $p^r$th roots of unity.
We fix an isomorphism $\Z/p^r \cong (\Z/p^r)(1)$ on $V$, and view $\alpha$ as an element of $\H^2(V, (\Z/p^r)(1))$.

We will show that the image of $\alpha$ in $\Br V_\p$ does not lie in $\fil_0 \Br V_\p$.
Let $\Kh$ be the Henselisation of the function field $K=L_\p(V)$ at the discrete valuation corresponding to $Y$, and let $\Kh_{nr}$ be its maximal unramified extension.
Comparing the spectral sequences of vanishing cycles for $V_\p$ and $\Kh$ gives a commutative diagram
\[
\begin{CD}
\H^2(V_\p, (\Z/p^r)(1)) @>{f}>> \H^0(Y, i^* \R^2 j_* (\Z/p^r)(1))  \\
@VVV @V{g}VV \\
\H^2_{p^r}(\Kh) @>{\res}>> \H^0(\Kh, \H^2_{p^r}(\Kh_{nr}))
\end{CD}
\]
in which $\res$ is the usual restriction map of Galois cohomology, $\gr^0 \H^2(V_\p, (\Z/p^r)(1))$ is the coimage of $f$ (the quotient of $\H^2(V_\p, (\Z/p^r)(1))$ by the kernel of $f$), and $\fil_0 \H^2_{p^r}(\Kh)$ is the kernel of $\res$.
By construction, $f(\alpha)$ is non-zero.  By Lemma~\ref{lem:tauinj}, $g$ is injective, showing that $g(f(\alpha))$ is non-zero.  So the image of $\alpha$ in $\H^2(\Kh)$ does not lie in $\fil_0 \H^2(\Kh)$. 

Let $\A$ be the image of $\alpha$ in $\Br V$.  By Theorem~\ref{thm:main}, after possibly replacing $L$ by a further finite extension, the evaluation map $\evmap{\A} \colon V(L_\p) \to \Br L_\p$ is non-constant, showing that $\A$ obstructs weak approximation on $V$.
\end{proof}

Our final task is to prove Theorem~\ref{thm:app}.
We begin by gathering some criteria which can be used to show that various graded pieces of the filtration on $\Br X$ vanish. Lemma~\ref{lem:O1O2} is not actually used in the proof of Theorem~\ref{thm:app} but is included as a first example of how one can deduce information about $\Br X$ from properties of the special fibre.

\begin{lemma}\label{lem:O1O2}
Suppose that $\H^0(Y, \Omega^1_Y) = \H^0(Y, \Omega^2_Y) = 0$.  Then $\fil_0 \Br X = \Br X$.
\end{lemma}
\begin{proof}
If $\A$ is an element of $\fil_n \Br X$ for $n\geq 1$, then~Theorem~\ref{thm:onelemma}(\ref{onelemma_reg}) shows that $\rsw_n(\A)=[\alpha,\beta]_{\pi,n}$ with $(\alpha,\beta) \in \H^0(Y,\Omega^2_Y) \oplus \H^0(Y,\Omega^1_Y)=0$.
This shows $\fil_n \Br X = \fil_{n-1} \Br X$ for all $n \ge 1$, and so $\fil_0 \Br X = \Br X$.
\end{proof}

\begin{lemma}\label{lem:O1e}
Suppose $\H^0(Y, \Omega^1_Y)=0$ and $e<p-1$.
Then $\fil_0 \Br X = \Br X$.
\end{lemma}
\begin{proof}
It suffices to show $\rsw_n(\A)=0$ for all $\A \in \fil_n \Br X$ with $n \ge 1$.
Suppose $\rsw_n(\A)=[\alpha,\beta]_{\pi,n}$.
If $p \nmid n$, then Lemma~\ref{lem:dbna} shows $n \alpha = d\beta$.
Since $\beta$ lies in $\H^0(Y,\Omega^1_Y)=0$, it follows that $\alpha=\beta=0$, completing the proof in this case.

We have $e' = ep/(p-1) < p$, and so $p \nmid n$ holds for all $n \le e'$.
The remaining case is when $\A \in \Br X$ has $\sw(\A)=n>e'$ with $p \mid n$.  Then Lemma~\ref{lem:rsw-same} shows $\sw(p\A) = n-e$, which is not divisible by $p$; as above, we deduce $\rsw_{n-e}(p\A)=0$ and therefore, by Lemma~\ref{lem:rsw-same} again, $\rsw_n(\A)=0$, contradicting the assumption $\sw(\A)=n$.

Thus we have $\fil_n \Br X = \fil_{n-1} \Br X$ for all $n \ge 1$, and so $\fil_0 \Br X = \Br X$.
\end{proof}

\begin{lemma}\label{lem:H1}
Suppose $\H^1(\bar{Y}, \Z/p)=0$.  Then $\film_{-1} \Br X\{p\} = \fil_0 \Br X\{p\}$.
\end{lemma}
\begin{proof}
Firstly, the group $\H^1(\bar{Y}, \Z/p^r)$ is trivial for all $r$: it is an Abelian $p$-group and its $p$-torsion subgroup $\H^1(\bar{Y}, \Z/p)$ is trivial.
Now, for every $r$, the Hochschild--Serre spectral sequence gives a short exact sequence
\[
0 \to \H^1(\F,\Z/p^r) \to \H^1(Y,\Z/p^r) \to \H^1(\bar{Y},\Z/p^r),
\]
showing that the natural map $\H^1(\F,\Z/p^r) \to \H^1(Y,\Z/p^r)$ is an isomorphism.  The result then follows from Proposition~\ref{prop:tame}.
\end{proof}

\begin{lemma}\label{lem:H1spec}
Let $\X \to \O_k$ be a smooth proper morphism such that the generic fibre $X$ is geometrically integral.  Let $n$ be a positive integer and suppose $\H^1(\bar{X},\Z/n)=0$.  Then the special fibre $Y$ satisfies $\H^1(\bar{Y},\Z/n)=0$.
\end{lemma}
\begin{proof}
Let $k'$ be a finite extension of $k$, with ring of integers $\O_{k'}$ and residue field $\F'$.
Let $\X' = \X \times_{\O_k} \O_{k'}$ be the base change and denote its special and generic fibres by $Y'$ and $X'$ respectively.
$\X'$ is proper over $\O_{k'}$, so the proper base change theorem gives an isomorphism $\H^1(Y',\Z/n) \cong \H^1(\X',\Z/n)$.
On the other hand, by \cite[Corollary~I.10.2]{SGA1}, an \'{e}tale cover of a connected normal scheme is uniquely determined by its fibre at the generic point,
so the natural map $\H^1(\X', \Z/n) \to \H^1(X',\Z/n)$ is injective.
We deduce that $\H^1(Y',\Z/n)$ injects into $\H^1(X',\Z/n)$.
Taking the limit over all finite extensions $k'/k$ shows that $\H^1(\bar{Y},\Z/n)$ injects into $\H^1(\bar{X},\Z/n)=0$.
\end{proof}

\begin{proof}[Proof of Theorem~\ref{thm:app}]
Since $V$ is smooth and proper over $L$, there exists a smooth proper model $\V \to \Spec\O_S$ for some finite set $S$ of places of $L$ containing all the infinite places. The assumption that $\Pic\bar{V}$ be torsion-free implies $\H^1(V,\O_V)=0$ and hence, by Hodge theory, $\H^0(V,\Omega^1_V)=0$.
For a finite place $\p \notin S$, denote by $\V(\p)$ the fibre $\V \times_{\O_S} k(\p)$.
Semi-continuity shows that, after possibly enlarging $S$, we have $\H^0(\V(\p), \Omega^1_{\V(\p)})=0$ for all $\p \notin S$.

Let $n$ be any positive integer.
Since $\Pic\bar{V}$ is torsion-free, the Kummer sequence gives $\H^1(\bar{V},\Z/n) \cong \H^1(\bar{V},\mmu_n)=0$. Suppose that $\p$ is a place of $L$ not contained in $S$.
By~\cite[VI.2.6]{M}, we have $\H^1(\bar{V} \times_{\bar{L}} \bar{L}_\p, \Z/n)=0$, and Lemma~\ref{lem:H1spec} applied to $\V \times_{\O_S} \O_{L_\p}$ shows $\H^1(\overline{\V(\p)}, \Z/n)=0$.

We enlarge $S$ to include all finite places $\p$ whose absolute ramification index $e_\p$ satisfies $e_\p \ge p-1$, where $p$ is the residue characteristic of $\p$.  (It is enough to include all primes ramified in $L$ and all primes above $2$.)  
Let $\p$ be a place not in $S$, of residue characteristic $p$.  Lemma~\ref{lem:O1e} and Lemma~\ref{lem:H1} show that, for any $\A \in \Br V\{p\}$, the evaluation map $\evmap{\A} \colon V(L_\p) \to \Br L_\p$ is constant.
\cite[Proposition~2.4]{CTSk} proves the same for $\A \in \Br V$ of order prime to $p$, completing the proof.
\end{proof}

\begin{remark}\label{rmk:fewerplaces}
If, for example, $V$ is a K3 surface, then there is no need to enlarge $S$ to ensure that $\H^0(\V(\p), \Omega^1_{\V(\p)})=0$ for all $\p \notin S$. In other words, there are no places included in the subset~(\ref{omega-1}) of Theorem~\ref{thm:app}.
Indeed, in this case, for any place $\p$ admitting a smooth proper model $\V \to \Spec \O_\p$, the reduction $\V(\p)$ is also a K3 surface, as follows.  Because $\Pic \V \to \Pic V$ is an isomorphism and $\omega_{V/k}$ is trivial, it follows that $\omega_{\V/\O_k}$ is also trivial, and therefore so is $\omega_{\V(\p)/\F_\p}$.
Serre duality then gives $h^2(\V(\p), \O_{\V(\p)}) = h^0(\V(\p), \omega_{\V(\p)})=h^0(\V(\p), \O_{\V(\p)})$.
Since $\V(\p)$ is geometrically connected, one has $h^0(\V(\p), \O_{\V(\p)})=1$, and the fact that the Euler characteristic is constant in flat families gives $h^1(\V(\p), \O_{\V(\p)})=0$, showing that $\V(\p)$ is a K3 surface.
It follows by~\cite[Theorem~9.5.1]{huybrechts} that $\H^0(\V(\p), \Omega^1_{\V(\p)})$ is trivial.
\end{remark}

\bibliographystyle{abbrv}
\bibliography{ferocious}

\providecommand\eatperiod[1]{\ifthenelse{\equal{#1}{.}}{}{#1}}
  \providecommand{\noopsort}[1]{}
\begin{thebibliography}{10}

\bibitem{SGA1}
{\em Rev\^{e}tements \'{e}tales et groupe fondamental}.
\newblock Lecture Notes in Mathematics, Vol. 224. Springer-Verlag, Berlin-New
  York, 1971.
\newblock S\'{e}minaire de G\'{e}om\'{e}trie Alg\'{e}brique du Bois Marie
  1960--1961 (SGA 1). Dirig\'{e} par Alexandre Grothendieck. Augment\'{e} de
  deux expos\'{e}s de M. Raynaud.

\bibitem{SGA43}
{\em Th\'{e}orie des topos et cohomologie \'{e}tale des sch\'{e}mas. {T}ome 3}.
\newblock Lecture Notes in Mathematics, Vol. 305. Springer-Verlag, Berlin-New
  York, 1973.
\newblock S\'{e}minaire de G\'{e}om\'{e}trie Alg\'{e}brique du Bois-Marie
  1963--1964 (SGA 4). Dirig\'{e} par M. Artin, A. Grothendieck et J. L.
  Verdier. Avec la collaboration de P. Deligne et B. Saint-Donat.

\bibitem{AS}
A.~Abbes and T.~Saito.
\newblock Ramification of local fields with imperfect residue fields.
\newblock {\em Amer. J. Math.}, 124(5):879--920, 2002.

\bibitem{BK}
S.~Bloch and K.~Kato.
\newblock {$p$}-adic \'etale cohomology.
\newblock {\em Inst. Hautes \'Etudes Sci. Publ. Math.}, (63):107--152, 1986.

\bibitem{BZ}
F.~A. Bogomolov and Y.~G. Zarhin.
\newblock Ordinary reduction of {$K3$} surfaces.
\newblock {\em Cent. Eur. J. Math.}, 7(2):206--213, 2009.

\bibitem{Borger}
J.~M. Borger.
\newblock Conductors and the moduli of residual perfection.
\newblock {\em Math. Ann.}, 329(1):1--30, 2004.

\bibitem{BM}
C.~Breuil and W.~Messing.
\newblock Torsion \'{e}tale and crystalline cohomologies.
\newblock {\em Ast\'{e}risque}, (279):81--124, 2002.
\newblock Cohomologies $p$-adiques et applications arithm\'{e}tiques, II.

\bibitem{bad}
M.~Bright.
\newblock Bad reduction of the {B}rauer--{M}anin obstruction.
\newblock {\em J. Lond. Math. Soc. (2)}, 91(3):643--666, 2015.

\bibitem{Kestutis}
K.~{\v{C}}esnavi{\v{c}}ius.
\newblock Purity for the {B}rauer group.
\newblock {\em Duke Math. J.}, 168(8):1461--1486, 2019.

\bibitem{CTS}
J.-L. Colliot-Th{\'e}l{\`e}ne and S.~Saito.
\newblock Z\'ero-cycles sur les vari\'et\'es {$p$}-adiques et groupe de
  {B}rauer.
\newblock {\em Internat. Math. Res. Notices}, (4):151--160, 1996.

\bibitem{CTSS}
J.-L. Colliot-Th{\'e}l{\`e}ne, J.-J. Sansuc, and C.~Soul{\'e}.
\newblock Torsion dans le groupe de {C}how de codimension deux.
\newblock {\em Duke Math. J.}, 50(3):763--801, 1983.

\bibitem{CTSk}
J.-L. Colliot-Th\'{e}l\`ene and A.~N. Skorobogatov.
\newblock Good reduction of the {B}rauer--{M}anin obstruction.
\newblock {\em Trans. Amer. Math. Soc.}, 365(2):579--590, 2013.

\bibitem{CTSbook}
J.-L. Colliot-Th\'{e}l\`ene and A.~N. Skorobogatov.
\newblock {\em The {B}rauer-{G}rothendieck group}, volume~71 of {\em Ergebnisse
  der Mathematik und ihrer Grenzgebiete. 3. Folge. A Series of Modern Surveys
  in Mathematics [Results in Mathematics and Related Areas. 3rd Series. A
  Series of Modern Surveys in Mathematics]}.
\newblock Springer, Cham, [2021] \copyright 2021.

\bibitem{CTSD}
J.-L. Colliot-Th\'{e}l\`ene and {\relax Sir
  Peter\aftergroup\eatperiod}.~Swinnerton-Dyer.
\newblock Hasse principle and weak approximation for pencils of
  {S}everi--{B}rauer and similar varieties.
\newblock {\em J. Reine Angew. Math.}, 453:49--112, 1994.

\bibitem{DI}
P.~Deligne and L.~Illusie.
\newblock Rel\`evements modulo {$p^2$} et d\'{e}composition du complexe de de
  {R}ham.
\newblock {\em Invent. Math.}, 89(2):247--270, 1987.

\bibitem{FM}
J.-M. Fontaine and W.~Messing.
\newblock {$p$}-adic periods and {$p$}-adic \'{e}tale cohomology.
\newblock In {\em Current trends in arithmetical algebraic geometry ({A}rcata,
  {C}alif., 1985)}, volume~67 of {\em Contemp. Math.}, pages 179--207. Amer.
  Math. Soc., Providence, RI, 1987.

\bibitem{Gabber-inj}
O.~Gabber.
\newblock An injectivity property for \'{e}tale cohomology.
\newblock {\em Compositio Math.}, 86(1):1--14, 1993.

\bibitem{Gabber}
O.~Gabber.
\newblock Affine analog of the proper base change theorem.
\newblock {\em Israel J. Math.}, 87(1-3):325--335, 1994.

\bibitem{GS}
P.~Gille and T.~Szamuely.
\newblock {\em Central simple algebras and {G}alois cohomology}, volume 101 of
  {\em Cambridge Studies in Advanced Mathematics}.
\newblock Cambridge University Press, Cambridge, 2006.

\bibitem{GIII}
A.~Grothendieck.
\newblock Le groupe de {B}rauer {III}.
\newblock In {\em Dix Expos\'es sur la Cohomologie des Sch\'emas}, volume~3 of
  {\em Advanced studies in mathematics}, pages 88--188. North-Holland,
  Amsterdam, 1968.

\bibitem{Harari}
D.~Harari.
\newblock Weak approximation and non-abelian fundamental groups.
\newblock {\em Ann. Sci. \'{E}cole Norm. Sup. (4)}, 33(4):467--484, 2000.

\bibitem{Hartshorne}
R.~Hartshorne.
\newblock {\em Algebraic geometry}.
\newblock Graduate Texts in Mathematics, No. 52. Springer-Verlag, New
  York-Heidelberg, 1977.

\bibitem{huybrechts}
D.~Huybrechts.
\newblock {\em Lectures on {K}3 surfaces}, volume 158 of {\em Cambridge Studies
  in Advanced Mathematics}.
\newblock Cambridge University Press, Cambridge, 2016.

\bibitem{I}
L.~Illusie.
\newblock Complexe de de\thinspace {R}ham--{W}itt et cohomologie cristalline.
\newblock {\em Ann. Sci. \'Ecole Norm. Sup. (4)}, 12(4):501--661, 1979.

\bibitem{joshi}
K.~Joshi.
\newblock On primes of ordinary and {H}odge--{W}itt reduction.
\newblock arXiv:1603.09404, 2016.

\bibitem{K3}
K.~Kato.
\newblock A generalization of local class field theory by using {$K$}-groups.
  {II}.
\newblock {\em J. Fac. Sci. Univ. Tokyo Sect. IA Math.}, 27(3):603--683, 1980.

\bibitem{KatoGC}
K.~Kato.
\newblock Galois cohomology of complete discrete valuation fields.
\newblock In {\em Algebraic {$K$}-theory, {P}art {II} ({O}berwolfach, 1980)},
  volume 967 of {\em Lecture Notes in Math.}, pages 215--238. Springer,
  Berlin-New York, 1982.

\bibitem{K4}
K.~Kato.
\newblock A generalization of local class field theory by using {$K$}-groups.
  {III}.
\newblock {\em J. Fac. Sci. Univ. Tokyo Sect. IA Math.}, 29(1):31--43, 1982.

\bibitem{KatoLog}
K.~Kato.
\newblock Logarithmic structures of {F}ontaine-{I}llusie.
\newblock In {\em Algebraic analysis, geometry, and number theory ({B}altimore,
  {MD}, 1988)}, pages 191--224. Johns Hopkins Univ. Press, Baltimore, MD, 1989.

\bibitem{K}
K.~Kato.
\newblock Swan conductors for characters of degree one in the imperfect residue
  field case.
\newblock In {\em Algebraic {$K$}-theory and algebraic number theory
  ({H}onolulu, {HI}, 1987)}, volume~83 of {\em Contemp. Math.}, pages 101--131.
  Amer. Math. Soc., Providence, RI, 1989.

\bibitem{KS}
K.~Kato and T.~Saito.
\newblock Coincidence of two {S}wan conductors of abelian characters.
\newblock {\em \'{E}pijournal Geom. Alg\'{e}brique}, 3:Art. 15, 16, 2019.

\bibitem{LW}
S.~Lang and A.~Weil.
\newblock Number of points of varieties in finite fields.
\newblock {\em Amer. J. Math.}, 76:819--827, 1954.

\bibitem{Manin}
{\relax Yu}.~I. Manin.
\newblock Le groupe de {B}rauer--{G}rothendieck en g\'eom\'etrie diophantienne.
\newblock In {\em Actes du Congr\`es International des Math\'ematiciens (Nice,
  1970), Tome 1}, pages 401--411. Gauthier-Villars, Paris, 1971.

\bibitem{Matsuda}
S.~Matsuda.
\newblock On the {S}wan conductor in positive characteristic.
\newblock {\em Amer. J. Math.}, 119(4):705--739, 1997.

\bibitem{M}
J.~S. Milne.
\newblock {\em Etale Cohomology}.
\newblock Number~33 in Princeton mathematical series. Princeton University
  Press, 1980.

\bibitem{SS}
S.~Saito and K.~Sato.
\newblock Zero-cycles on varieties over {$p$}-adic fields and {B}rauer groups.
\newblock {\em Ann. Sci. \'Ecole Norm. Sup. (4)}, 47(3):505--537, 2014.

\bibitem{Saito}
T.~Saito.
\newblock A characterization of ramification groups of local fields with
  imperfect residue field, 2020.

\bibitem{Serre}
J.-P. Serre.
\newblock {\em Local fields}, volume~67 of {\em Graduate Texts in Mathematics}.
\newblock Springer-Verlag, New York, 1979.
\newblock Translated from the French by Marvin Jay Greenberg.

\bibitem{SBook}
A.~Skorobogatov.
\newblock {\em Torsors and rational points}, volume 144 of {\em Cambridge
  Tracts in Mathematics}.
\newblock Cambridge University Press, Cambridge, 2001.

\bibitem{SZ}
A.~N. Skorobogatov and Y.~G. Zarhin.
\newblock A finiteness theorem for the {B}rauer group of abelian varieties and
  {$K3$} surfaces.
\newblock {\em J. Algebraic Geom.}, 17(3):481--502, 2008.

\bibitem{stacks-project}
{The Stacks project authors}.
\newblock The {S}tacks project.
\newblock \url{https://stacks.math.columbia.edu}, 2020.

\bibitem{U}
T.~Uematsu.
\newblock A note on evaluation maps associated with the {B}rauer--{M}anin
  pairing.
\newblock
  \url{https://ccmath.meijo-u.ac.jp/~uematsu/paper/A_note_on_evaluation_maps_associated_with_the_Brauer-Manin_pairing.pdf},
  2012.
\newblock Unpublished.

\bibitem{Y}
T.~Yamazaki.
\newblock On {S}wan conductors for {B}rauer groups of curves over local fields.
\newblock {\em Proc. Amer. Math. Soc.}, 127(5):1269--1274, 1999.

\end{thebibliography}

\end{document}